\documentclass[11pt,a4paper,twoside]{amsart}
\usepackage[activeacute,english]{babel}
\usepackage[latin1]{inputenc}
\usepackage{amsfonts}
\usepackage{amsmath,amsthm}
\usepackage{amssymb}
\usepackage{graphics}
\usepackage{pdfsync}
 
\newcommand{\BB}{{\mathcal B}}
\newcommand{\CC}{{\mathcal C}}
\newcommand{\DD}{{\mathcal D}}

\newcommand{\FF}{{\mathcal F}}
\newcommand{\GG}{{\mathcal G}}
\newcommand{\HH}{{\mathcal H}}

\newcommand{\LL}{{\mathcal L}}

\newcommand{\OO}{{\mathcal O}}
\newcommand{\PP}{{\mathcal P}}
\newcommand{\RR}{{\mathcal R}}
\newcommand{\SSS}{{\mathcal S}}
\newcommand{\TT}{{\mathcal T}}
\newcommand{\VV}{{\mathcal V}}

\newcommand{\C}{{\mathbb C}}
\newcommand{\N}{{\mathbb N}}
\newcommand{\R}{{\mathbb R}}
\newcommand{\Z}{{\mathbb Z}}

\newcommand{\wit}{\widetilde}
\newcommand{\wih}{\widehat}

\newcommand{\supp}{{\operatorname{supp}}}
\newcommand{\diam}{{\operatorname{diam}}}

\newcommand{\Lip}{{\operatorname{Lip}}}
\newcommand{\dist}{{\operatorname{dist}}}
\newcommand{\Top}{{\operatorname{Top}}}
\newcommand{\Tree}{{\operatorname{Tree}}}
\newcommand{\Child}{{\operatorname{Ch}}}
\newcommand{\Tre}{{\operatorname{Tr}}}
\newcommand{\Trees}{{\operatorname{Trs}}}
\newcommand{\Stop}{{\operatorname{Stp}}}

\newtheorem{teo}{Theorem}[section]

\newtheorem{lema}[teo]{Lemma}
\newtheorem{coro}[teo]{Corollary}
\newtheorem{propo}[teo]{Proposition}

\newtheorem{ques}[teo]{Question}

\newtheorem{defi}[teo]{Definition}
\newtheorem{claim}[teo]{Claim}
\newtheorem{remarko}[teo]{{\em Remark}}
{\theoremstyle{remark} }
\allowdisplaybreaks[1]

\title[Variation for the Riesz transform and uniform rectifiability]
{Variation for the Riesz transform\\and uniform rectifiability}
\author[A. MAS AND X. TOLSA]{ALBERT MAS\; AND\; XAVIER TOLSA}
\date{September, 2011}
\thanks{Both authors are partially supported by grants
2009SGR-000420 (Generalitat de
Catalunya) and MTM2010-16232 (Spain). Albert Mas is also
supported by grant AP2006-02416 (FPU program, Spain).}
\address{Albert Mas. Departament de Matem\`atiques, Universitat
Aut\`onoma de Bar\-ce\-lo\-na. 08193 Barcelona, Catalonia. Spain} \email{amblesa@mat.uab.cat}
\address{Xavier Tolsa. Instituci\'{o} Catalana de Recerca
i Estudis Avan\c{c}ats (ICREA) and Departament de
Ma\-te\-m\`a\-ti\-ques, Universitat Aut\`onoma de Bar\-ce\-lo\-na.
08193 Barcelona, Catalonia} \email{xtolsa@mat.uab.cat}

\begin{document}

\begin{abstract}
For $1\leq n<d$ integers and $\rho>2$, we prove that an $n$-dimensional Ahlfors-David regular measure $\mu$ in $\R^d$ is uniformly $n$-rectifiable if and only if the $\rho$-variation for the Riesz transform with respect to $\mu$ is a bounded operator in $L^2(\mu)$. This result can be considered as a partial solution to a well known open problem posed by G. David and S. Semmes which relates the $L^2(\mu)$ boundedness of the Riesz transforms to the uniform rectifiability of $\mu$.
\end{abstract}
\maketitle

\section{Introduction}

In this paper we characterize the notion of uniform rectifiability in the sense of David and Semmes \cite{DS2}
in terms of the $L^2$ boundedness of the $\rho$-variation for the Riesz transform, with $\rho>2$.

Given $1\leq n<d$ integers and a Borel measure $\mu$ in $\R^d$, one defines the
$n$-dimensional Riesz transform of a function $f\in L^1(\mu)$ by
$R^\mu f(x)=\lim_{\epsilon\searrow0}R^\mu_\epsilon f(x)$
(whenever the limit exists), where
$$R^\mu_\epsilon f(x)=\int_{|x-y|>\epsilon}\frac{x-y}
{|x-y|^{n+1}}\,f(y)\,d\mu(y),\qquad x\in\R^d.$$
We will use the notation $\RR^\mu f(x):=\{R_{\epsilon}^\mu f(x)\}_{\epsilon>0}$. 
When $d=2$ (i.e., $\mu$ is a Borel measure in $\C$), one defines the
Cauchy transform of  $f\in L^1(\mu)$ by $C^\mu
f(x)=\lim_{\epsilon\searrow0}C^\mu_\epsilon f(x)$ (whenever the
limit exists), where
$$C^\mu_\epsilon f(x)=\int_{|x-y|>\epsilon}\frac{f(y)}
{x-y}\,d\mu(y),\qquad x\in\C.$$
To avoid the problem of existence of the preceding limits, it is useful to consider the maximal
operators $R^\mu_* f(x)=\sup_{\epsilon>0}|R^\mu_\epsilon f(x)|$
and $C^\mu_* f(x)=\sup_{\epsilon>0}|C^\mu_\epsilon f(x)|$. 
Notice that the Cauchy transform coincides with the $1$-dimensional Riesz transform in $\R^2$ modulo conjugation, since $1/x=\overline x/|x|^{2}$ for all $x\in\C\setminus\{0\}$.

The Cauchy and Riesz transforms are two very important examples
of singular integral operators with a Calder\'{o}n-Zygmund kernel.
Given $d\geq2$, the kernels $K:\R^d\setminus\{0\}\to\R$ that we consider in this
paper satisfy
\begin{equation}\label{4eq333}
|K(x)|\leq \frac{C}{|x|^{n}},\quad|\partial_{x^i}K(x)|\leq
\frac{C}{|x|^{n+1}}\quad\text{and}\quad|\partial_{x^i}\partial_{x^j}K(x)|\leq
\frac{C}{|x|^{n+2}},
\end{equation}
for all $1\leq i,j\leq d$ and
$x=(x^1,\ldots,x^d)\in\R^d\setminus\{0\}$, where $1\leq n<d$ is some
integer and $C>0$ is some constant; and moreover $K(-x)=-K(x)$ for
all $x\neq0$ (i.e. $K$ is odd). Notice that the $n$-dimensional
Riesz transform corresponds to the vector kernel
$(x^1,\ldots,x^d)/|x|^{n+1}$, and the Cauchy transform
to $(x^1,-x^2)/|x|^{2}$ (so, we may consider $K$
to be any scalar component of these vector kernels).
For $f\in L^1(\mu)$ and $x\in\R^d$, we set
\begin{equation*}
T_{\epsilon}^\mu f(x)\equiv T_{\epsilon}(f\mu)(x):=\int_{|x-y|>\epsilon}K(x-y)f(y)\,d\mu(y),
\end{equation*}
and we denote $\TT^\mu f(x)=\{ T_{\epsilon}^\mu f(x)\}_{\epsilon>0}$.

\begin{defi}[$\rho$-variation and oscillation]
Let
$\mathcal{F}:=\{F_\epsilon\}_{\epsilon>0}$ be a family of
functions defined on $\R^d$. Given  $\rho>0$, the $\rho$-{\em
variation} of $\FF$ at $x\in\R^d$ is defined by
\begin{equation*}
\VV_{\rho}(\FF)(x):=\sup_{\{\epsilon_{m}\}}\bigg(\sum_{m\in\Z}
|F_{\epsilon_{m+1}}(x)-F_{\epsilon_{m}}(x)|^{\rho}\bigg)^{1/\rho},
\end{equation*}
where the pointwise supremum is taken over all decreasing
sequences $\{\epsilon_{m}\}_{m\in\Z}\subset(0,\infty)$. Fix a decreasing
sequence $\{r_{m}\}_{m\in\Z}\subset(0,\infty)$. The {\em oscillation} of $\FF$
at $x\in\R^d$ is defined by
\begin{equation*}
\OO(\FF)(x):=\sup_{\{\epsilon_m\},\{\delta_{m}\}}\bigg(\sum_{m\in\Z}
|F_{\epsilon_m}(x)-F_{\delta_m}(x)|^{2}\bigg)^{1/2},
\end{equation*}
where the pointwise supremum is taken over all sequences
$\{\epsilon_m\}_{m\in\Z}$ and $\{\delta_{m}\}_{m\in\Z}$
such that $r_{m+1}\leq\epsilon_{m}\leq\delta_{m}\leq r_{m}$ for all
$m\in\Z$.
\end{defi}

The $\rho$-variation and oscillation for martingales and some
families of operators have been studied in many recent papers on
probability, ergodic theory, and harmonic analysis (see
\cite{Lepingle}, \cite{Bourgain}, \cite{JKRW-ergodic},
\cite{CJRW-Hilbert}, \cite{JSW}, \cite{Lacey}, and
\cite{OSTTW}, for example).
In this paper we are interested in the $\rho$-variation and oscillation of the family $\TT^\mu f$. That is, given 
a Borel measure $\mu$ in $\R^d$ and $f\in L^1(\mu)$ we will deal with
\begin{equation*}
\begin{split}
&(\VV_{\rho}\circ\TT^\mu)f(x):=\VV_{\rho}(\TT^\mu f)(x),\quad
(\OO\circ\TT^\mu)f(x):=\OO(\TT^\mu f)(x).
\end{split}
\end{equation*} 
We are specially interested in the case $\TT^\mu=\RR^\mu$. Notice, by the way, that $T_*^\mu f(x)
\leq (\VV_{\rho}\circ\TT^\mu)f(x)$ for 
any compactly supported function $f\in L^1(\mu)$ and all $x\in\R^d$.

When $\mu$ coincides with the Lebesgue
measure in the real line and $K(x)=1/x$ is the kernel of the Hilbert transform, Campbell, Jones,  Reinhold and Wierdl
\cite{CJRW-Hilbert} showed that 
$\VV_{\rho}\circ\TT^\mu$ and $\OO\circ\TT^\mu$ are
bounded in $L^p(\mu)$, for $1<p<\infty$, and of weak type $(1,1)$.
This result was extended to other singular integral operators in higher dimensions in 
\cite{CJRW-singular integrals}. The case of the Cauchy transform and other odd Calder\'on-Zygmund
operators on Lipschitz graphs was studied recently in \cite{MT}.

Let us turn our attention to uniform rectifiability now.
Recall that a Borel measure $\mu$ in $\R^d$ is called $n$-rectifiable if there exists a countable family of $n$-dimensional $C^1$ submanifolds $\{M_i\}_{i\in\N}$ in $\R^d$ such that $\mu(E\setminus\bigcup_{i\in\N}M_i)=0$.
Moreover, $\mu$ is said to be $n$-dimensional Ahlfors-David
regular, or simply AD regular, if there exists some constant $C>0$
such that $C^{-1}r^n\leq\mu(B(x,r))\leq Cr^n$ for all $x\in\supp\mu$
and $0<r\leq\diam(\supp\mu)$.
One also says that $\mu$ is uniformly $n$-rectifiable if there exist $\theta,M>0$ so that, for each
$x\in\supp\mu$ and $r>0$, there is a Lipschitz mapping $g$ from
the $n$-dimensional ball $B^n(0,r)\subset\R^n$ into $\R^d$ such that
$\Lip(g)\leq M$ and
$\mu\big(B(x,r)\cap g(B^n(0,r))\big) \geq \theta r^n,$
where $\Lip(g)$ stands for the Lipschitz constant of $g$. In particular, uniform rectifiability implies rectifiability. Given a set $E\subset\R^d$, we denote by $\HH^n_E$ the $n$-dimensional Hausdorff measure restricted to $E$. Then $E$ is called, respectively, $n$-rectifiable, AD regular, or uniformly $n$-rectifiable if $\HH^n_E$ is so. By the Lebesgue differentiation theorem, any $n$-dimensional AD regular measure $\mu$ is of the form $\mu=f\HH^n_{\supp\mu}$ with $C^{-1}\leq f(x)\leq C$ for some constant $C>0$ and all $x\in\supp\mu$.

G. David and S. Semmes asked more than twenty years ago the following question, which is still open (see, for example, \cite[Chapter 7]{Pajot}):
\begin{ques}\label{ques David-Semmes}
Is it true that
an $n$-dimensional AD regular measure $\mu$ is uniformly
$n$-rectifiable if and only if $R_*^\mu$ is bounded in $L^2(\mu)$?
\end{ques}

Some comments are in order. By the results in \cite{DS1}, the ``only if'' implication of the question above is already known to hold. Also in \cite{DS1}, G. David and S. Semmes gave a positive answer to Question \ref{ques David-Semmes} if one replaces the $L^2$ boundedness of $R^\mu_*$ by the $L^2$ boundedness of $T^\mu_*$ for a wide class of odd kernels $K$. In the case $n=1$ (in particular, for the Cauchy transform), the ``if'' implication was proved by P. Mattila, M. Melnikov and J. Verdera in \cite{MMV} using the notion of curvature of measures. Later on,  G. David and J. C. L\'{e}ger \cite{Leger} proved that the $L^2$ boundedness $C^\mu_*$ implies that $\mu$ is rectifiable, even without the AD regularity assumption (with $n=1$). 

When $\mu$ is the $n$-dimensional Hausdorff measure on a set $E\subset\R^d$ such that $\mu(E)<\infty$,  the rectifiability of $\mu$ is also related with the existence $\mu$-a.e. of the principal value of the Riesz transform of $\mu$, that is,  the existence of $R^\mu1(x)=\lim_{\epsilon\searrow0}R^\mu_\epsilon1(x)$ for $\mu$-a.e. $x\in E$. In \cite{Mattila-Preiss}, P. Mattila and D. Preiss proved that, under the additional assumption that 
\begin{equation}\label{eqjo88}
\lim\inf_{r\to 0}r^{-n}\mu(B(x,r))>0 \qquad\mbox{for $\mu$-a.e. $x\in E$,}
\end{equation}
 the rectifiability of $E$ is equivalent to the existence of $R^\mu1(x)$ $\mu$-a.e. $x\in E$. 
Later on, in \cite{Tolsa8}  X. Tolsa removed the assumption \eqref{eqjo88} and proved the result in full generality. Let us mention that, for the case $n=1$ and $d=2$ (that is, for the Cauchy transform), the analogous results had been obtained previously by \cite{Mattila2} under the assumption \eqref{eqjo88}, and in \cite{Tolsa7}, in full generality, by using the notion of curvature of measures.

In this paper we prove the following:

\begin{teo}\label{teo88}
Let $1\leq n<d$ and $\rho>2$. An $n$-dimensional AD regular Borel measure $\mu$ in $\R^d$ is uniformly $n$-rectifiable if and only if $\VV_\rho\circ\RR^\mu$ is a bounded operator in $L^2(\mu)$. Moreover, if $\mu$ is $n$-uniformly rectifiable, then for any kernel $K$ satisfying \eqref{4eq333}, the operator $\VV_\rho\circ\TT^\mu$ is bounded in $L^2(\mu)$.
\end{teo}

Let us compare this result with the David-Semmes Question \ref{ques David-Semmes}.
Notice that the preceding theorem asserts that if we replace the $L^2(\mu)$ boundedness of $R_*^\mu$ by the stronger
assumption that $\VV_\rho\circ\RR^\mu$ is bounded in $L^2(\mu)$, then $\mu$ must be uniformly rectifiable.
On the other hand, the theorem claims that the variation for odd singular integral operators with any kernel 
satisfying \eqref{4eq333}, in particular for the $n$-dimensional Riesz transforms, is bounded in $L^2(\mu)$.

A natural question then arises. Given an arbitrary measure $\mu$ on $\R^d$, without atoms say,  
does the $L^2(\mu)$ boundedness of $R_*^\mu$ implies the $L^2(\mu)$ boundedness of $\VV_\rho\circ\RR^\mu$, for $\rho>2$?
By the results of \cite{MMV} and Theorem \ref{teo88}, this is true in the case $n=1$ if $\mu$ is AD regular
$1$-dimensional. Clearly, a positive answer in the general case $n\geq1$ would solve the David-Semmes problem in the affirmative.
Nevertheless, such an approach to try to solve this problem looks quite difficult. In fact, we recall
that is not even known if the $L^2(\mu)$ boundedness of $R_*^\mu$ ensures the $\mu$-a.e.\ existence of the
principal values of $R^\mu1$, which is a necessary condition for the
$L^2(\mu)$ boundedness of $\VV_\rho\circ\RR^\mu$.

Concerning the proof of Theorem \ref{teo88}, in our previous paper \cite{MT} 
we showed that, if $\mu$ stands for the $n$-dimensional Hausdorff-measure on an $n$-dimensional
Lipschitz graph, then
 the $\rho$-variation for Riesz transforms and odd Calder\'on-Zygmund operators with
smooth truncations are bounded in $L^2(\mu)$. This is a fundamental step to prove that 
$\VV_\rho\circ\RR^\mu$ and, more generally,  $\VV_\rho\circ\TT^\mu$, are bounded in $L^2(\mu)$ if $\mu$ is 
uniformly $n$-rectifiable. Another basic tool in our arguments is the geometric corona decomposition
of uniformly rectifiable measures introduced by David and Semmes in \cite{DS1}, which, roughly speaking, 
describes how $\supp(\mu)$ can be approximated at different scales by $n$-dimensional Lipschitz graphs.
 
The proof of the fact that the $L^2(\mu)$ boundedness of 
$\VV_\rho\circ\RR^\mu$ implies the uniform rectifiability of $\mu$ is not so laborious as the one
of the converse implication. 
As remarked above, if $\VV_\rho\circ\RR^\mu$ is bounded in $L^2(\mu)$, then the principal
values of $R^\mu1$ exist $\mu$-a.e., which implies the $n$-rectifiability of $\mu$, by the results of 
\cite{Mattila-Preiss} or \cite{Tolsa8}. However, this is not enough to ensure the {\em uniform} $n$-rectifiability of $\mu$.
We will prove the uniform $n$-rectifiability by arguments partially inspired by some of the techniques in \cite{To}.

Finally, let us remark that Theorem \ref{teo88} follows from a more general result, namely Theorem
\ref{4main theorem} below, which also deals with the variation for Riesz transforms and odd Calder\'on-Zygmund 
operators with smooth truncations.



As usual, in the paper the letter `$C$' stands
for some constant which may change its value at different
occurrences, and which quite often only depends on $n$ and $d$. The notation $A\lesssim B$ ($A\gtrsim B$) means that
there is some fixed constant $C$ such that $A\leq CB$ ($A\geq CB$),
with $C$ as above. Also, $A\approx B$ is equivalent to $A\lesssim B\lesssim A$.

\section{Preliminaries}

\subsection{The main theorem}

\begin{defi}[families of truncations]\label{4defi varphi}
Let $\chi_\R:=\chi_{[1,\infty)}$ and let $\varphi_\R:[0,+\infty)\to[0,+\infty)$ be a non decreasing $\CC^2$ function with $\chi_{[4,\infty)}\leq\varphi_\R\leq
\chi_{[1/4,\infty)}$. Suppose moreover that $|\varphi'_\R|$ is bounded below away from zero in $[1/3,3]$, i.e., $\chi_{[1/3,3]} \leq C|\varphi'_\R|$ for some $C>0$.

Given $x\in\R^d$, and $0<\epsilon\leq\delta$, we set 
\begin{gather*}
\chi_\epsilon(x):=\chi_{\R}(|x|/\epsilon)\quad\text{and}\quad
\chi_\epsilon^\delta(x):=\chi_\epsilon(x)-\chi_\delta(x),\\
\varphi_\epsilon(x):=\varphi_{\R}(|x|^2/\epsilon^2)\quad\text{and}\quad
\varphi_\epsilon^\delta(x):=\varphi_\epsilon(x)-\varphi_\delta(x).
\end{gather*}
Notice that, for any finite Borel measure $\mu$, $T_{\epsilon}\mu(x) = (K\chi_\epsilon*\mu)(x)$.
Given $x=(x^1,\ldots,x^d)\in\R^d$, we denote $\wit x=(x^1,\ldots,x^n,0,\ldots,0)\in\R^d$, and we set $\wit\varphi_\epsilon(x):=\varphi_\epsilon(\wit x)$ and $\wit\varphi_\epsilon^\delta(x):=\varphi_\epsilon^\delta(\wit x)$.
Finally, for $f\in L^1(\mu)$ we set $\TT^\mu f\equiv \TT(f\mu):=\{T_\epsilon^\mu f\}_{\epsilon>0}$, 
\begin{gather*}
T_{\varphi_\epsilon}^\mu f(x)\equiv T_{\varphi_\epsilon}(f\mu)(x):=(K\varphi_\epsilon*\mu)(x)
\quad\text{and}\quad\TT_\varphi^\mu f\equiv \TT_\varphi(f\mu):=\{T_{\varphi_\epsilon}^\mu f\}_{\epsilon>0},\\
T_{\wit\varphi_\epsilon}^\mu f(x)\equiv T_{\wit\varphi_\epsilon}(f\mu)(x):=(K\wit\varphi_\epsilon*\mu)(x)
\quad\text{and}\quad\TT_{\wit\varphi}^\mu f\equiv \TT_{\wit\varphi}(f\mu):=\{T_{\wit\varphi_\epsilon}^\mu f\}_{\epsilon>0}.
\end{gather*}
\end{defi}

\begin{remarko}{\em
In the definition, the choice of $[4,\infty)$, $[1/4,\infty)$, and $[1/3,3]$ is not specially relevant, it is just for definiteness. One can replace the preceding  intervals by other suitable intervals, and all the proofs in the paper remain almost the same.}
\end{remarko}

We will prove the following.

\begin{teo}[Main Theorem]\label{4main theorem}
Let $1\leq n<d$ be integers. Let $\mu$ be an $n$-dimensional AD regular Borel measure on $\R^d$. The following are equivalent:
\begin{itemize}
\item[$(a)$] $\mu$ is uniformly $n$-rectifiable.
\item[$(b)$] For any $K$ satisfying \eqref{4eq333} and any $\rho>2$, the operator $\VV_\rho\circ\TT^\mu_\varphi$ is bounded in $L^p(\mu)$ for all $1<p<\infty$, and from $L^1(\mu)$ 
into $L^{1,\infty}(\mu)$.
\item[$(c)$] For any $K$ satisfying \eqref{4eq333} and any $\rho>2$, the operator $\VV_\rho\circ\TT^\mu$ is bounded in $L^2(\mu)$.
\item[$(d)$] For some $\rho>0$, the operator $\VV_\rho\circ\RR^\mu$ is bounded in $L^2(\mu)$.
\item[$(e)$] For $K(x)=x/|x|^{n+1}$ and some $\rho>0$, the operator $\VV_\rho\circ\TT^\mu_\varphi$ is bounded in $L^2(\mu)$.
\end{itemize}
\end{teo}

Clearly, Theorem \ref{teo88} is a direct consequence of the preceding result.

\begin{remarko}\label{4remark oscil2}{\em
Let $\{r_m\}_{m\in\Z}\subset(0,\infty)$ be a fixed decreasing sequence defining $\OO$. Then, the implications $(a)\Rightarrow(b),\ldots,(e)$ in the theorem above still hold if one replaces $\VV_\rho$ by $\OO$.  If there exists $C>0$ such that $C^{-1}r_m\leq r_m-r_{m+1}\leq Cr_m$ for all $m\in\Z$, then the implications $(b),\ldots,(e)\Rightarrow(a)$ also hold (so Theorem \ref{4main theorem} remains true replacing $\VV_\rho$ by $\OO$), but we do not know if they are still true without this additional assumption (see Remark \ref{4remark oscil}).}
\end{remarko}

Notice that, by Theorem \ref{4main theorem}, besides $\VV_\rho\circ\RR^\mu$ and $\OO\circ\RR^\mu$, the operators $\VV_\rho\circ\TT^\mu_\varphi$ and $\OO\circ\TT^\mu_\varphi$ for $K(x)=x/|x|^{n+1}$  characterize completely the $n$-AD regular measures $\mu$ which are uniformly $n$-rectifiable. 

One of the main ingredients for the proof of Theorem \ref{4main theorem} is the following result, which strengthens one of the endpoint estimates obtained in \cite{MT}. Let $M(\R^d)$ be the space of finite real Borel measures on $\R^d$, with the norm induced by the variation of measures.
\begin{teo}\label{4unif rectif teo3}
Let $\rho>2$ and let $\mu$ be the $n$-dimensional Hausdorff measure restricted to an $n$-dimensional Lipschitz graph.  
Then, $\VV_\rho\circ\TT_\varphi$ is a bounded operator from $M(\R^d)$ to $L^{1,\infty}(\mu)$.
In particular, $\VV_\rho\circ\TT_\varphi^{\mu}$ is of weak type $(1,1)$. 
The bound of the norm of this operator only depends on $n$, $d$, $K$, $\rho$, $\varphi_\R$, and the maximal slope of $\Gamma$. 
\end{teo}

By an $n$-dimensional Lipschitz graph $\Gamma\subset\R^d$ we mean any translation and rotation of a set of the type $\{x\in\R^d\,:\,x=(y,A(y)),\, y\in\R^n\}$, where $A:\R^n\to\R^{d-n}$ is some Lipschitz function with Lipschitz constant $\Lip(A)$, which coincides with the maximal slope of $\Gamma$. 

\begin{remarko}{\em 
The theorem above remains valid if one replaces $\VV_\rho$ by $\OO$. Moreover, the norm of $\OO\circ\TT_\varphi^{\mu}$ is bounded independently of the sequence that defines $\OO$.}
\end{remarko}

The plan to prove Theorem \ref{4main theorem}
is the following: in Section \ref{4sec unif rectif teo3} we deal with Theorem \ref{4unif rectif teo3}, which is used in the subsequent Section \ref{4sec acotacio unif rectif} to obtain the
implication $(a)\,\Longrightarrow\,(b)$ of Theorem \ref{4main theorem}. In Section \ref{5s var no suau} we prove $(a)\,\Longrightarrow\,(c)$ in Theorem \ref{5teo var no suau acotada L2}, and in Section \ref{4sec acotacio implica rectif} we prove Theorem \ref{4rectif teorema}, which gives $(d)\,\Longrightarrow\,(a)$ and $(e)\,\Longrightarrow\,(a)$, and finishes the proof of Theorem \ref{4main theorem}, taking into account that the implications  $(b)\,\Longrightarrow\,(e)$ and $(c)\,\Longrightarrow\,(d)$ are trivial.

Theorems \ref{4main theorem} and \ref{4unif rectif teo3} are stated in terms of $\VV_\rho$, but they also hold for $\OO$, as remarked above. However, we will only give the proof of these results for $\VV_\rho$, because the case of $\OO$ follows by very similar arguments and computations.

\subsection{Calder\'{o}n-Zygmund decomposition for measures}\label{4secdob}

Given a cube $Q\subset\R^d$ and $a>0$, we denote by $\ell(Q)$ the side length of $Q$ and by $aQ$ the cube concentric with $Q$ with side length $a\ell(Q)$. The cubes that we consider in this paper have sides parallel to the coordinate axes in $\R^d$.

A proof of the following result can be found in \cite[Chapter 2]{Tolsa-llibre} or \cite[Lemma 5.1.2]{Mas-thesis}.

\begin{lema}[Calder\'{o}n-Zygmund decomposition]\label{4lema CZ}
Assume that $\mu:=\HH^n_{\Gamma\cap B}$, where $\Gamma$ is an $n$-dimensional Lipschitz graph and $B\subset\R^d$ is some fixed ball. For any $\nu\in M(\R^d)$ with compact support and any $\lambda>2^{d+1}\|\nu\|/\|\mu\|$, the following holds:
\begin{itemize}
\item[$(a)$] There exists a finite or countable collection of almost disjoint cubes $\{Q_j\}_j\subset\R^d$ (that is, $\sum_j\chi_{Q_j}\leq C$) and a function $f\in L^1(\mu)$ such that
\begin{gather}
|\nu|(Q_j)>2^{-d-1}\lambda\mu(2Q_j),\label{4lema CZ 1}\\
|\nu|(\eta Q_j)\leq2^{-d-1}\lambda\mu(2\eta Q_j)\quad\text{for }\eta>2,\label{4lema CZ 2}\\
\nu=f\mu\text{ in }\R^d\setminus\textstyle{\bigcup_j}Q_j\text{ with }|f|\leq\lambda\;\,\mu\text{-a.e}.\label{4lema CZ 3}
\end{gather}
\item[$(b)$]For each $j$, let $R_j:=6Q_j$ and denote $w_j:=\chi_{Q_j}\big(\sum_k\chi_{Q_k}\big)^{-1}$. Then, there exists a family of functions $\{b_j\}_j$ with $\supp b_j\subset R_j$ and with constant sign satisfying
\begin{gather}
\int b_j\,d\mu=\int w_j\,d\nu,\label{4lema CZ 4}\\
\| b_j\|_{L^\infty(\mu)}\mu(R_j)\leq C|\nu|(Q_j),\text{ and}\label{4lema CZ 5}\\
{\textstyle\sum_j}|b_j|\leq C_0\lambda\quad
\text{(where $C_0$ is some absolute constant).}\label{4lema CZ 6}
\end{gather}
\end{itemize}
\end{lema}

\subsection{Dyadic lattices}\label{dyadic lattice}
For the study of the uniformly rectifiable measures we will use the ``dyadic cubes'' built by G. David in \cite[Appendix 1]{David-LNM} (see also \cite[Chapter 3 of Part I]{DS2}). These dyadic cubes are not true cubes, but they play this role with respect to a given $n$-dimenasional AD regular Borel measure $\mu$, in a sense. To distinguish them from the usual cubes, we will call them {\em $\mu$-cubes}. 

Let us explain
which are the precise results and properties about the lattice of dyadic $\mu$-cubes.
Given an $n$-dimensional AD regular Borel measure $\mu$ in $\R^d$ (for simplicity, we may assume $\diam(\supp\mu)=\infty$), for each $j\in\Z$ there exists a family $\DD_j$ of Borel subsets of $\supp\mu$ (the dyadic $\mu$-cubes of the $j$-th generation) such that:
\begin{itemize}
\item[$(a)$] each $\DD_j$ is a partition of $\supp\mu$, i.e.\ $\supp\mu=\bigcup_{Q\in \DD_j} Q$ and $Q\cap Q'=\emptyset$ whenever $Q,Q'\in\DD_j$ and
$Q\neq Q'$;
\item[$(b)$] if $Q\in\DD_j$ and $Q'\in\DD_k$ with $k\leq j$, then either $Q\subset Q'$ or $Q\cap Q'=\emptyset$;
\item[$(c)$] for all $j\in\Z$ and $Q\in\DD_j$, we have $2^{-j}\lesssim\diam(Q)\leq2^{-j}$ and $\mu(Q)\approx 2^{-jn}$;
\item[$(d)$] there exists $C>0$ such that, for all $j\in\Z$, $Q\in\DD_j$, and $0<\tau<1$,
\begin{equation}\label{small boundary condition}
\begin{split}
\mu\big(\{x\in Q:\, &\dist(x,\supp\mu\setminus Q)\leq\tau2^{-j}\}\big)\\&+\mu\big(\{x\in \supp\mu\setminus Q:\, \dist(x,Q)\leq\tau2^{-j}\}\big)\leq C\tau^{1/C}2^{-jn}.
\end{split}
\end{equation}
This property is usually called the {\em small boundaries condition}.
From (\ref{small boundary condition}), it follows that there is a point $z_Q\in Q$ (the center of $Q$) such that $\dist(z_Q,\supp\mu\setminus Q)\gtrsim 2^{-j}$ (see \cite[Lemma 3.5 of Part I]{DS2}).
\end{itemize}

We denote $\DD:=\bigcup_{j\in\Z}\DD_j$. For $Q\in \DD_j$, we define the side length
of $Q$ as $\ell(Q)=2^{-j}$. Notice that $\ell(Q)\lesssim\diam(Q)\leq \ell(Q)$.
Actually it may happen that a $\mu$-cube $Q$ belongs to $\DD_ j\cap \DD_k$ with $j\neq k$. In this case, $\ell(Q)$ is not well defined. However, this problem can be solved in many ways.
For example, the reader may think that a $\mu$-cube is not only a subset of $\supp\mu$, but a couple $(Q,j)$, where $Q$ is a subset of $\supp\mu$ and $j\in\Z$ is such that $Q\in\DD_j$. 

Given $a>1$ and $Q\in\DD$, we set
$a Q:= \bigl\{x\in \supp\mu: \dist(x,Q)\leq (a-1)\ell(Q)\bigr\}.$
Observe that $\diam(a Q)\leq \diam(Q) + 2(a-1)\ell(Q)\leq (2a-1)\ell(Q)$.

\subsection{Corona decomposition}\label{5ss corona decomposition}
Given an $n$-dimensional AD regular Borel measure $\mu$ on $\R^d$, let $\DD:=\{Q\in\DD_j\,:\,j\in\Z\}$ be the dyadic lattice associated to $\mu$ introduced in Subsection \ref{dyadic lattice}. 
Following \cite[Definitions 3.13 and 3.19 of Part I]{DS2}, one says that $\mu$ admits a corona decomposition if, for each $\eta>0$ and $\theta>0$, one can find a triple $(\BB,\GG,\Trees)$, where $\BB$ and $\GG$ are two subsets of $\DD$ (the ``bad $\mu$-cubes'' and the ``good $\mu$-cubes'') and $\Trees$ is a family of subsets $S\subset\GG$ (that we will call {\em trees}), which satisfy the following conditions::
\begin{itemize}
\item[$(a)$] $\DD=\BB\cup\GG$\quad and\quad$\BB\cap\GG=\emptyset.$
\item[$(b)$] $\BB$ satisfies a Carleson packing condition, i.e., $\sum_{Q\in\BB:\, Q\subset R}\mu(Q)\lesssim\mu(R)$ for all $R\in\DD$.
\item[$(c)$] $\GG=\biguplus_{S\in\Trees}S$, i.e., any $Q\in\GG$ belongs to only one $S\in\Trees$.
\item[$(d)$] Each $S\in\Trees$ is {\em coherent}. This means that each $S\in\Trees$ has a unique maximal element $Q_S$ which contains all other elements of $S$ as subsets, that $Q'\in S$ as soon as $Q'\in\DD$ satisfies $Q\subset Q'\subset Q_S$ for some $Q\in S$, and that if $Q\in S$ then either all of the children of $Q$ lie in $S$ or none of them do (if $Q\in\DD_j$, the {\em children} of $Q$ is defined as the collection of $\mu$-cubes $Q'\in\DD_{j+1}$ such that $Q'\subset Q$).
\item[$(e)$] The maximal $\mu$-cubes $Q_S$, for $S\in\Trees$, satisfy a Carleson packing condition. That is, $\sum_{S\in\Trees:\, Q_S\subset R}\mu(Q_S)\lesssim\mu(R)$ for all $R\in\DD$.
\item[$(f)$] For each $S\in\Trees$, there exists an $n$-dimensional Lipschitz graph $\Gamma_S$ with constant smaller than $\eta$ such that $\dist(x,\Gamma_S)\leq\theta\,\diam(Q)$ whenever $x\in2Q$ and $Q\in S$ (one can replace ``$x\in2Q$'' by ``$x\in C_{cor}Q$'' for any constant $C_{cor}\geq2$ given in advance, by \cite[Lemma 3.31 of Part I]{DS2}).
\end{itemize}

It is shown in \cite{DS1} (see also \cite{DS2}) that if $\mu$ is uniformly rectifiable then it
admits a corona decomposition for all parameters $k>2$ and $\eta,\theta>0$. Conversely,
the existence of a corona decomposition for a single set of parameters $k>2$ and $\eta,\theta>0$ implies that $\mu$ is uniformly rectifiable.

\subsection{The $\alpha$ and $\beta$ coefficients}
Let  $\mu$ be an $n$-dimensional AD regular Borel measure in $\R^d$ and $\DD$ as in Subsection \ref{dyadic lattice}. Given $1\leq p<\infty$ and a $\mu$-cube $Q\in\DD$, one sets (see \cite{DS2})
\begin{equation*}
\beta_{p,\mu}(Q) = \inf_L
\biggl\{ \frac1{\ell(Q)^n}\int_{2Q} \biggl(\frac{\dist(y,L)}{\ell(Q)}\biggr)^pd\mu(y)\biggr\}^{1/p},
\end{equation*}
where the infimum is taken over all $n$-planes $L$ in $\R^d$.
For $p=\infty$ one replaces the $L^p$ norm by the supremum norm. The $\beta_{\infty,\mu}$ coefficients were first introduced by P. Jones in his celebrated work on rectifiability \cite{Jones-salesman}, while the $\beta_{p,\mu}$'s for $1\leq p<\infty$ were introduced by G. David and S. Semmes in their pioneering work on uniform rectifiability (see \cite{DS1} for example).

Other coefficients that have been proved useful in the study of uniform rectifiability and boundedness of Calder\'on-Zygmund operators are the $\alpha$ coefficients introduced in \cite{To}.
Let $F\subset\R^d$ be the closure of an open set. Given two finite
Borel measures $\sigma$, $\nu$ on $\R^d$, one sets
$\dist_F(\sigma,\nu):= \sup\bigl\{ \bigl|{\textstyle \int f\,d\sigma -
\int f\,d\nu}\bigr|:\,{\rm Lip}(f) \leq1,\,\supp f\subset
F\bigr\}.$
Finally, given a $\mu$-cube $Q\in\DD$,
consider the closed ball $B_Q:=B(z_Q,6\sqrt{d}\ell(Q))$, where $z_Q$ denotes the center of $Q$. Then one defines
\begin{equation*}
\alpha_\mu(Q):=\frac1{\ell(Q)^{n+1}}\,\inf_{c\geq0,L} \,\dist_{B_Q}(\mu,\,c\HH^n_{L}),
\end{equation*}
where the infimum is taken over all constants $c\geq0$ and all $n$-planes $L$ in $\R^d$.

The following result characterizes the uniform rectifiability of $\mu$ in terms of the $\alpha$ and $\beta$ coefficients (see \cite{DS1} for $(a)\Longleftrightarrow(b)$ and \cite{To} for $(a)\Longleftrightarrow(c)$).
\begin{teo}\label{alphas unifrectif}
Let $p\in[1,2]$ and let $\mu$ be an $n$-dimensional AD regular Borel measure in $\R^d$. The following are equivalent:
\begin{itemize}
\item[$(a)$] $\mu$ is uniformly $n$-rectifiable.
\item[$(b)$] $\sum_{Q\in\DD:\,Q\subset R}\beta_{p,\mu}(Q)^2\ell(Q)^n\lesssim\ell(R)^n$ for all $\mu$-cubes $R\in\DD$.
\item[$(c)$] $\sum_{Q\in\DD:\,Q\subset R}\alpha_{\mu}(Q)^2\ell(Q)^n\lesssim\ell(R)^n$ for all $\mu$-cubes $R\in\DD$.
\end{itemize}
\end{teo}

For the case $\mu=\HH^n_\Gamma$ for some Lipschitz graph $\Gamma=\{x\in\R^d\,:\,x=(y,A(y)),\, y\in\R^n\}$, one can take  $\DD=\{\wit Q\times\R^{d-n}\cap\Gamma:\,\wit Q\in\DD(\R^n)\}$, where $\DD(\R^n)$ denotes the standard dyadic lattice of $\R^n$. For $Q= (\wit Q\times\R^{d-n})\cap\Gamma\in\DD$, we set
\begin{equation*}
\wit\alpha_\mu(Q):=\frac1{\ell(\wit Q)^{n+1}}\,\inf_{c\geq0,L} \,\dist_{6\wit Q\times\R^{d-n}}(\mu,\,c\HH^n_{L}),
\end{equation*}
where the infimum is taken over all constants $c\geq0$ and all $n$-planes $L$ in $\R^d$. Then, it is easy to show that $\wit\alpha_\mu(Q)\approx\alpha_\mu(Q)$ for all $Q\in\DD$. 

One can also define $\wit\beta_{p,\mu}(Q)$ in an analogous manner. By Theorem \ref{alphas unifrectif},
\begin{equation}\label{pack alpha graph}
\sum_{Q\in\DD:\,Q\subset R}(\wit\beta_{p,\mu}(Q)^2+\wit\alpha_\mu(Q)^2)\ell(Q)^n\leq C\ell(R)^n
\end{equation}
for all $R\in\DD$, with $C$ independent of $R$. Moreover, one can also show that this last inequality also holds replacing $Q$ and $R$ by $k_1Q$ and $k_2R$ for any $k_1,k_2\geq1$ given in advance, where $kQ:=(k\wit Q\times\R^{d-n})\cap\Gamma$ for $k>0$.

\section{If $\Gamma$ is an $n$-dimensional Lipschitz graph, then\\
$\VV_\rho\circ\TT_\varphi:\,M(\R^d)\to L^{1,\infty}(\HH^n_\Gamma)$ is a bounded operator}\label{4sec unif rectif teo3}

The following result is contained in \cite[Theorem 1.1]{MT} (see also \cite[Main Theorem 3.0.1]{Mas-thesis}).
\begin{teo}\label{main teo lip}
Let $\rho>2$ and let $\mu$ be the $n$-dimensional Hausdorff measure restricted to an $n$-dimensional Lipschitz graph.  
Then, the operator $\VV_\rho\circ\TT^\mu_{\wit\varphi}$ is bounded in $L^{2}(\mu)$. The bound of the norm only depends on $n$, $d$, $K$, $\rho$, $\varphi_\R$, and the slope of the graph. 
\end{teo}

By very similar techniques to the ones used in the proof of the theorem above, one can prove the following.
\begin{teo}\label{theorem lip}
Let $\rho>2$ and let $\mu$ be the $n$-dimensional Hausdorff measure restricted to an $n$-dimensional Lipschitz graph.  
Then, the operator $\VV_\rho\circ\TT^\mu_{\varphi}$ is bounded in $L^{2}(\mu)$. The bound of the norm only depends on $n$, $d$, $K$, $\rho$, $\varphi_\R$, and the slope of the graph. 
\end{teo}

\begin{proof}[{\bf{\em Sketch of the proof}}]
The first step consists in obtaining the following basic estimate: 
Fix a cube $\widetilde P\subset\R^n$. Set $\Gamma:=\{x\in
\R^d\,:\,x=(y,A(y)),\,y\in\R^n\},$ where
$A:\R^n\to\R^{d-n}$ is a Lipschitz function supported in $\widetilde
P$, and set $P:=(\widetilde P\times \R^{d-n})\cap\Gamma$. Set
$\mu:=f\HH^n_{\Gamma}$, where $f(x)=1$ for all $x\in
\Gamma\setminus P$ and $C_0^{-1}\leq f(x)\leq C_0$ for all $x\in P$, for some constant $C_0>0$.

For each $x\in\Gamma$, define
\begin{equation}\label{definicio Wmu}
W\mu(x)^2:=\sum_{m\in\Z}|(K\varphi_{2^{-m}}*\mu)(x)-(K\widetilde \varphi_{2^{-m}}*\mu)(x)|^{2}.
\end{equation}
and
\begin{equation}\label{definicio Smu}
S\mu(x)^2:=\sup_{\{\epsilon_{m}\}}\sum_{j\in\Z}
\,\sum_{m\in\Z:\,\epsilon_{m},\epsilon_{m+1}\in I_j}
|(K\varphi_{\epsilon_{m+1}}^{\,\epsilon_{m}}*\mu)(x)|^{2},
\end{equation}
where $I_j=[2^{-j-1},2^{-j})$ and the supremum is taken over all decreasing
sequences of positive numbers $\{\epsilon_{m}\}_{m\in\Z}$.
Then, we claim that
\begin{equation}\label{basic estimate}
\left\|W\mu\right\|^{2}_{L^{2}(\mu)}+\left\|S\mu\right\|^{2}_{L^{2}(\mu)}\lesssim \sum_{Q\in\DD}\big(\,\wit\alpha_\mu(C_1Q)^{2}+\wit\beta_{2,\mu}(Q)^2\,\big)\ell(Q)^n,
\end{equation}
where $C_1>0$ only depends on $C_0$, $n$, $d$, $K$, $\varphi_\R$, and $\Lip(A)$, and where $\DD$ denotes the dyadic lattice associated to $\HH^n_\Gamma$ defined below Theorem \ref{alphas unifrectif}.

Let us prove the claim. If we define $\wit S\mu$ like $S\mu$ but replacing $\varphi_{\epsilon_{m+1}}^{\,\epsilon_{m}}$ by $\wit\varphi_{\epsilon_{m+1}}^{\,\epsilon_{m}}$, in the proof of Theorem \ref{main teo lip} in \cite{MT} it is shown that $\|\wit S\mu\|^{2}_{L^{2}(\mu)}$ is bounded above by the right hand side of (\ref{basic estimate}). The proof for $\|S\mu\|^{2}_{L^{2}(\mu)}$ is almost the same.

Let us deal now with $W\mu$. Fix $D:=(\wit D\times\R^{d-n})\cap\Gamma\in\DD$ with $\ell(D)=2^{-m}$ and $x\in D$. Let $L_D$ be an $n$-plane that minimizes $\wit\alpha_\mu(C_1D)$, where $C_1>0$ is some constant big enough which will be fixed later, and let $\sigma_D:=c_D\HH^n_{L_D}$ be a minimizing measure for $\wit\alpha_\mu(C_1D)$. Let $L_D^x$ be the $n$-plane parallel to $L_D$ which contains $x$, and set $\sigma^x_D:=c_D\HH^n_{L^x_D}$.

Since $x\in D$ and $\ell(D)=2^{-m}$, $(\varphi_{2^{-m}}(x-\cdot)-\wit\varphi_{2^{-m}}(x-\cdot))K(x-\cdot)$ is a function supported in $C_1\wit D\times\R^{d-n}$ (for some constant $C_1$ big enough) and with Lipschitz constant smaller than $C2^{m(n+1)}$. Moreover, by the antisymmetry of the function $(\varphi_{2^{-m}}(x-\cdot)-\wit\varphi_{2^{-m}}(x-\cdot))K(x-\cdot)$, and since $\sigma_D^x$ is a multiple of the $n$-dimensional Hausdorff measure on an $n$-plane which contains $x$, we have $(K\varphi_{2^{-m}}*\sigma_D^x)(x)-(K\widetilde \varphi_{2^{-m}}*\sigma_D^x)(x)=0.$ Therefore,
\begin{equation}\label{basic eq1}
\begin{split}
(K\varphi_{2^{-m}}&*\mu)(x)-(K\widetilde \varphi_{2^{-m}}*\mu)(x)=
(K(\varphi_{2^{-m}}-\widetilde \varphi_{2^{-m}})*\mu)(x)\\
&=(K(\varphi_{2^{-m}}-\widetilde \varphi_{2^{-m}})*(\mu-\sigma_D))(x)
+(K(\varphi_{2^{-m}}-\widetilde \varphi_{2^{-m}})*(\sigma_D-\sigma_D^x))(x).
\end{split}
\end{equation}
Using the definition of $\wit\alpha_\mu$, we get
\begin{equation}\label{dsds1}
\begin{split}
|(K(\varphi_{2^{-m}}-\widetilde \varphi_{2^{-m}})*(\mu-\sigma_D))(x)|
\lesssim2^{m(n+1)}\dist_{C_1\wit D\times\R^{d-n}}(\mu,\sigma_D)\lesssim\wit\alpha_\mu(C_1 D).
\end{split}
\end{equation}
Since $L_D^x$ is a translation of $L_D$, by standard estimates it is not hard to show that 
\begin{equation}\label{dsds}
|(K(\varphi_{2^{-m}}-\widetilde \varphi_{2^{-m}})*(\sigma_D-\sigma_D^x))(x)|\lesssim2^{m}\dist(x,L_D)=\dist(x,L_D)/\ell(D).
\end{equation}
Let $\dist_\HH(E,F)$ denote the Hausdorff distance of two given sets $E,F\subset\R^d$, and set $\wit B_D:=6\wit D\times \R^{d-n}$. If $L_D^1$ and $L_D^2$ denote a minimizing $n$-plane for $\wit\beta_{1,\mu}(D)$ and $\wit\beta_{2,\mu}(D)$, respectively, one can show that
$\dist_\HH(L_D\cap \wit B_D,L_D^1\cap  \wit B_D)\lesssim\wit\alpha_\mu(D)\ell(D)$ and that $\dist_\HH(L_D^1\cap \wit B_D,L_D^2\cap \wit B_D)\lesssim\wit\beta_{2,\mu}(D)\ell(D)$. This easily implies that $\dist(x,L_D)\lesssim\dist(x,L^2_D)+\wit\beta_{2,\mu}(D)\ell(D)+\wit\alpha_\mu(D)\ell(D)$ for all $x\in D$.
Applying this to (\ref{dsds}), and using also (\ref{dsds1}) and (\ref{basic eq1}), we obtain
\begin{equation*}
\begin{split}
\left\|W\mu\right\|^{2}_{L^{2}(\mu)}&=
\int\sum_{m\in\Z}|(K(\varphi_{2^{-m}}-\widetilde \varphi_{2^{-m}})*\mu)(x)|^{2}\,d\mu(x)\\
&=\sum_{m\in\Z}\,\sum_{D\in\DD:\,\ell(D)=2^{-m}}\int_D|(K(\varphi_{2^{-m}}-\widetilde \varphi_{2^{-m}})*\mu)(x)|^{2}\,d\mu(x)\\
&\lesssim\sum_{m\in\Z}\,\sum_{D\in\DD:\,\ell(D)=2^{-m}}\int_D\big(\dist(x,L^2_D)/\ell(D)+\wit\beta_{2,\mu}(D)+\wit\alpha_\mu(C_1D)\big)^2\,d\mu(x)\\
&\lesssim\sum_{D\in\DD}\big(\wit\alpha_\mu(C_1 D)^2+\wit\beta_{2,\mu}(D)^2\big)\ell(D)^n,
\end{split}
\end{equation*}
which proves (\ref{basic estimate}).

Let now $\mu$ be as in Theorem \ref{theorem lip}. 
Using (\ref{basic estimate}) and Theorem \ref{main teo lip}, one can 
show that there exists $C>0$ such that, for any cube $\widetilde D\subset\R^n$ and any $g\in L^\infty(\mu)$ supported in $D$ (where $D:=\widetilde D\times\R^{d-n}$),
\begin{equation*}
\int_{D}\big((\VV_\rho\circ\TT_{\varphi}^{\mu})g\big)^2\,d\mu\leq
C\|g\|^2_{L^\infty(\mu)}\mu(D).
\end{equation*}
This yields the endpoint estimates $\VV_\rho\circ\TT^\mu_{\varphi}:H^1(\mu)\to L^1(\mu)$ and 
$\VV_\rho\circ\TT^\mu_{\varphi}:L^\infty(\mu)\to BMO(\mu)$, where $H^1(\mu)$ denotes the atomic Hardy space related to $\mu$. Then, by interpolation, one finally deduces that $\VV_\rho\circ\TT^\mu_{\varphi}$ is bounded in $L^2(\mu)$. Since this part of the proof is analogous to the one in the proof of Theorem \ref{main teo lip} (see \cite[Theorem 1.1]{MT}), we omit it.
\end{proof}

\subsection{Proof of Theorem \ref{4unif rectif teo3}}
The proof of Theorem \ref{4unif rectif teo3} uses the Calder\'{o}n-Zygmund decomposition of Lemma \ref{4lema CZ} and rather standard arguments.
Set $\mu:=\HH^n_{\Gamma\cap B}$, where is some fixed ball $B\subset\R^d$. Let $\nu\in M(\R^d)$ be a finite Radon measure with compact support and $\lambda>2^{d+1}\|\nu\|/\|\mu\|$. We will show that
\begin{equation}\label{4cota mesuresL1debil}
\mu\big(\big\{x\in\R^d\,:\,(\VV_\rho\circ\TT_\varphi)\nu(x)>\lambda\big\}\big)\leq\frac{C}{\lambda}\,\|\nu\|,
\end{equation}
where $C>0$ depends on $n$, $d$, $K$, $\rho$ and $\Gamma$, but not on $B$. Let us check that this implies that $\VV_\rho\circ\TT_\varphi$ is bounded from $M(\R^d)$ into $L^{1,\infty}(\HH^n_\Gamma)$. 
First, we show that (\ref{4cota mesuresL1debil}) also holds for $\nu$ without compact support. Set $\nu_N = \chi_{B(0,N)}\,\nu$ and
let $N_0$ be such that $\supp\mu\subset B(0,N_0)$. Then it is not hard to show that, for $x\in\supp\mu$,
$$|(\VV_\rho\circ\TT_\varphi)\nu(x)-(\VV_\rho\circ\TT_\varphi)\nu_N(x)|\leq C\,\frac{|\nu|(\R^d\setminus B(0,N))}{N-N_0},$$
thus $(\VV_\rho\circ\TT_\varphi)\nu_N(x)\to (\VV_\rho\circ\TT_\varphi)\nu(x)$ for all $x\in \supp\mu$, and since the estimate (\ref{4cota mesuresL1debil}) 
 holds by assumption
for $\nu_N$, letting $N\to\infty$, we deduce that it also holds for $\nu$. Now, by increasing the size of the ball $B$ and by monotone convergence, we deduce that 
$\HH^n_\Gamma\big(\big\{x\in\R^d\,:\,(\VV_\rho\circ\TT_\varphi)\nu(x)>\lambda\big\}\big)\leq C\lambda^{-1}\|\nu\|$, as desired. 

To prove (\ref{4cota mesuresL1debil}) for $\nu\in M(\R^d)$ with compact support, let $\{Q_j\}_j$ be the almost disjoint
family of cubes of Lemma \ref{4lema CZ}, and set $\Omega:=\bigcup_jQ_j$ and $R_j:=6Q_j$.
Then we can write $\nu=g\mu+\nu_b$, with
$$g\mu=\chi_{\R^d\setminus\Omega}\nu+ \sum_j b_j\mu\quad\text{and}\quad
\nu_b=\sum_j\nu_b^j:=\sum_j\left(w_j\nu-b_j\mu\right),$$
where the functions $b_j$ satisfy (\ref{4lema CZ 4}), (\ref{4lema CZ 5}), (\ref{4lema CZ 6}) and
$w_j=\chi_{Q_j}\big(\sum_k \chi_{Q_k}\big)^{-1}$.

By the subadditivity of $\VV_\rho\circ\TT_\varphi$, we have
\begin{equation}\label{4cota mesuresL1debil 2}
\begin{split}
\mu\big(&\big\{x\in\R^d\,:\,(\VV_\rho\circ\TT_\varphi)\nu(x)>\lambda\big\}\big)\\
&\leq\mu\big(\big\{x\in\R^d\,:\,(\VV_\rho\circ\TT_\varphi^\mu)g(x)>\lambda/2\big\}\big)
+\mu\big(\big\{x\in\R^d\,:\,(\VV_\rho\circ\TT_\varphi)\nu_b(x)>\lambda/2\big\}\big).
\end{split}
\end{equation}

Since $\VV_\rho\circ\TT_\varphi^{\HH^n_\Gamma}$ is bounded in $L^2(\HH^n_\Gamma)$ by Theorem \ref{theorem lip}, it is easy to show that $\VV_\rho\circ\TT_\varphi^{\mu}$ is bounded in $L^2(\mu)$, with a bound independent of $B$. Notice that $|g|\leq C\lambda$ by (\ref{4lema CZ 3}) and (\ref{4lema CZ 6}). Then, using (\ref{4lema CZ 5}),
\begin{equation}\label{4cota mesuresL1debil 3}
\begin{split}
\mu\big(\big\{x\in\R^d\,:\,(\VV_\rho\circ\TT_\varphi^\mu)g(x)>\lambda/2\big\}\big)
&\lesssim\frac{1}{\lambda^{2}}\int|(\VV_\rho\circ\TT_\varphi^\mu)g|^2\,d\mu
\lesssim\frac{1}{\lambda^{2}}\int|g|^2\,d\mu\\
&\lesssim\frac{1}{\lambda}\int|g|\,d\mu
\lesssim\frac{1}{\lambda}\bigg(|\nu|(\R^d\setminus\Omega)+\sum_j\int_{R_j}|b_j|\,d\mu\bigg)\\
&\lesssim\frac{1}{\lambda}\bigg(|\nu|(\R^d\setminus\Omega)+\sum_j|\nu|(Q_j)\bigg)\lesssim\frac{\|\nu\|}{\lambda}.
\end{split}
\end{equation}

Let $\wih\Omega:=\bigcup_j2Q_j$. By (\ref{4lema CZ 1}), we have
$\mu(\wih\Omega)\leq\sum_j\mu(2Q_j)\lesssim\lambda^{-1}\sum_j|\nu|(Q_j)\lesssim\lambda^{-1}\|\nu\|$. We are going to show now that
\begin{equation}\label{4cota mesuresL1debil 1}
\mu\big(\big\{x\in\R^d\setminus\wih\Omega\,:\,(\VV_\rho\circ\TT_\varphi)\nu_b(x)>\lambda/2\big\}\big)\leq\frac{C}{\lambda}\,\|\nu\|,
\end{equation}
and then (\ref{4cota mesuresL1debil}) is a direct consequence of (\ref{4cota mesuresL1debil 2}), (\ref{4cota mesuresL1debil 3}), (\ref{4cota mesuresL1debil 1}) and the estimate $\mu(\wih\Omega)\lesssim\lambda^{-1}\|\nu\|$.
Since $\VV_\rho\circ\TT_\varphi$ is sublinear,
\begin{equation}\label{4cota mesuresL1debil equa1}
\begin{split}
\mu\big(\big\{x\in\R^d\setminus\wih\Omega\,:\,(\VV_\rho&\circ\TT_\varphi)\nu_b(x)>\lambda/2\big\}\big)
\lesssim\frac{1}{\lambda}\sum_{j}\int_{\R^d\setminus\wih\Omega}(\VV_\rho\circ\TT_\varphi)\nu_b^j\,d\mu\\
&\leq\frac{1}{\lambda}\sum_{j}\int_{\R^d\setminus 2R_j}(\VV_\rho\circ\TT_\varphi)\nu_b^j\,d\mu
+\frac{1}{\lambda}\sum_{j}\int_{2R_j\setminus 2Q_j}(\VV_\rho\circ\TT_\varphi)\nu_b^j\,d\mu.
\end{split}
\end{equation}

We are going to estimate the two terms on the right of (\ref{4cota mesuresL1debil equa1}) separately. Let us start with the first one. Given $j$ and $x\in\supp\mu\setminus2R_j$, let $\{\epsilon_m\}_{m\in\Z}$ be a decreasing sequence of positive numbers (which depends on $j$ and $x$, i.e. $\epsilon_m\equiv\epsilon_m(j,x)$) such that
\begin{equation}\label{4cota mesuresL1debil equa3}
\begin{split}
(\VV_\rho\circ\TT_\varphi)\nu_b^j(x)&
\leq2\bigg(\sum_{m\in\Z}|(K\varphi_{\epsilon_{m+1}}^{\epsilon_m}*\nu_b^j)(x)|^\rho\bigg)^{1/\rho}.
\end{split}
\end{equation}
If we set $I_k:=[2^{-k-1},2^{-k})$, we can decompose $\Z=\SSS\cup\LL$, where
\begin{equation*}
\begin{split}
&\LL:=\{m\in\Z\,:\,\epsilon_m\in I_k,\,\epsilon_{m+1}\in I_i,\text{ for }i>k\},\\
&\SSS:=\bigcup_{k\in\Z}\SSS_k,\quad\SSS_k:=\{m\in\Z\,:\,\epsilon_{m},\epsilon_{m+1}\in I_k\}.
\end{split}
\end{equation*}

Let $z_j$ denote the center of $Q_j$ (and of $R_j$). Then, since $\nu_b^j(R_j)=0$ and $\supp\nu_b^j\subset R_j$,
\begin{equation}\label{4cota mesuresL1debil equa2}
\begin{split}
|(K\varphi_{\epsilon_{m+1}}^{\epsilon_m}*\nu_b^j)(x)|
&=\bigg|\int\varphi_{\epsilon_{m+1}}^{\epsilon_m}(x-y)K(x-y)\,d\nu_b^j(y)\bigg|\\
&\leq\int\big|\varphi_{\epsilon_{m+1}}^{\epsilon_m}(x-y)K(x-y)
-\varphi_{\epsilon_{m+1}}^{\epsilon_m}(x-z_j)K(x-z_j)\big|\,d|\nu_b^j|(y).
\end{split}
\end{equation}

If $m\in\LL$, it is easy to see that
$|\nabla(\varphi_{\epsilon_{m+1}}^{\epsilon_m}K)(t)|\leq|\nabla(\varphi_{\epsilon_{m+1}}K)(t)|+|\nabla(\varphi_{\epsilon_m}K)(t)|\lesssim|t|^{-n-1}$ for all $t\in\R^d\setminus\{0\}$. Moreover, since $x\in\R^d\setminus 2R_j$ and $\supp\nu_b^j\subset R_j$, there are finitely many $m\in\LL$ (which depends only on $n$ and $d$) such that  $(K\varphi_{\epsilon_{m+1}}^{\epsilon_m}*\nu_b^j)(x)\neq0$, and this number only depends on $n$ and $d$. 
On the other hand, if $m\in\SSS_k$,  it is not hard to show that 
$|\nabla(\varphi_{\epsilon_{m+1}}^{\epsilon_m}K)(t)|\lesssim2^{k}|\epsilon_m-\epsilon_{m+1}||t|^{-n-1}$. Actually, this follows from the fact that
$(\varphi_{\epsilon_{m+1}}^{\epsilon_m}K)(t)\neq0$ only if $|t|\approx 2^{-k}$ and the estimates
\begin{equation}\label{Short eq1}
\begin{split}
|\varphi_{\epsilon_{m+1}}^{\,\epsilon_{m}}(t)|&=\left|\varphi_\R\bigg(\frac{|t|}{\epsilon_{m+1}}\bigg)-\varphi_\R\bigg(\frac{|t|}{\epsilon_{m}}\bigg)\right|
\leq\|\varphi'_\R\|_{L^\infty(\R)}\left|\frac{|t|}{\epsilon_{m+1}}-\frac{|t|}{\epsilon_{m}}\right|\\
&=\|\varphi'_\R\|_{\infty}|t|\,\frac{\epsilon_{m}-\epsilon_{m+1}}{\epsilon_{m}\epsilon_{m+1}}\lesssim2^{k}|\epsilon_{m}-\epsilon_{m+1}|
\end{split}
\end{equation}
and
\begin{equation}\label{Short eq2}
\begin{split}
\big|\partial_{t^i}\big(\varphi_{\epsilon_{m+1}}^{\,\epsilon_m}(t)\big)\big|
&\leq\left|\varphi'_\R\left(\frac{|t|}{\epsilon_{m}}\right)\frac{1}{\epsilon_m}-
\varphi'_\R\left(\frac{|t|}{\epsilon_{m+1}}\right)\frac{1}{\epsilon_{m+1}}\right|\\
&\leq\left|\varphi'_\R\left(\frac{|t|}{\epsilon_{m}}\right)\right|\left|\frac{1}{\epsilon_{m}}-\frac{1}{\epsilon_{m+1}}\right|
+\left|\varphi'_\R\left(\frac{|t|}{\epsilon_{m}}\right)-\varphi'_\R\left(\frac{|t|}{\epsilon_{m+1}}\right)\right|\frac{1}{\epsilon_{m+1}}\\
&\leq\left(\|\varphi_\R'\|_{\infty}+\|\varphi''_\R\|_{\infty}\frac{|t|}{\epsilon_{m+1}}\right)\frac{\epsilon_{m}-\epsilon_{m+1}}{\epsilon_{m}\epsilon_{m+1}}\lesssim2^{k}(\epsilon_{m}-\epsilon_{m+1})|t|^{-1},
\end{split}
\end{equation}
where $1\leq i\leq d$ and $t^i$ denotes the $i$'th coordinate of $t\in\R^d$ (recall that $\epsilon_m\approx\epsilon_{m+1}\approx 2^{-k}$ for $m\in\SSS_k$ and we assumed $|t|\approx 2^{-k}$).
Similarly to the case $m\in\LL$, there are finitely many $k\in\Z$ such that $\supp\varphi_{2^{-k-1}}^{2^{-k}}(x-\cdot)\cap R_j\neq\emptyset$, and the number only depends on $n$ and $d$ (notice that 
$\supp\varphi_{\epsilon_{m+1}}^{\epsilon_m}(x-\cdot)\subset\supp\varphi_{2^{-k-1}}^{2^{-k}}(x-\cdot)$ for all $m\in\SSS_k$).

From these estimates and remarks, and (\ref{4cota mesuresL1debil equa3}), (\ref{4cota mesuresL1debil equa2}), we obtain
\begin{equation*}
\begin{split}
(\VV_\rho\circ\TT_\varphi)\nu_b^j(x)
&\lesssim
\sum_{k\in\Z}\,\sum_{m\in\SSS_k}|(K\varphi_{\epsilon_{m+1}}^{\epsilon_m}*\nu_b^j)(x)|+\sum_{m\in\LL}|(K\varphi_{\epsilon_{m+1}}^{\epsilon_m}*\nu_b^j)(x)|\\
&\lesssim
\sum_{{\begin{subarray}{c}k\in\Z:\,
\supp\varphi_{2^{-k-1}}^{2^{-k}}(x-\cdot)\cap R_j\neq\emptyset
\end{subarray}}}\,\sum_{m\in\SSS_k}
2^{k}|\epsilon_m-\epsilon_{m+1}||x-z_j|^{-n-1}\ell(R_j)\|\nu_b^j\|\\
&\quad+\sum_{\begin{subarray}{c}m\in\LL:\,
\supp\varphi_{\epsilon_{m+1}}^{\epsilon_m}(x-\cdot)\cap R_j\neq\emptyset
\end{subarray}}|x-z_j|^{-n-1}\ell(R_j)\|\nu_b^j\|
\lesssim|x-z_j|^{-n-1}\ell(R_j)\|\nu_b^j\|
\end{split}
\end{equation*}
for all $j$ and $x\in\supp\mu\setminus 2R_j$. Therefore, using that $\mu$ has $n$-dimensional growth, that $\|\nu_b^j\|\lesssim|\nu|(Q_j)$, and that the $Q_j$'s are semidisjoint,
\begin{equation}\label{4cota mesuresL1debil equa5}
\begin{split}
\sum_{j}\int_{\R^d\setminus 2R_j}(\VV_\rho\circ\TT_\varphi)\nu_b^j\,d\mu\lesssim
\sum_{j}\ell(R_j)\|\nu_b^j\|\int_{\R^d\setminus 2R_j}|x-z_j|^{-n-1}\,d\mu\lesssim
\sum_{j}\|\nu_b^j\|\lesssim\|\nu\|.
\end{split}
\end{equation}

Let us now estimate the second term on the right hand side of (\ref{4cota mesuresL1debil equa1}).
As above, given $j$ and $x\in 2R_j\setminus2Q_j$,  let $\{\epsilon_m\}_{m\in\Z}$ be a decreasing sequence of positive numbers  such that
\begin{equation*}
\begin{split}
(\VV_\rho\circ\TT_\varphi)(w_j\nu)(x)&
\leq2\bigg(\sum_{m\in\Z}|(K\varphi_{\epsilon_{m+1}}^{\epsilon_m}*(w_j\nu))(x)|^\rho\bigg)^{1/\rho},
\end{split}
\end{equation*}
where $w_j=\chi_{Q_j}\big(\sum_k\chi_{Q_k}\big)^{-1}$.
Since $\rho>2$, $\VV_\rho\circ\TT_\varphi$ is sublinear, and since $\nu_b^j=w_j\nu-b_j\mu$, for $x\in2R_j\setminus2Q_j$ we have
\begin{equation*}
\begin{split}
(\VV_\rho\circ\TT_\varphi)\nu_b^j(x)&\leq(\VV_\rho\circ\TT_\varphi)(w_j\nu)(x)+(\VV_\rho\circ\TT)(b_j\mu)(x)\\
&\leq2\sum_{m\in\Z}\big|(K\varphi_{\epsilon_{m+1}}^{\epsilon_m}*(w_j\nu))(x)\big|+(\VV_\rho\circ\TT_\varphi^\mu)b_j(x)\\
&\lesssim|\nu|(Q_j)|x-z_j|^{-n}+(\VV_\rho\circ\TT_\varphi^\mu)b_j(x).
\end{split}
\end{equation*}
Since $\VV_\rho\circ\TT_\varphi^\mu$ is bounded in $L^2(\mu)$, using the estimate above and Cauchy-Schwarz we get
\begin{equation*}
\begin{split}
\sum_j\int_{2R_j\setminus 2Q_j}(\VV_\rho\circ\TT_\varphi)\nu_b^j\,d\mu
&\lesssim\sum_j\int_{2R_j\setminus2Q_j}\frac{|\nu|(Q_j)}{|x-z_j|^{n}}\,d\mu(x)
+\sum_{j}\int_{2R_j\setminus 2Q_j}(\VV_\rho\circ\TT_\varphi^\mu)b_j\,d\mu\\
&\lesssim\sum_j|\nu|(Q_j)\frac{\mu(2R_j)}{\ell(Q_j)^{n}}
+\sum_{j}\|(\VV_\rho\circ\TT_\varphi^\mu)b_j\|_{L^2(\mu)}\mu(2R_j)^{1/2}\\
&\lesssim\sum_j|\nu|(Q_j)
+\sum_{j}\|b_j\|_{L^\infty(\mu)}\mu(R_j)
\lesssim\sum_j|\nu|(Q_j)\lesssim\|\nu\|.
\end{split}
\end{equation*}
Together with (\ref{4cota mesuresL1debil equa5}) and (\ref{4cota mesuresL1debil equa1}), this proves (\ref{4cota mesuresL1debil 1}), and Theorem \ref{4unif rectif teo3} follows.

\section{If $\mu$ is a uniformly $n$-rectifiable measure, then\\
$\VV_\rho\circ\TT_\varphi^\mu:\,L^p(\mu)\to L^p(\mu)$ is a bounded operator for $1<p<\infty$}\label{4sec acotacio unif rectif}

The purpose of this section consists in proving the following theorem and the subsequent corollary. 

\begin{teo}\label{4unif rectif teo}
Let $\mu$ be an $n$-dimensional AD regular Borel measure in $\R^d$ and let $\rho>2$. Assume that there exist constants $C_0$ and $C_1$ such that, for each ball $B$ centered on $\supp\mu$, there is a set $F=F_B$ such that:
\begin{itemize}
\item[$(a)$]$\mu(F\cap B)\geq C_0\mu(B),$
\item[$(b)$] $\VV_\rho\circ\TT_\varphi$ is bounded from $M(\R^d)$ to $L^{1,\infty}(\HH^n_F)$ with constant bounded by $C_1$.
\end{itemize}
Then $\VV_\rho\circ\TT_\varphi$ is bounded from $M(\R^d)$ to $L^{1,\infty}(\mu)$, and $\VV_\rho\circ\TT_\varphi^{\mu}$ is a bounded operator in $L^p(\mu)$ for all $1<p<\infty$.
\end{teo}

\begin{coro}\label{unif rectif implies var smooth}
If  $\mu$ is an $n$-dimensional AD regular uniformly $n$-rectifiable measure, then $\VV_\rho\circ\TT_\varphi^{\mu}$ is a bounded operator in $L^p(\mu)$ for all $1<p<\infty$ and $\rho>2$. Moreover, the operator $\VV_\rho\circ\TT_\varphi$ is bounded from $M(\R^d)$ to $L^{1,\infty}(\mu)$, so $\VV_\rho\circ\TT_\varphi^{\mu}$ is also of weak type $(1,1)$.
\end{coro}

\begin{proof}[{\bf {\em Proof.}}]
Recall from \cite[Definition 1.26]{DS2} that a Borel measure $\nu$ in $\R^d$ has $BPLG$ ({\em big pieces of Lipschitz graphs}) if $\nu$ is $n$-dimensional AD regular and if there exist constants $C_1>0$ and $\theta>0$ such that, for any $x\in \supp\nu$ and $0<r<\diam(\supp\nu)$, there is (a rotation and translation of) an $n$-dimensional Lipschitz graph $\Gamma$ with constant less than $C_1$ such that $\nu(\Gamma\cap B(x,r))\geq\theta r^n$. Thus, if $\nu$ has $BPLG$, the assumption $(a)$ of Theorem \ref{4unif rectif teo} is satisfied for $\nu$ by taking $F=\Gamma$, while Theorem \ref{4unif rectif teo3} implies that the assumption $(b)$ holds with a uniform constant. Therefore, from Theorem \ref{4unif rectif teo} we deduce that, if $\nu$ has $BPLG$ and $\rho>2$, then $\VV_\rho\circ\TT_\varphi$ is bounded from $M(\R^d)$ to $L^{1,\infty}(\nu)$.

Similarly, a measure $\nu$ has $(BP)^2LG$ ({\em big pieces of big pieces of Lipschitz graphs}) if there exist constants $C_g$, $\theta$, and $0<\alpha\leq1$ so that, if $B$ is any ball centered on $\supp\nu$, then there is an $n$-dimensional AD regular set $F\subset\R^d$ (with constant bounded by $C_g$) such that $\nu(F\cap B)\geq\alpha\nu(B)$ and such that $\HH^n_ F$ has $BPLG$ with uniform constants. So $\VV_\rho\circ\TT_\varphi$ is a bounded operator from $M(\R^d)$ to $L^{1,\infty}(\HH^n_F)$, by the comments above. Hence, we can apply once again Theorem \ref{4unif rectif teo} to $\nu$ (now $(b)$ is satisfied for the big pieces $F$ of $\nu$), and we deduce that, for any measure $\nu$ which has $(BP)^2LG$, $\VV_\rho\circ\TT_\varphi$ is bounded from $M(\R^d)$ to $L^{1,\infty}(\nu)$. 
Similar arguments yield that $\VV_\rho\circ\TT_\varphi^{\nu}$ is a bounded operator in $L^p(\nu)$ for all $1<p<\infty$.

Finally, from \cite[page 22]{DS2} and the remark given in \cite[page 16]{DS2}, we know that if $\mu$ is $n$-dimensional AD regular, then being  uniformly $n$-rectifiable is equivalent to having $(BP)^2LG$. Therefore, the corollary is proved by applying the comments above to $\nu=\mu$.
\end{proof}

Since the arguments for proving Theorem \ref{4unif rectif teo} are more or less standard in Calder\'on-Zygmund theory, for the sake of shortness we will only sketch its proof (see \cite[Chapter 2]{Tolsa-llibre} or \cite[Proposition 1.28 of Part I]{DS2} for a similar argument).

\begin{proof}[{\bf{\em Sketch of the proof of} Theorem \ref{4unif rectif teo}}]
The proof follows by the so-called {\em good $\lambda$ inequality} method.
Fix $\rho>2$ and let $M^\mu$ denote the Hardy-Littlewood maximal operator 
$$M^\mu \nu(x):=\sup_{r>0}\frac{|\nu|(B(x,r))}{\mu(B(x,r))},\quad\text{ for }\nu\in M(\R^d)\text{ and } x\in\supp\mu.$$
{\em The good $\lambda$ inequality}: one shows that there exists some absolute constant $\eta>0$ such that for all $\epsilon>0$ there exists $\delta:=\delta(\epsilon)>0$ such that
\begin{equation}\label{4eqgli0}
\begin{split}
\mu\big(\big\{x\in\R^d\,:\,(\VV_\rho\circ\TT_\varphi)\nu(x)&>(1+\epsilon)\lambda,\;M^\mu \nu(x)\leq\delta\lambda\big\}\big)\\
&\leq(1-\eta)\mu\big(\big\{x\in\R^d\,:\,(\VV_\rho\circ\TT_\varphi)\nu(x)>\lambda\big\}\big)
\end{split}
\end{equation}
for all $\lambda>0$ and $\nu\in M(\R^d)$.
It is easy to check that this implies that 
$\VV_\rho\circ\TT_\varphi$ is bounded from $M(\R^d)$ to $L^{1,\infty}(\mu)$, and that $\VV_\rho\circ\TT_\varphi^{\mu}$ is bounded in $L^p(\mu)$ for all $1<p<\infty$, by standard arguments (recall that $M^\mu$ is bounded in these spaces).

The proof of (\ref{4eqgli0}) is quite standard. The interested reader may look at \cite[Theorem 5.2.1]{Mas-thesis} for the detailed proof, or to \cite[Chapter 2]{Tolsa-llibre} for similar arguments. The only point that we should mention is that, in order to pursue the good $\lambda$ inequality method, one needs the following estimate: let $\nu\in M(\R^d)$, consider a ball $B\subset\R^d$ and take $x,z\in B$. Then,
\begin{equation}\label{4eqgli01}
\big|(\VV_\rho\circ\TT_\varphi)(\chi_{\R^d\setminus 2B}\nu)(x)-(\VV_\rho\circ\TT_\varphi)(\chi_{\R^d\setminus 2B}\nu)(z)\big|\lesssim M^\mu\nu(x).
\end{equation}
We finish the sketch of the proof of Theorem \ref{4unif rectif teo} by showing (\ref{4eqgli01}). Since $x,z\in B$ and $\VV_\rho\circ\TT_\varphi$ is sublinear and positive, by the mean value theorem,
\begin{equation}\label{4eq unif rectif1}
\begin{split}
\big|(\VV_\rho\circ\TT_\varphi)(&\chi_{\R^d\setminus 2B}\nu)(x)-(\VV_\rho\circ\TT_\varphi)(\chi_{\R^d\setminus 2B}\nu)(z)\big|\\
&\leq\sup_{\epsilon_m}\bigg(\sum_{m\in\Z}
|(K{\varphi_{\epsilon_{m+1}}^{\epsilon_m}}*(\chi_{\R^d\setminus 2B}\nu))(x)-(K{\varphi_{\epsilon_{m+1}}^{\epsilon_m}}*(\chi_{\R^d\setminus 2B}\nu))(z)|^\rho\bigg)^{1/\rho}\\
&\leq\sup_{\epsilon_m}\bigg(\sum_{m\in\Z}
\bigg(\int_{B_m(x,z)}|\nabla(\varphi_{\epsilon_{m+1}}^{\epsilon_m}K)(u_{x,z}(y)-y)||x-z|\,d|\nu|(y)\bigg)^\rho\bigg)^{1/\rho},
\end{split}
\end{equation}
where $B_m(x,z):=({\R^d\setminus 2B})\cap(\supp\varphi_{\epsilon_{m+1}}^{\epsilon_m}(x-\cdot)\cup\supp\varphi_{\epsilon_{m+1}}^{\epsilon_m}(z-\cdot))$ and $u_{x,z}(y)$ is some point lying on the segment joining $x$ and $z$.  For each $x$ and $z$, let $\epsilon_m\equiv\epsilon_m(x,z)$ be a sequence that realizes the supremum in the right hand side of (\ref{4eq unif rectif1}).
Given $\epsilon_m>0$, let $j(\epsilon_m)$ denote the integer such that $\epsilon_m\in[2^{-j(\epsilon_m)-1},2^{-j(\epsilon_m)})$. For $j\in\Z$ set $I_j:=[2^{-j-1},2^{-j})$. As usual, we decompose $\Z=\SSS\cup\LL$, where
\begin{equation*}
\begin{split}
&\SSS:=\bigcup_{j\in\Z}\SSS_j,\quad\SSS_j:=\{m\in\Z\,:\,\epsilon_m,\epsilon_{m+1}\in I_j\},\\
&\LL:=\{m\in\Z\,:\,\epsilon_m\in I_i,\,\epsilon_{m+1}\in I_j\text{ for }i<j\}.
\end{split}
\end{equation*}

Notice that if $2^{-j+2}<r(B)$, where $r(B)$ denotes the radius of $B$, then $B_m(x.z)=\emptyset$ for all $m\in\SSS_j$. Therefore, we can assume that $j\leq\log_2(4/r(B))$.
If $m\in\SSS_j$, then $B_m(x,z)\subset B(x,2^{-j+3})$, and for $t\in\supp(\varphi_{\epsilon_{m+1}}^{\epsilon_m}K)$ we have that 
$|\nabla(\varphi_{\epsilon_{m+1}}^{\epsilon_m}K)(t)|\lesssim2^{j(n+2)}|\epsilon_m-\epsilon_{m+1}|$ (see (\ref{Short eq1}) and (\ref{Short eq2})). If $m\in\LL$, we easily have
 $|\nabla(\varphi_{\epsilon_{m+1}}^{\epsilon_m}K)(t)|\lesssim|t|^{-n-1}$. 
Therefore, using (\ref{4eq unif rectif1}), that $\rho>2$, that the sets $B_m(x,z)$ have bounded overlap for $m\in\LL$, and that $|x-z|\lesssim r(B)$, we get
\begin{equation*}
\begin{split}
\big|(\VV_\rho\circ\TT_\varphi)(&\chi_{\R^d\setminus 2B}\nu)(x)-(\VV_\rho\circ\TT_\varphi)(\chi_{\R^d\setminus 2B}\nu)(z)\big|\\
&\lesssim\sum_{j\leq\log_2(4/r(B))}\,\sum_{m\in\SSS_j}
|x-z|2^{j(n+2)}|\epsilon_m-\epsilon_{m+1}|\int_{B(x,2^{-j+3})}\,d|\nu|(y)\\
&\quad+|x-z|\sum_{m\in\LL}\int_{B_m(x,z)}|x-y|^{-n-1}\,d|\nu|(y)\\
&\lesssim\sum_{j\leq\log_2(4/r(B))}
r(B)2^{j(n+1)}\int_{B(x,2^{-j+3})}\,d|\nu|(y)
+r(B)\int_{\R^d\setminus 2B}\frac{d|\nu|(y)}{|x-y|^{n+1}}\\
&\lesssim\sum_{j\leq\log_2(4/r(B))}
\frac{r(B)2^{j}}{\mu(B(x,2^{-j+3}))}\int_{B(x,2^{-j+3})}\,d|\nu|(y)\\
&\quad+r(B)\sum_{k\geq1}\int_{2^{k+2}r(B)\geq|x-y|\geq2^{k-1}r(B)}\frac{d|\nu|(y)}{|x-y|^{n+1}}\\
&\lesssim M^\mu \nu(x)
+\sum_{k\geq1}\frac{2^{-k}}{\mu(B(x,2^{k+2}r(B_i)))}\int_{B(x,2^{k+2}r(B))}\,d|\nu|(y)
\lesssim M^\mu\nu(x).
\end{split}
\end{equation*}
\end{proof}
\begin{remarko}{\em
Notice that, to prove (\ref{4eqgli01}), it is a key fact that we are considering smooth truncations (given by $\varphi_\R$) in the definition of $\TT_\varphi$. These computations are no longer valid if one replaces $\TT_\varphi$ by $\TT$.}
\end{remarko}

\section{If $\mu$ is a uniformly $n$-rectifiable measure, then\\
$\VV_\rho\circ\TT^\mu:\,L^2(\mu)\to L^2(\mu)$ is a bounded operator}\label{5s var no suau}

This section is devoted to the proof of the following result.
\begin{teo}\label{5teo var no suau acotada L2}
Let $\rho>2$ and let $\mu$ be an $n$-dimensional AD regular Borel measure on $\R^d$. If $\mu$ is uniformly $n$-rectifiable, then $\VV_\rho\circ\TT^\mu$ is a bounded operator in $L^2(\mu)$.
\end{teo}

\subsection{Short and long variation}
Given $j\in\Z$, set $I_j:=[2^{-j-1},2^{-j})$. Then, using the triangle inequality, we can split the variation operator into the so-called short variation and long variation operators, i.e., 
$(\VV_\rho\circ\TT^\mu)f(x)\leq (\VV^\SSS_\rho\circ\TT^\mu)f(x)+(\VV^\LL_\rho\circ\TT^\mu)f(x),$ where 
\begin{equation}\label{5 short and long variation}
\begin{split}
&(\VV^\SSS_\rho\circ\TT^\mu)f(x):=\sup_{\{\epsilon_m\}}\bigg(\sum_{j\in\Z}\,\sum_{\epsilon_m,\epsilon_{m+1}\in I_j}|(K\chi_{\epsilon_{m+1}}^{\epsilon_m}*(f\mu))(x)|^{\rho}\bigg)^{1/\rho},\\
&(\VV^\LL_\rho\circ\TT^\mu)f(x):=\sup_{\{\epsilon_m\}}\bigg(\sum_{\begin{subarray}{c}m\in\Z:\,\epsilon_m\in I_j,\,\epsilon_{m+1}\in I_k\\ \text{ for some }j<k\end{subarray}}|(K\chi_{\epsilon_{m+1}}^{\epsilon_m}*(f\mu))(x)|^{\rho}\bigg)^{1/\rho},
\end{split}
\end{equation}
and, in both cases, the pointwise  supremum is taken over all the sequences of positive numbers $\{\epsilon_m\}_{m\in\Z}$ decreasing to zero. To prove Theorem \ref{5teo var no suau acotada L2} we will show that both the short and long variation operators are bounded in $L^2(\mu)$.

\subsection{$L^2(\mu)$ boundedness of $\VV^\LL_\rho\circ\TT^\mu$}\label{5sslong}
The $L^2(\mu)$-norm of the long variation operator $\VV^\LL_\rho\circ\TT^\mu$ can be handled by comparing it with its smoothed version $\VV_\rho\circ\TT_\varphi^\mu$, using Corollary \ref{unif rectif implies var smooth}, and estimating the error terms by the short variation operator. 
\begin{lema}\label{lqlq} We have
$\|(\VV^{\LL}_\rho\circ\TT^\mu)f\|_{L^2(\mu)}
\lesssim\|(\VV^\SSS_\rho\circ\TT^\mu)f\|_{L^2(\mu)}+\|f\|_{L^2(\mu)}.$
\end{lema}

\begin{proof}[{\bf {\em Proof.}}] We decompose
\begin{equation}\label{5 short and long variation1}
\begin{split}
\big((\VV^\LL_\rho&\circ\TT^\mu)f(x)\big)^{\rho}=\sup_{\{\epsilon_m\}}\sum_{\begin{subarray}{c}m\in\Z:\,\epsilon_m\in I_j,\,\epsilon_{m+1}\in I_k\\ \text{ for some }j<k\end{subarray}}|(K\chi_{\epsilon_{m+1}}^{\epsilon_m}*(f\mu))(x)|^{\rho}\\
&\lesssim\sup_{\{\epsilon_m\}}\sum_{\begin{subarray}{c}m\in\Z:\\\epsilon_m\in I_j,\,\epsilon_{m+1}\in I_k\\ \text{ for some }j<k\end{subarray}}\Big(|(K(\chi_{\epsilon_{m+1}}^{\epsilon_m}-\varphi_{\epsilon_{m+1}}^{\epsilon_m})*(f\mu))(x)|^{\rho}+
|(K\varphi_{\epsilon_{m+1}}^{\epsilon_m}*(f\mu))(x)|^{\rho}\Big)\\
&\lesssim\sup_{\{\epsilon_m\}}\sum_{\begin{subarray}{c}m\in\Z:\,\epsilon_m\in I_j,\,\epsilon_{m+1}\in I_k\\ \text{ for some }j<k\end{subarray}}|(K(\chi_{\epsilon_{m+1}}^{\epsilon_m}-\varphi_{\epsilon_{m+1}}^{\epsilon_m})*(f\mu))(x)|^{\rho}+
\big((\VV_\rho\circ\TT_\varphi^\mu)f(x)\big)^\rho.
\end{split}
\end{equation}
For simplicity, we denote by $\big((\VV^{\LL}_\rho\circ\TT_{\chi-\varphi}^\mu)f(x)\big)^\rho$ the first term on the right hand side of (\ref{5 short and long variation1}). Notice that, given $\epsilon,\delta>0$, we have 
$\chi_\epsilon^\delta-\varphi_\epsilon^\delta=(\chi_\epsilon-\varphi_\epsilon)
-(\chi_\delta-\varphi_\delta)$. Recall that, in the definition of $\varphi_\R$ in Definition \ref{4defi varphi}, we have taken $\chi_{[4,\infty)}\leq\varphi_\R\leq\chi_{[1/4,\infty)}$. Hence, given $t\geq0$,
$$\chi_{\R}(t)-\varphi_{\R}(t)=\chi_{[1,\infty)}(t)-\int_{1/4}^4\varphi'_\R(s)\chi_{[s,\infty)}(t)\,ds=
\int_{1/4}^4\varphi'_\R(s)\big(\chi_{[1,\infty)}(t)-\chi_{[s,\infty)}(t)\big)\,ds$$
(that is, $\chi_{\R}-\varphi_{\R}$ is a convex combination of $\chi_{[1,\infty)}-\chi_{[s,\infty)}$ for $1/4\leq s\leq 4$), and thus, by Fubini's theorem, 
\begin{equation*}
\begin{split}
\big(K&(\chi_\epsilon-\varphi_\epsilon)*(f\mu)\big)(x)
=\int\big(\chi_\R(|x-y|^2/\epsilon^2)-\varphi_\R(|x-y|^2/\epsilon^2)\big)K(x-y)f(y)\,d\mu(y)\\
&=\int_{1/4}^4\varphi'_\R(s)\int\Big(\chi_{[1,\infty)}(|x-y|^2/\epsilon^2)-\chi_{[s,\infty)}(|x-y|^2/\epsilon^2)\Big)K(x-y)f(y)\,d\mu(y)\,ds\\
&=\int_{1/4}^4\varphi'_\R(s)\int\chi_{\epsilon}^{\epsilon\sqrt{s}}(x-y)K(x-y)f(y)\,d\mu(y)\,ds
=\int_{1/4}^4\varphi'_\R(s)\big((K\chi_{\epsilon}^{\epsilon\sqrt{s}}*(f\mu))(x)\big)\,ds.
\end{split}
\end{equation*}
Therefore, by the triangle inequality and  Minkowski's integral inequality, we get
\begin{equation*}
\begin{split}
\|(\VV^{\LL}_\rho\circ\TT_{\chi-\varphi}^\mu)f&\|_{L^2(\mu)}
\leq2\bigg\|\sup_{\{\epsilon_m\in I_m:\,m\in\Z\}}\bigg(\sum_{m\in\Z}
|(K(\chi_{\epsilon_{m}}-\varphi_{\epsilon_{m}})*(f\mu))(x)|^{\rho}\bigg)^{1/\rho}\bigg\|_{L^2(\mu)}\\
&\leq2\int_{1/4}^4\varphi'_\R(s)\bigg\|\sup_{\{\epsilon_m\in I_m:\,m\in\Z\}}\bigg(\sum_{m\in\Z}|(K\chi_{\epsilon_m}^{\epsilon_m\sqrt{s}}*(f\mu))(x)|^\rho\bigg)^{1/\rho}\bigg\|_{L^2(\mu)}\,ds.
\end{split}
\end{equation*}
One can easily verify that  $\sup_{\{\epsilon_m\in I_m:\,m\in\Z\}}\big(\sum_{m\in\Z}|(K\chi_{\epsilon_m}^{\epsilon_m\sqrt{s}}*(f\mu))(x)|^\rho\big)^{1/\rho}\lesssim(\VV^\SSS_\rho\circ\TT^\mu)f(x)$ for all $s\in[1/4,4]$ with uniform bounds. Hence
\begin{equation}\label{5 short and long variation1a}
\begin{split}
\|(\VV^{\LL}_\rho\circ\TT_{\chi-\varphi}^\mu)f\|_{L^2(\mu)}
\lesssim\int_{1/4}^4\varphi'_\R(s)\|(\VV^\SSS_\rho\circ\TT^\mu)f\|_{L^2(\mu)}\,ds\lesssim\|(\VV^\SSS_\rho\circ\TT^\mu)f\|_{L^2(\mu)}.
\end{split}
\end{equation}

Finally, using (\ref{5 short and long variation1}), (\ref{5 short and long variation1a}), and Corollary \ref{unif rectif implies var smooth},
\begin{equation*}
\begin{split}
\|(\VV^{\LL}_\rho\circ\TT^\mu)f\|_{L^2(\mu)}
&\lesssim\|(\VV^{\LL}_\rho\circ\TT_{\chi-\varphi}^\mu)f\|_{L^2(\mu)}
+\|(\VV_\rho\circ\TT_{\varphi}^\mu)f\|_{L^2(\mu)}\\
&\lesssim\|(\VV^\SSS_\rho\circ\TT^\mu)f\|_{L^2(\mu)}+\|f\|_{L^2(\mu)}.
\end{split}
\end{equation*}
\end{proof}

Thus, to prove Theorem \ref{5teo var no suau acotada L2}, it only remains to show the 
$L^2(\mu)$ boundedness of $\VV^\SSS_\rho\circ\TT^\mu$.
 
\subsection{$L^2(\mu)$ boundedness of $\VV^\SSS_\rho\circ\TT^\mu$}\label{5ssshort}
Given $f\in L^2(\mu)$ and $x\in\supp\mu$, let $\{\epsilon_m\}_{m\in\Z}$ be a decreasing sequence of positive numbers (depending on $x$) such that 
\begin{equation*}
\big((\VV^\SSS_2\circ\TT^\mu)f(x)\big)^2\leq2\sum_{j\in\Z}\,\sum_{\epsilon_m,\epsilon_{m+1}\in I_j}|(K\chi_{\epsilon_{m+1}}^{\epsilon_m}*(f\mu))(x)|^{2}.
\end{equation*}
Given $D\in\DD_j$ and $x\in D$, we set $\SSS_D(x):=\{m\in\Z:\,\epsilon_{m},\epsilon_{m+1}\in I_j\}$. Since $\rho\geq2$, we have
\begin{equation*}
\begin{split}
\|(\VV^\SSS_\rho\circ\TT^\mu)f\|_{L^2(\mu)}^2&\leq\|(\VV^\SSS_2\circ\TT^\mu)f\|_{L^2(\mu)}^2
\lesssim\int\sum_{j\in\Z}\,\sum_{\epsilon_m,\epsilon_{m+1}\in I_j}|(K\chi_{\epsilon_{m+1}}^{\epsilon_m}*(f\mu))(x)|^{2}\,d\mu(x)\\
&=\sum_{D\in\DD}\int_D\sum_{m\in\SSS_D(x)}|(K\chi_{\epsilon_{m+1}}^{\epsilon_m}*(f\mu))(x)|^{2}\,d\mu(x).
\end{split}
\end{equation*}

Let $\eta$ and $\theta$ be two positive numbers that will be fixed below (see the proofs of Claims \ref{5 claim1} and \ref{5 claim2}). Consider a corona decomposition of $\mu$ with parameters $\eta$ and $\theta$ as in Subsection \ref{5ss corona decomposition}. Then, we can decompose $\DD=\BB\cup(\bigcup_{S\in\Trees}S)$, so that
\begin{equation}\label{5 var eq1}
\begin{split}
\|(\VV^\SSS_\rho\circ\TT^\mu)f\|_{L^2(\mu)}^2
&\lesssim\sum_{D\in\BB}\int_D\sum_{m\in\SSS_D(x)}|(K\chi_{\epsilon_{m+1}}^{\epsilon_m}*(f\mu))(x)|^{2}\,d\mu(x)\\
&\quad+\sum_{S\in\Trees}\,\sum_{D\in S}\int_D\sum_{m\in\SSS_D(x)}|(K\chi_{\epsilon_{m+1}}^{\epsilon_m}*(f\mu))(x)|^{2}\,d\mu(x).
\end{split}
\end{equation}

Since the $\mu$-cubes in $\BB$ satisfy a Carleson packing condition, we can use Carleson's embedding theorem to estimate the sum on the right hand side of (\ref{5 var eq1}) over the $\mu$-cubes in $\BB$. More precisely, if we set $m_D^\mu f:=\mu(D)^{-1}\int_{D}f\,d\mu$ for $D\in\DD$, we have
\begin{equation}\label{5 var eq2}
\begin{split}
\sum_{D\in\BB}\int_D&\sum_{m\in\SSS_D(x)}|(K\chi_{\epsilon_{m+1}}^{\epsilon_m}*(f\mu))(x)|^{2}\,d\mu(x)\\
&\leq\sum_{D\in\BB}\int_D\sum_{m\in\SSS_D(x)}\bigg(\int_{\epsilon_{m+1}\leq|x-y|\leq\epsilon_m}|K(x-y)||f(y)|\,d\mu(y)\bigg)^{2}\,d\mu(x)\\
&\lesssim\sum_{D\in\BB}\int_D\bigg(\frac{1}{\ell(D)^n}\int_{5D}|f|\,d\mu\bigg)^{2}\,d\mu
\approx\sum_{D\in\BB}\big(m_{5D}^\mu|f|\big)^{2}\mu(D)\lesssim\|f\|^2_{L^2(\mu)}.
\end{split}
\end{equation}

Now we are going to estimate now the second term on the right hand side of (\ref{5 var eq1}), that is the sum  over the $\mu$-cubes in $S$, for all $S\in\Trees$. To this end, we need to introduce some notation. Given $R\in\DD_j$ for some $j\in\Z$, let $P(R)$ denote the $\mu$-cube in $\DD_{j-1}$ which contains $R$ (the {\em parent} of $R$), and set
\begin{equation}\label{5 pares fills veins}
\begin{split}
& \Child( R):=\{Q\in\DD_{j+1}:\,Q\subset R\},\\
&V(R):=\{Q\in\DD_j:\, Q\cap B(y,\ell(R))\neq\emptyset\text{ for some }y\in R\}
\end{split}
\end{equation}
($\Child(R)$ are the {\em children} of $R$, and $V(R)$ stands for the {\em vicinity} of $R$). Notice that $P(R)$ is a $\mu$-cube but $ \Child( R)$ and $V(R)$ are collections of $\mu$-cubes. It is not hard to show that the number of $\mu$-cubes in $ \Child( R)$ and $V(R)$ is bounded by some constant depending only on $n$ and the AD regularity constant of $\mu$. If $R\in S$ for some $S\in\Trees$, we denote by $\Tre(R)$ the set of $\mu$-cubes $Q\in S$ such that $Q\subset R$ (the {\em tree} of $R$). Otherwise, i.e., if $R\in\BB$, we set $\Tre(R):=\emptyset$. Finally, if $\Tre(R)\neq\emptyset$, let $\Stop(R)$ denote the set of $\mu$-cubes $Q\in\BB\cup(\GG\setminus \Tre(R))$ such that $Q\subset R$ and $P(Q)\in \Tre(R)$ (the {\em stopping} $\mu$-cubes relative to $R$), so actually $Q\subsetneq R$.  On the other hand, if $R\in\BB$, we set $\Stop(R):=\{R\}$.

Fix $S\in\Trees$, $D\in S$, and $x\in D$. To deal with the second term on the right hand side of (\ref{5 var eq1}), we have to estimate the sum 
$\sum_{m\in\SSS_D(x)}|(K\chi_{\epsilon_{m+1}}^{\epsilon_m}*(f\mu))(x)|^2$. By the definition of $\SSS_D(x)$, we have 
\begin{equation}\label{5 haar decomposition2}
\sum_{m\in\SSS_D(x)}|(K\chi_{\epsilon_{m+1}}^{\epsilon_m}*(f\mu))(x)|^2=
\sum_{m\in\SSS_D(x)}|(K\chi_{\epsilon_{m+1}}^{\epsilon_m}*(\chi_{\wit D}f\mu))(x)|^2,
\end{equation}
where $\wit D:=\bigcup_{R\in V(D)}R$. Since this union of $\mu$-cubes is disjoint, we can decompose the function $\chi_{\wit D}f$ using a Haar basis adapted to $\DD$ in the following manner:
\begin{equation}\label{5 haar decomposition}
\chi_{\wit D}f=\sum_{R\in V(D)}\bigg((m_R^\mu f)\chi_R+\sum_{Q\in \Tre(R)}\Delta_Q f
+\sum_{Q\in \Stop(R)}\wit\Delta_Q f\bigg),
\end{equation}
where we have set 
\begin{equation*}
\begin{split}
\Delta_Q f:=\sum_{U\in  \Child( Q)}\chi_U(m_U^\mu f-m_Q^\mu f),\!\!\quad\text{and}\quad\!\!
\wit\Delta_Q f:=\sum_{U\in  \Child( Q)}\chi_U(f-m_Q^\mu f)=\chi_Q(f-m_Q^\mu f).
\end{split}
\end{equation*}
Using (\ref{5 haar decomposition}), we split the left hand side of (\ref{5 haar decomposition2}) as follows:
\begin{equation}\label{eqmain}
\begin{split}
\sum_{m\in\SSS_D(x)}|(K\chi_{\epsilon_{m+1}}^{\epsilon_m}*(f\mu))(x)|^2
&\lesssim\sum_{m\in\SSS_D(x)}\bigg|\sum_{R\in V(D)}
(K\chi_{\epsilon_{m+1}}^{\epsilon_m}*((m_R^\mu f)\chi_R\mu))(x)\bigg|^2\\
&+\sum_{m\in\SSS_D(x)}\bigg|\sum_{R\in V(D)}\sum_{Q\in \Tre(R)}(K\chi_{\epsilon_{m+1}}^{\epsilon_m}*(\Delta_Q f\mu))(x)\bigg|^2\\
&+\sum_{m\in\SSS_D(x)}\bigg|\sum_{R\in V(D)}\sum_{Q\in \Stop(R)}(K\chi_{\epsilon_{m+1}}^{\epsilon_m}*(\wit\Delta_Q f\mu))(x)\bigg|^2.
\end{split}
\end{equation}
In the following subsections, we will estimate each part separately.

\subsubsection{{\bf Estimate of 
$\sum_{m\in\SSS_D(x)}\big|\sum_{R\in V(D)}\sum_{Q\in \Tre(R)}(K\chi_{\epsilon_{m+1}}^{\epsilon_m}*(\Delta_Q f\mu))(x)\big|^2$ from (\ref{eqmain})}}\label{5 ss5321}

\begin{lema}\label{5 var eq8}
Under the notation above, we have
\begin{equation*}
\sum_{S\in\Trees}\,\sum_{D\in S}\int_D\sum_{m\in\SSS_D(x)}\bigg|\sum_{R\in V(D)}\sum_{Q\in \Tre(R)}(K\chi_{\epsilon_{m+1}}^{\epsilon_m}*(\Delta_Q f\mu))(x)\bigg|^2d\mu(x)\lesssim\|f\|^2_{L^2(\mu)}.
\end{equation*}
\end{lema}
\begin{proof}[{\bf {\em Proof.}}]
Let $C_0>0$ be a small constant to be fixed below. Given $m\in\SSS_D(x)$ set $A_m(x):=A(x,\epsilon_{m+1},\epsilon_m)$, and given $R\in V(D)$ let
\begin{equation*}
\begin{split}
J_m^{1,R}:=\{Q\in \Tre(R):\,Q\cap A_m(x)\neq\emptyset,\,\ell(Q)> C_0(\epsilon_m-\epsilon_{m+1})\},\\
J_m^{2,R}:=\{Q\in \Tre(R):\,Q\cap A_m(x)\neq\emptyset,\,\ell(Q)\leq C_0(\epsilon_m-\epsilon_{m+1})\}.
\end{split}
\end{equation*}
For $Q\in J_m^{1,R}$, we write
$|(K\chi_{\epsilon_{m+1}}^{\epsilon_m}*(\Delta_Q f\mu))(x)|
\lesssim\ell(D)^{-n}\|\chi_{A_m(x)}\Delta_Q f\|_{L^1(\mu)}$. The following claim will be proved in Subsection \ref{5ss proof of claims} below.
\begin{claim}\label{5 claim1}
The following estimate holds: $\sum_{Q\in J_m^{1,R}}\ell(Q)^{n-1/2}\lesssim\ell(D)^{n-1/2}$.
\end{claim}
Using that $V(D)$ has finitely many elements (depending only on $n$ and the AD regularity constant of $\mu$), Cauchy-Schwarz inequality, Claim \ref{5 claim1}, and the previous estimate, we obtain
\begin{equation}\label{5 var eq3}
\begin{split}
\sum_{m\in\SSS_D(x)}\bigg|\sum_{R\in V(D)}&\,\sum_{Q\in J_m^{1,R}}(K\chi_{\epsilon_{m+1}}^{\epsilon_m}*(\Delta_Q f\mu))(x)\bigg|^2\\
&\lesssim\sum_{R\in V(D)}\,\sum_{m\in\SSS_D(x)}\bigg(\sum_{Q\in J_m^{1,R}}
\ell(D)^{-n}\|\chi_{A_m(x)}\Delta_Q f\|_{L^1(\mu)}\bigg)^2\\
&\lesssim\sum_{R\in V(D)}\,\sum_{m\in\SSS_D(x)}\bigg(\sum_{Q\in J_m^{1,R}}\ell(Q)^{n-1/2}\bigg)
\bigg(\sum_{Q\in J_m^{1,R}}
\frac{\|\chi_{A_m(x)}\Delta_Q f\|^2_{L^1(\mu)}}{\ell(D)^{2n}\ell(Q)^{n-1/2}}\bigg)\\
&\lesssim\sum_{R\in V(D)}\,\sum_{m\in\SSS_D(x)}
\,\sum_{Q\in \Tre(R)}
\frac{\|\chi_{A_m(x)}\Delta_Q f\|^2_{L^1(\mu)}}{\ell(D)^{n+1/2}\ell(Q)^{n-1/2}}\\
&\lesssim\sum_{R\in V(D)}\,
\,\sum_{Q\in \Tre(R)}
\bigg(\frac{\ell(Q)}{\ell(D)}\bigg)^{1/2}\frac{\|\Delta_Q f\|^2_{L^1(\mu)}}{\ell(D)^{n}\ell(Q)^{n}}.
\end{split}
\end{equation}

We deal now with the $\mu$-cubes $Q\in J_m^{2,R}$. Let $z_Q$ denote the center of $Q$. Since $\int\Delta_Q f\,d\mu=0$, we can decompose
\begin{equation}\label{5 var eq4}
\begin{split}
(K\chi_{\epsilon_{m+1}}^{\epsilon_m}\!\!*\!(\Delta_Q f\mu))(x)
&=\!\int\!\big( \chi_{A_m(x)}(y)K(x-y)
\!-\!\chi_{A_m(x)}(z_Q)K(x-z_Q)\big)\Delta_Qf(y)\,d\mu(y)\\
&=\!\int\!\chi_{A_m(x)}(y)\Big(K(x-y)\!-\!K(x-z_Q)\Big)\Delta_Qf(y)\,d\mu(y)\\
&\quad+\int\Big(\chi_{A_m(x)}(y)-\chi_{A_m(x)}(z_Q)\Big)K(x-z_Q)\Delta_Qf(y)\,d\mu(y)\\
&=:T_m^{1,\mu}(\Delta_Qf)(x)+T_m^{2,\mu}(\Delta_Qf)(x).
\end{split}
\end{equation}

For the first term on the right hand side of the last equality, we have the standard estimate (by assuming $C_0$ small enough, so any $Q\in J_m^{2,R}$ is far from $x$)
$$|T_m^{1,\mu}(\Delta_Qf)(x)|\lesssim\int_{A_m(x)}\frac{|y-z_Q|}{|x-y|^{n+1}}\,|\Delta_Qf(y)|\,d\mu(y)
\lesssim\frac{\ell(Q)}{\ell(D)^{n+1}}\|\chi_{A_m(x)}\Delta_Qf\|_{L^1(\mu)}.$$
From this estimate and Cauchy-Schwarz inequality, we obtain
\begin{equation*}
\begin{split}
\sum_{m\in\SSS_D(x)}\bigg|\sum_{R\in V(D)}&\,
\sum_{Q\in J_m^{2,R}}T_m^{1,\mu}(\Delta_Qf)(x)\bigg|^2\\
&\lesssim\sum_{R\in V(D)}\,\sum_{m\in\SSS_D(x)}\bigg(\sum_{Q\in J_m^{2,R}}
\frac{\ell(Q)}{\ell(D)^{n+1}}\|\chi_{A_m(x)}\Delta_Qf\|_{L^1(\mu)}\bigg)^2\\
&\lesssim\sum_{R\in V(D)}\bigg(\sum_{Q\in \Tre(R)}\frac{\ell(Q)}{\ell(D)^{n+1}}\sum_{m\in\SSS_D(x)}
\|\chi_{A_m(x)}\Delta_Qf\|_{L^1(\mu)}\bigg)^2\\
&\lesssim\sum_{R\in V(D)}\bigg(\sum_{Q\in \Tre(R)}\frac{\ell(Q)^{n+1}}{\ell(D)^{n+1}}\bigg)\bigg(\sum_{Q\in \Tre(R)}\frac{\|\Delta_Qf\|^2_{L^1(\mu)}}
{\ell(Q)^{n-1}\ell(D)^{n+1}}\bigg).
\end{split}
\end{equation*}
Since $\ell(R)=\ell(D)$ for all $R\in V(D)$, we have 
$\sum_{Q\in \Tre(R)}\big(\frac{\ell(Q)}{\ell(D)}\big)^{n+1}\leq\sum_{Q\in\DD:\,Q\subset R}\big(\frac{\ell(Q)}{\ell(R)}\big)^{n+1}\lesssim1$. Thus, using that $t\lesssim\sqrt{t}$ for all $t\lesssim1$, we conclude
\begin{equation}\label{5 var eq5}
\begin{split}
\sum_{m\in\SSS_D(x)}\bigg|&\sum_{R\in V(D)}\,\sum_{Q\in J_m^{2,R}}T_m^{1,\mu}(\Delta_Qf)(x)\bigg|^2\lesssim\sum_{R\in V(D)}\sum_{Q\in \Tre(R)}\bigg(\frac{\ell(Q)}{\ell(D)}\bigg)^{1/2}\frac{\|\Delta_Qf\|^2_{L^1(\mu)}}{\ell(Q)^{n}\ell(D)^{n}}.
\end{split}
\end{equation}

We deal now with the second term on the right hand side of (\ref{5 var eq4}). Given $Q\in J_m^{2,R}$, since $\supp(\Delta_Qf)\subset Q$, if $Q\subset A_m(x)$ or $Q\subset(A_m(x))^c$ then we obviously have $\chi_{A_m(x)}(y)-\chi_{A_m(x)}(z_Q)=0$ for all $y\in\supp(\Delta_Qf)$. Therefore, to estimate the sum of $T_m^{2,\mu}(\Delta_Qf)(x)$ over all $Q\in J_m^{2,R}$, we can replace $J_m^{2,R}$ by
$$J_m^{3,R}:=\{Q\in \Tre(R):\,Q\cap A_m(x)\neq\emptyset,\,Q\cap (A_m(x))^c\neq\emptyset,\,\ell(Q)\leq C_0(\epsilon_m-\epsilon_{m+1})\}.$$
For $m\in\SSS_D(x)$ and $Q\in J_m^{3,R}$, we will use the estimate $|T_m^{2,\mu}(\Delta_Qf)(x)|\lesssim
\ell(D)^{-n}\|\Delta_Qf\|_{L^1(\mu)}.$
\begin{claim}\label{5 claim2}
The following holds: $\sum_{Q\in J_m^{3,R}}\ell(Q)^{n-1/2}\lesssim\ell(D)^{n-1}(\epsilon_m-\epsilon_{m+1})^{1/2}$.
\end{claim}
Hence, using that $V(D)$ has finitely many terms, Cauchy-Schwarz inequality, assuming Claim \ref{5 claim2} (see Subsection \ref{5ss proof of claims}), and by the previous estimate, we deduce
\begin{equation*}
\begin{split}
\sum_{m\in\SSS_D(x)}\bigg|&\sum_{R\in V(D)}\,\sum_{Q\in J_m^{2,R}}T_m^{2,\mu}(\Delta_Qf)(x)\bigg|^2\lesssim\sum_{R\in V(D)}\,\sum_{m\in\SSS_D(x)}\bigg(\sum_{Q\in J_m^{3,R}}
\frac{\|\Delta_Qf\|_{L^1(\mu)}}{\ell(D)^{n}}\bigg)^2\\
&\leq\sum_{R\in V(D)}\,\sum_{m\in\SSS_D(x)}
\bigg(\sum_{Q\in J_m^{3,R}}\frac{\ell(Q)^{n-1/2}}{\ell(D)^{n-1/2}}\bigg)
\bigg(\sum_{Q\in J_m^{3,R}}
\frac{\ell(Q)^{1/2-n}}{\ell(D)^{n+1/2}}\,\|\Delta_Qf\|^2_{L^1(\mu)}\bigg)\\
&\lesssim\sum_{R\in V(D)}\,\sum_{m\in\SSS_D(x)}\bigg(\frac{\epsilon_m-\epsilon_{m+1}}{\ell(D)}\bigg)^{1/2}
\sum_{Q\in J_m^{3,R}}
\frac{\ell(Q)^{1/2-n}}{\ell(D)^{n+1/2}}\,\|\Delta_Qf\|^2_{L^1(\mu)}\\
&\leq\sum_{R\in V(D)}\,\sum_{Q\in \Tre(R)}
\frac{\ell(Q)^{1/2-n}}{\ell(D)^{n+1/2}}\,\|\Delta_Qf\|^2_{L^1(\mu)}\sum_{\begin{subarray}{c}m\in\SSS_D(x):\,A_m(x)\cap Q\neq\emptyset,\\\ell(Q)\leq C_0(\epsilon_m-\epsilon_{m+1})\end{subarray}}\bigg(\frac{\epsilon_m-\epsilon_{m+1}}{\ell(D)}\bigg)^{1/2}.
\end{split}
\end{equation*}
The sum over $m$ on the right hand side of the last inequality can be easily bounded by some constant depending on $C_0$, thus we finally obtain
\begin{equation}\label{5 var eq6}
\begin{split}
\sum_{m\in\SSS_D(x)}\bigg|&\sum_{R\in V(D)}\,\sum_{Q\in J_m^{2,R}}T_m^{2,\mu}(\Delta_Qf)(x)\bigg|^2\lesssim\sum_{R\in V(D)}\sum_{Q\in \Tre(R)}\bigg(\frac{\ell(Q)}{\ell(D)}\bigg)^{1/2}\frac{\|\Delta_Qf\|^2_{L^1(\mu)}}{\ell(Q)^{n}\ell(D)^{n}}.
\end{split}
\end{equation}

Finally, combining (\ref{5 var eq3}), (\ref{5 var eq4}), (\ref{5 var eq5}), and (\ref{5 var eq6}), we conclude
\begin{equation}\label{5 var eq7}
\begin{split}
\sum_{m\in\SSS_D(x)}\!\bigg|\!\sum_{R\in V(D)}\sum_{Q\in \Tre(R)}\!\!(K\chi_{\epsilon_{m+1}}^{\epsilon_m}\!\!*\!(\Delta_Q f\mu))(x)\bigg|^2
\!\lesssim\!\!\sum_{R\in V(D)}\sum_{Q\in \Tre(R)}\!\!\!\bigg(\frac{\ell(Q)}{\ell(D)}\bigg)^{1/2}\frac{\|\Delta_Qf\|^2_{L^1(\mu)}}{\ell(Q)^{n}\ell(D)^{n}},
\end{split}
\end{equation}
Since $\|\Delta_Qf\|_{L^1(\mu)}\lesssim\|\Delta_Qf\|_{L^2(\mu)}\ell(Q)^{n/2}$ by H\"{o}lder's inequality, since $V(D)$ has finitely many terms, and since $\ell(R)=\ell(D)$ for all $R\in V(D)$, we get
\begin{equation*}
\begin{split}
\sum_{S\in\Trees}\,\sum_{D\in S}\int_D\sum_{m\in\SSS_D(x)}\bigg|&\sum_{R\in V(D)}\sum_{Q\in \Tre(R)}(K\chi_{\epsilon_{m+1}}^{\epsilon_m}*(\Delta_Q f\mu))(x)\bigg|^2d\mu(x)\\
&\lesssim\sum_{S\in\Trees}\,\sum_{D\in S}\,\sum_{R\in V(D)}\,\sum_{Q\in \Tre(R)}\bigg(\frac{\ell(Q)}{\ell(D)}\bigg)^{1/2}\|\Delta_Qf\|^2_{L^2(\mu)}\\
&\leq\sum_{S\in\Trees}\,\sum_{Q\in S}\,\sum_{R\in \DD:\,R\supset Q}\,\sum_{D\in V(R)}\bigg(\frac{\ell(Q)}{\ell(R)}\bigg)^{1/2}\|\Delta_Qf\|^2_{L^2(\mu)}\\
&\lesssim\sum_{S\in\Trees}\,\sum_{Q\in S}\|\Delta_Qf\|^2_{L^2(\mu)}
\leq\sum_{Q\in\DD}\|\Delta_Qf\|^2_{L^2(\mu)}\leq\|f\|^2_{L^2(\mu)}.
\end{split}
\end{equation*}

To complete the proof of Lemma \ref{5 var eq8}, it only remains to show Claims \ref{5 claim1} and \ref{5 claim2}.
\end{proof}

\subsubsection{{\bf Proof of Claims \ref{5 claim1} and \ref{5 claim2}}}\label{5ss proof of claims}

First of all, we need an auxiliary result whose easy proof is left for the reader.
\begin{lema}\label{lema pendent petita3}
Let $\Gamma:=\{x\in \R^d\,:\,x=(y,A(y)),\,y\in\R^n\}$ be the graph of a Lipschitz function $A:\R^n\to\R^{d-n}$ such that $\Lip(A)$ is small enough. Then, $\HH^n_\Gamma(A^d(z,a,b))\lesssim(b-a)b^{n-1}$ for all $0<a\leq b$ and $z\in\Gamma$.
\end{lema}

\begin{remarko}\label{rem499}{\em
Actually, to obtain the conclusion of the lemma, one only needs $\Lip(A)<1$ (see \cite[Lemma 4.1.9]{Mas-thesis}). Let us mention that this assumption is sharp in the sense that if $\Lip(A)\geq1$ then the lemma fails. However, we do not need this stronger version for our purposes.}
\end{remarko}

Claims \ref{5 claim1} and \ref{5 claim2} follow from the next lemma, which will be proved using Lemma \ref{lema pendent petita3}.
\begin{lema}\label{5 lema claims}
Let $C_0>0$ be some constant depending only on $n$, $d$, and the AD regularity constant of $\mu$, and consider $x\in D\in\DD_j$ for some $j\in\Z$.  Let $\epsilon\in[2^{-j-1},2^{-j}).$ Given $k\geq j$ and $R\in V(D)$, set
$$\Lambda_k:=\{Q\in \Tre(R)\cap\DD_k:\, Q\subset A(x,\epsilon-C_02^{-k},\epsilon+C_02^{-k})\}.$$
Then, $\mu\big(\bigcup_{Q\in\Lambda_k}Q\big)\lesssim2^{-k}\ell(D)^{n-1}\approx2^{-k-j(n-1)}$.
\end{lema}

\begin{proof}[{\bf {\em Proof.}}]
First of all, we can assume $k\gg j$ (otherwise, the claim follows easily using the AD regularity of $\mu$), thus we may assume that $\dist(x,Q)\geq\frac{3}{4}\,\epsilon$.
For simplicity, set $S\equiv \Tre(R)$. By the property $(f)$ of the corona decomposition of $\mu$, there exists a (rotation and translation of an) $n$-dimensional Lipschitz graph $\Gamma_S$ with $\Lip(\Gamma_S)\leq\eta$ such that $\dist(y,\Gamma_S)\leq\theta\,\diam(Q)$ whenever $y\in C_{cor}Q$ and $Q\in S$, for some given constant $C_{cor}\geq2$. Since $x\in D$ and $R\in V(D)$, we have $x\in C_{cor}Q$ assuming $C_{cor}$ big enough, and so $\dist(x,\Gamma_S)\leq\theta\,\diam(Q)$. Hence, if $\eta$ and $\theta$ are small enough, one can easily modify $\Gamma_S$ inside $B(x,\frac{1}{4}\,\epsilon)$ to obtain a Lipschitz graph $\Gamma_S^x$ such that $x\in\Gamma_S^x$, and moreover
\begin{equation}\label{5lema anell}
\Lip(\Gamma_S^x)\leq\eta'\text{ for some }\eta'\text{ small enough},\quad\text{and}\quad\Gamma_S^x
\setminus B(x,\epsilon/4)=\Gamma_S\setminus B(x,\epsilon/4).
\end{equation}
Using that $\dist(x,Q)\geq\frac{3}{4}\,\epsilon$ for all $Q\in\Lambda_k$, that $\dist(z_Q,\Gamma_S)\leq\theta\,\diam(Q)$ for the centre $z_Q$ of $Q$, and the last part of (\ref{5lema anell}), we deduce  that $\dist(z_Q,\Gamma_S^x)\leq\theta\,\diam(Q)$ for all $Q\in\Lambda_k$.
So $B(z_Q,\theta\,\diam(Q))\cap\Gamma_S^x\neq\emptyset$, which in turn yields 
$\HH^n\big({\Gamma_S^x}\cap B(z_Q,2\theta\,\diam(Q))\big)\gtrsim(\theta\,\diam(Q))^n$. Therefore, since $\{B(z_Q,2\theta\,\diam(Q))\}_{Q\in\Lambda_k}$ is a family with finite overlap bounded by some constant depending only on $n$, $\theta$, and the AD regularity constant of $\mu$, we have
\begin{equation*}
\begin{split}
\mu\bigg(\bigcup_{Q\in\Lambda_k}Q\bigg)&
\approx\sum_{Q\in\Lambda_k}\ell(Q)^n
\lesssim\theta^{-n}\sum_{Q\in\Lambda_k}\HH^n\big({\Gamma_S^x}\cap B(z_Q,2\theta\,\diam(Q))\big)\\
&\lesssim\theta^{-n}\HH^n_{\Gamma_S^x}\bigg(\bigcup_{Q\in\Lambda_k}B(z_Q,2\theta\,\diam(Q))\bigg)\\
&\lesssim\theta^{-n}\HH^n_{\Gamma_S^x}\big(A(x,\epsilon-C_02^{-k},\epsilon+C_02^{-k})\big)
\lesssim\theta^{-n}2^{-k-j(n-1)},
\end{split}
\end{equation*}
where we used Lemma \ref{lema pendent petita3} and that $\epsilon\approx2^{-j}$ in the last inequality. The lemma is proved. 
\end{proof}

\begin{proof}[{\bf{\em Proof of} Claim \ref{5 claim1}}]
Recall that $J_m^{1,R}:=\{Q\in \Tre(R):\,Q\cap A_m(x)\neq\emptyset,\,\ell(Q)\geq C_0(\epsilon_m-\epsilon_{m+1})\},$ where $R\in V(D)$ and $D\in\DD_j$. We have to check that $\sum_{Q\in J_m^{1,R}}\ell(Q)^{n-1/2}\lesssim\ell(D)^{n-1/2}$.
We will split the sum into different scales and we will apply Lemma \ref{5 lema claims} at each scale.

Given $i\in\Z$ such that $2^{-i}\geq C_0(\epsilon_m-\epsilon_{m+1})$, the number of $\mu$-cubes $Q\in\DD_i$ such that $Q\subset R$ and $Q\cap A_m(x)\neq\emptyset$ is bounded by $C\ell(R)^{n-1}2^{i(n-1)}\approx2^{-j(n-1)+i(n-1)}$, since 
for all these $\mu$-cubes, $Q\subset A(x,\epsilon_{m+1}-C2^{-i},\epsilon_m+C2^{-i})\subset
A(x,\epsilon_{m}-C2^{-i+1},\epsilon_m+C2^{-i+1})$ for some constant $C>0$ big enough, and then by Lemma \ref{5 lema claims}, 
$\mu\big(\bigcup_{Q\in J_m^{1,R}\cap\DD_i}Q\big)\lesssim 2^{-i}\ell(D)^{n-1}$.  Therefore, 
\begin{equation*}
\begin{split}
\sum_{Q\in J_m^{1,R}}\ell(Q)^{n-1/2}&=\sum_{i\in\Z:\,i\geq j}2^{i/2}\sum_{Q\in J_m^{1,R}\cap\DD_i}\ell(Q)^{n}\lesssim\sum_{i\in\Z:\,i\geq j}2^{i/2}2^{-i}\ell(D)^{n-1}\\
&\approx2^{-j/2}\ell(D)^{n-1}=\ell(D)^{n-1/2}.
\end{split}
\end{equation*}
\end{proof}

\begin{proof}[{\bf{\em Proof of} Claim \ref{5 claim2}}]
Recall that $J_m^{3,R}:=\{Q\in \Tre(R):\,Q\cap A_m(x)\neq\emptyset,\,Q\cap (A_m(x))^c\neq\emptyset,\,\ell(Q)\leq C_0(\epsilon_m-\epsilon_{m+1})\},$ where $R\in V(D)$ and $D\in\DD_j$. We have to check that $$\sum_{Q\in J_m^{3,R}}\ell(Q)^{n-1/2}\lesssim\ell(D)^{n-1}(\epsilon_m-\epsilon_{m+1})^{1/2}.$$ As before, 
we will split the sum into the different scales and we will apply Lemma \ref{5 lema claims} at each scale.
Given $i\in\Z$ such that $2^{-i}\leq C_0(\epsilon_m-\epsilon_{m+1})$, since for any $Q\in J_m^{3,R}\cap\DD_i$ we have $Q\subset A(x,\epsilon_{m+1}-C2^{-i},\epsilon_{m+1}+C2^{-i})\cup A(x,\epsilon_{m}-C2^{-i},\epsilon_{m}+C2^{-i}$) for some constant $C>0$ big enough, by Lemma \ref{5 lema claims} applied to both annuli we have $\mu\big(\bigcup_{Q\in J_m^{3,R}\cap\DD_i}Q\big)\lesssim2^{-i}\ell(D)^{n-1}$. Therefore,
\begin{equation*}
\begin{split}
\sum_{Q\in J_m^{3,R}}\ell(Q)^{n-1/2}&=\sum_{\begin{subarray}{c}i\in\Z:\,i\geq-\log_2( C_0(\epsilon_m-\epsilon_{m+1}))\end{subarray}}2^{i/2}\sum_{Q\in J_m^{3,R}\cap\DD_i}\ell(Q)^n\\
&\lesssim\sum_{\begin{subarray}{c}i\in\Z:\,i\geq-\log_2( C_0(\epsilon_m-\epsilon_{m+1}))\end{subarray}}2^{-i/2}\ell(D)^{n-1}\approx(\epsilon_m-\epsilon_{m+1})^{1/2}\ell(D)^{n-1}.
\end{split}
\end{equation*}
\end{proof}

\subsubsection{{\bf Estimate of 
$\sum_{m\in\SSS_D(x)}\big|\sum_{R\in V(D)}\sum_{Q\in \Stop(R)}(K\chi_{\epsilon_{m+1}}^{\epsilon_m}*(\wit\Delta_Q f\mu))(x)\big|^2$ from (\ref{eqmain})}}\label{5 ss5321b}
\begin{lema}\label{5 var eq11}
Under the notation above, we have
\begin{equation*}
\sum_{S\in\Trees}\,\sum_{D\in S}\int_D\sum_{m\in\SSS_D(x)}\bigg|\sum_{R\in V(D)}\sum_{Q\in \Stop(R)}(K\chi_{\epsilon_{m+1}}^{\epsilon_m}*(\wit\Delta_Q f\mu))(x)\bigg|^2d\mu(x)\lesssim\|f\|^2_{L^2(\mu)}.
\end{equation*}
\end{lema}
\begin{proof}[{\bf{\em Proof}}]
Given $R\in V(D)$, consider a $\mu$-cube $Q\in \Stop(R)$. If $\Tre(R)\neq\emptyset$, then
$Q\in\BB\cup(\GG\setminus \Tre(R))$, $Q\subset R$ and $P(Q)\in \Tre(R)$ (in particular, $Q\subsetneq R$). Take $S\in\Trees$ such that $R\in S$. By property $(f)$ of the corona decomposition (see Subsection \ref{5ss corona decomposition}), we have $\dist(y,\Gamma_S)\leq\theta\diam(P(Q))$ for all $y\in C_{cor}P(Q)$. Hence, $\dist(y,\Gamma_S)\leq C\theta\diam(Q)$ for all $y\in C_{cor}Q$. On the other hand, if $\Tre(R)=\emptyset$ we have set $\Stop(R)=\{R\}$. In this case, we have $R\in\BB$. Take $S$ such that $D\in S$. Since $R\in V(D)$, we have $R\subset C_{cor}D$ if $C_{cor}$ is chosen big enough, and thus $\dist(y,\Gamma_S)\leq C\theta \diam(R)$ for all $y\in C'R$, where $C$ is as above and $C'$ depends on $C_{cor}$.

Taking into account the comments above, one can prove the following claims using similar arguments to the ones in the proof of  Claims \ref{5 claim1} and \ref{5 claim2}.

\begin{claim}\label{5 claim3}
Let $x\in D\in\DD$, $R\in V(D)$, and $m\in\SSS_D(x)$. If we set 
$J_m^{1,R}:=\{Q\in \Stop(R):\,Q\cap A_m(x)\neq\emptyset,\,\ell(Q)\geq C_0(\epsilon_m-\epsilon_{m+1})\},$
then $\sum_{Q\in J_m^{1,R}}\ell(Q)^{n-1/2}\lesssim\ell(D)^{n-1/2}$.
\end{claim}

\begin{claim}\label{5 claim4}
Let $x\in D\in\DD$, $R\in V(D)$, and $m\in\SSS_D(x)$. If we set 
$J_m^{3,R}:=\{Q\in \Stop(R):\,Q\cap A_m(x)\neq\emptyset,\,Q\cap (A_m(x))^c\neq\emptyset,\,\ell(Q)\leq C_0(\epsilon_m-\epsilon_{m+1})\},$
then $\sum_{Q\in J_m^{3,R}}\ell(Q)^{n-1/2}\lesssim\ell(D)^{n-1}(\epsilon_m-\epsilon_{m+1})^{1/2}$.
\end{claim}

The only properties of $\Delta_Q f$ that we used to obtain (\ref{5 var eq7}) were that $\Delta_Q f$ is supported in $Q$ and that $\int\Delta_Q f\,d\mu=0$. The function $\wit\Delta_Qf$ is also supported in $Q$ and has vanishing integral. Thus, if we replace $\Tre(R)$ by $\Stop(R)$, Claims \ref{5 claim1} and \ref{5 claim2} by Claims \ref{5 claim3} and \ref{5 claim4}, and $\Delta_Q f$ by $\wit\Delta_Q f$, the same arguments that gave us (\ref{5 var eq7}) yield the following estimate:
\begin{equation}\label{5 var eq9}
\begin{split}
\!\sum_{m\in\SSS_D(x)}\!\bigg|\!\sum_{R\in V(D)}\sum_{Q\in \Stop(R)}\!\!(K\chi_{\epsilon_{m+1}}^{\epsilon_m}\!\!*\!(\wit\Delta_Q f\mu))(x)\bigg|^2
\!\!\lesssim\!\!\!\sum_{R\in V(D)}\sum_{Q\in \Stop(R)}\!\!\frac{\ell(Q)^{1/2-n}}{\ell(D)^{1/2+n}}\,\|\wit\Delta_Qf\|^2_{L^1(\mu)}.
\end{split}
\end{equation}
Below we will use that 
$\|\wit\Delta_Qf\|^2_{L^1(\mu)}\ell(Q)^{-n}=\big(\int_Q|f-m^\mu_Qf|\,d\mu\big)^2\ell(Q)^{-n}\lesssim\big(m^\mu_Q|f|\big)^2\mu(Q).$
Notice that, by the definition of $\Stop(R)$ and since the corona decomposition is coherent (property $(d)$), any $Q\in \Stop(R)$ is actually a maximal $\mu$-cube $Q_S$ of some $S\in\Trees$ or $Q\in\BB$ (and in this case $\Tre(R)$ is empty). Hence, if we integrate (\ref{5 var eq9}) in $D$, we sum over all $D\in S\in\Trees$, and we change the order of summation, we get
\begin{equation*}
\begin{split}
\sum_{S\in\Trees}\,\sum_{D\in S}\int_D\sum_{m\in\SSS_D(x)}&\bigg|\sum_{R\in V(D)}\sum_{Q\in \Stop(R)}(K\chi_{\epsilon_{m+1}}^{\epsilon_m}*(\wit\Delta_Q f\mu))(x)\bigg|^2d\mu(x)\\
&\lesssim\sum_{S\in\Trees}\,\sum_{D\in S}\,\sum_{R\in V(D)}\,\sum_{Q\in \Stop(R)}\bigg(\frac{\ell(Q)}{\ell(D)}\bigg)^{1/2}\frac{\|\wit\Delta_Qf\|^2_{L^1(\mu)}}{\ell(Q)^{n}}\\
&\lesssim\sum_{D\in\DD}\,\sum_{R\in V(D)}\,\sum_{S\in\Trees:\,Q_S\subset R}\bigg(\frac{\ell(Q_S)}{\ell(D)}\bigg)^{1/2}\big(m^\mu_{Q_S}|f|\big)^2\mu(Q_S)\\
&\quad+\sum_{D\in\DD}\,\sum_{R\in V(D)}\,\sum_{Q\in\BB:\,Q\subset R}\bigg(\frac{\ell(Q)}{\ell(D)}\bigg)^{1/2}\big(m^\mu_{Q}|f|\big)^2\mu(Q)\\
&=\sum_{S\in\Trees}\,\sum_{R\in\DD:\,R\supset Q_S}\,\sum_{D\in V(R)}\bigg(\frac{\ell(Q_S)}{\ell(R)}\bigg)^{1/2}\big(m^\mu_{Q_S}|f|\big)^2\mu(Q_S)\\
&\quad+\sum_{Q\in\BB}\,\sum_{R\in\DD:\,R\supset Q}\,\sum_{D\in V(R)}\bigg(\frac{\ell(Q)}{\ell(R)}\bigg)^{1/2}\big(m^\mu_{Q}|f|\big)^2\mu(Q).
\end{split}
\end{equation*}

Finally, using that $V(R)$ has finitely many elements, and that the $\mu$-cubes $Q_S$ with $S\in\Trees$ and the $\mu$-cubes $Q\in\BB$ satisfy a Carleson packing condition (so we can apply Carleson's embedding theorem), we deduce
\begin{equation*}
\begin{split}
\sum_{S\in\Trees}&\,\sum_{D\in S}\int_D\sum_{m\in\SSS_D(x)}\bigg|\sum_{R\in V(D)}\sum_{Q\in \Stop(R)}(K\chi_{\epsilon_{m+1}}^{\epsilon_m}*(\wit\Delta_Q f\mu))(x)\bigg|^2d\mu(x)\\
&\lesssim\sum_{S\in\Trees}\big(m^\mu_{Q_S}|f|\big)^2\mu(Q_S)\sum_{R\in\DD:\,R\supset Q_S}\frac{\ell(Q_S)^{1/2}}{\ell(R)^{1/2}}+\sum_{Q\in\BB}\big(m^\mu_{Q}|f|\big)^2\mu(Q)\sum_{R\in\DD:\,R\supset Q}\frac{\ell(Q)^{1/2}}{\ell(R)^{1/2}}\\
&\lesssim\sum_{S\in\Trees}\big(m^\mu_{Q_S}|f|\big)^2\mu(Q_S)
+\sum_{Q\in\BB}\big(m^\mu_{Q}|f|\big)^2\mu(Q)
\lesssim\|f\|^2_{L^2(\mu)}.
\end{split}
\end{equation*}
\end{proof}

\subsubsection{{\bf Estimate of 
$\sum_{m\in\SSS_D(x)}\big|\sum_{R\in V(D)}
(K\chi_{\epsilon_{m+1}}^{\epsilon_m}*((m_R^\mu f)\chi_R\mu))(x)\big|^2$ from (\ref{eqmain})}}
We will need the following auxiliary lemma, which we prove for completeness, despite we think it is already known.
\begin{lema}\label{5 veins quasiortogonalitat}
Given $D\in\DD$ and $f\in L^2(\mu)$, set $a_D(f):=\sum_{R\in V(D)}|m_R^\mu f-m_D^\mu f|$. Then, there exists $C>0$ depending only $n$ and the AD regularity constant of $\mu$ such that
$$\sum_{D\in\DD}(a_D(f))^2\mu(D)\leq C\|f\|^2_{L^2(\mu)}.$$
\end{lema}

\begin{proof}[{\bf {\em Proof.}}]
By subtracting a constant if necessary, we can assume that $f$ has mean zero. Consider the representation of $f$ with respect to the Haar basis associated to $\DD$, that is $f=\sum_{Q\in\DD}\Delta_Qf$. For $m\in\Z$, we define the function $u_m=\sum_{Q\in\DD_m}\Delta_Qf$, so $f=\sum_{m\in\Z}u_m$ and the equality holds in $L^2(\mu)$.
Given $j\in\Z$, define the operator $$S_j(f):=\bigg(\sum_{D\in\DD_j}(a_D(f))^2\chi_D\bigg)^{1/2}.$$
We will prove that there exists a sequence $\{\sigma(k)\}_{k\in\Z}$ such that 
\begin{equation}\label{5 vei quasiort}
\sum_{k\in\Z}\sigma(k)\leq C<\infty\quad\text{and}\quad\|S_j(u_m)\|_{L^2(\mu)}\lesssim\sigma(|m-j|)\|u_m\|_{L^2(\mu)}.
\end{equation}

Assume for the moment that (\ref{5 vei quasiort}) holds. Then, since each $S_j$ is sublinear, by Cauchy-Schwarz inequality and the orthogonality of the $u_m$'s,
\begin{equation*}
\begin{split}
\sum_{D\in\DD}(a_D(f))^2\mu(D)&=\sum_{j\in\Z}\int\sum_{D\in\DD_j}(a_D(f))^2\chi_D\,d\mu
=\sum_{j\in\Z}\|S_j(f)\|^2_{L^2(\mu)}\\
&=\sum_{j\in\Z}\bigg\|S_j\bigg(\sum_{m\in\Z}u_m\bigg)\bigg\|^2_{L^2(\mu)}
\leq\sum_{j\in\Z}\bigg(\sum_{m\in\Z}\|S_j(u_m)\|_{L^2(\mu)}\bigg)^2\\
&\leq\sum_{j\in\Z}\bigg(\sum_{m\in\Z}\sigma(|m-j|)\bigg)\bigg(\sum_{m\in\Z}\sigma(|m-j|)^{-1}\|S_j(u_m)\|^2_{L^2(\mu)}\bigg)\\
&\lesssim\sum_{j\in\Z}\sum_{m\in\Z}\sigma(|m-j|)\|u_m\|^2_{L^2(\mu)}
=\sum_{m\in\Z}\|u_m\|^2_{L^2(\mu)}\sum_{j\in\Z}\sigma(|m-j|)\\
&\lesssim\sum_{m\in\Z}\|u_m\|^2_{L^2(\mu)}=\|f\|^2_{L^2(\mu)},
\end{split}
\end{equation*}
and the lemma follows. Let us verify (\ref{5 vei quasiort}) now. By definition,
\begin{equation}\label{5 def Sj}
\|S_j(u_m)\|^2_{L^2(\mu)}
=\sum_{D\in\DD_j}\bigg(\sum_{R\in V(D)}\bigg|\sum_{Q\in\DD_m}\int\Delta_Qf
\bigg(\frac{\chi_R}{\mu(R)}-\frac{\chi_D}{\mu(D)}\bigg)\,d\mu\,\bigg|\bigg)^2\mu(D).
\end{equation}
Assume first that $m\geq j$. If $D\in\DD_j$, $R\in V(D)$, and $Q\in\DD_m$, then either $Q\cap R=\emptyset$ or $Q\subset R$. In both cases, since $\Delta_Qf$ has mean zero and is supported in $Q$, we have $\int\Delta_Q f\,\chi_R\,d\mu=0$. Thus, the right hand side of (\ref{5 def Sj}) vanishes (obviously $D\in V(D)$), and (\ref{5 vei quasiort}) follows. 

Assume now that $m<j$. Set $\wit D:=\bigcup_{R\in V(D)}R$. Recall that $\Delta_Q f:=\sum_{U\in  \Child( Q)}\chi_U(m_U^\mu f-m_Q^\mu f)$, so $\Delta_Q f$ is constant in each $U\in  \Child( Q)$. Hence, if for some $U\in  \Child( Q)$ we have $\wit D\subset U$ or $\wit D\subset\supp\mu\setminus U$, then $(R\cup D)\subset U$ or $(R\cup D)\cap U=\emptyset$ for all $R\in V(D)$, and so
$$\int\chi_U(m_U^\mu f-m_Q^\mu f)\bigg(\frac{\chi_R}{\mu(R)}-\frac{\chi_D}{\mu(D)}\bigg)\,d\mu
=(m_U^\mu f-m_Q^\mu f)\int_U\bigg(\frac{\chi_R}{\mu(R)}-\frac{\chi_D}{\mu(D)}\bigg)\,d\mu=0$$
for all $R\in V(D)$. Therefore, if we set $m_{U,Q}^\mu f:=(m_U^\mu f-m_Q^\mu f)$, using that $V(D)$ has finitely many elements and that $\int|\mu(R)^{-1}\chi_R-\mu(D)^{-1}\chi_D|\,d\mu\leq2$ for all $R\in V(D)$, we deduce from (\ref{5 def Sj}) that
\begin{equation}\label{5 def Sj1}
\begin{split}
\|S_j(u_m)\|^2_{L^2(\mu)}
&=\sum_{D\in\DD_j}\!\bigg(\!\sum_{R\in V(D)}\!\bigg|\!\sum_{Q\in\DD_m}\!\int\!\!\sum_{\begin{subarray}{c}U\in  \Child( Q):\\\wit D\cap U\neq\emptyset,\\\wit D\cap U^c\neq\emptyset\end{subarray}}\!\!\chi_U\,m_{U,Q}^\mu f
\bigg(\frac{\chi_R}{\mu(R)}-\frac{\chi_D}{\mu(D)}\bigg)\,d\mu\,\bigg|\bigg)^2\!\mu(D)\\
&\lesssim\sum_{D\in\DD_j}\bigg(\sum_{Q\in\DD_m}\,
\sum_{\begin{subarray}{c}U\in  \Child( Q):\,\wit D\cap U\neq\emptyset,\\\wit D\cap U^c\neq\emptyset\end{subarray}}
|m_{U,Q}^\mu f|\bigg)^2\mu(D)\\
&=\sum_{D\in\DD_j}\bigg(\sum_{\begin{subarray}{c}U\in\DD_{m+1}:\,\wit D\cap U\neq\emptyset,\\\wit D\cap U^c\neq\emptyset\end{subarray}}|m_{U,P(U)}^\mu f|\bigg)^2\mu(D).
\end{split}
\end{equation}
It is not hard to show that, since $m<j$ and $D\in\DD_j$, the number of $\mu$-cubes $U\in\DD_{m+1}$ such that $\wit D\cap U\neq\emptyset$ and $\wit D\cap U^c\neq\emptyset$ is bounded by some constant depending only on $n$ and the AD regularity constant of $\mu$ (but not on the precise value of $m$). Hence,
\begin{equation}\label{5 def Sj2}
\begin{split}
\sum_{D\in\DD_j}\bigg(\sum_{\begin{subarray}{c}U\in\DD_{m+1}:\,\wit D\cap U\neq\emptyset,\\\wit D\cap U^c\neq\emptyset\end{subarray}}|m_{U,P(U)}^\mu f|\bigg)^2\mu(D)
&\lesssim\sum_{D\in\DD_j}\,\sum_{\begin{subarray}{c}U\in\DD_{m+1}:\,\wit D\cap U\neq\emptyset,\\\wit D\cap U^c\neq\emptyset\end{subarray}}|m_{U,P(U)}^\mu f|^2\mu(D)\\
&=\sum_{U\in\DD_{m+1}}|m_{U,P(U)}^\mu f|^2
\,\mu\bigg(\bigcup_{\begin{subarray}{c}D\in\DD_{j}:\,\wit D\cap U\neq\emptyset,\\\wit D\cap U^c\neq\emptyset\end{subarray}}D\bigg).
\end{split}
\end{equation}
Fix $U\in\DD_{m+1}$. Recall that $\wit D:=\bigcup_{R\in V(D)}R$, so $\diam(\wit D)\approx\diam(D)$. Thus, there exists a constant $\tau_0>0$ such that 
\begin{equation*}
\begin{split}
\bigcup_{\begin{subarray}{c}D\in\DD_{j}:\,\wit D\cap U\neq\emptyset,\,\wit D\cap U^c\neq\emptyset \end{subarray}}D
&\subset\{x\in U:\, \dist(x,\supp\mu\setminus U)\leq\tau_0\ell(D)\}\\
&\quad\cup\{x\in \supp\mu\setminus U:\, \dist(x,U)\leq\tau_0\ell(D)\}\\
&=\{x\in U:\, \dist(x,\supp\mu\setminus U)\leq\tau_02^{m-j+1}\ell(U)\}\\
&\quad\cup\{x\in \supp\mu\setminus U:\, \dist(x,U)\leq\tau_02^{m-j+1}\ell(U)\}.
\end{split}
\end{equation*}
If $m\ll j$, then $\tau:=\tau_02^{m-j+1}<1$, so we can apply the {\em small boundaries condition} (\ref{small boundary condition}) of Subsection \ref{dyadic lattice} to obtain
$\mu\big(\bigcup_{D\in\DD_{j}:\,\wit D\cap U\neq\emptyset,\,\wit D\cap U^c\neq\emptyset}D\big)\leq C\tau^{1/C}2^{-mn}.$
On the contrary, if $|m-j|\lesssim1$, then $\tau^{1/C}\approx1$, so 
$\mu\big(\bigcup_{D\in\DD_{j}:\,\wit D\cap U\neq\emptyset,\,\wit D\cap U^c\neq\emptyset}D\big)\leq\mu(C_1U)\lesssim2^{-mn}\approx\tau^{1/C}2^{-mn}$, for some big constant $C_1>0$. Thus, in any case, 
$\mu\big(\bigcup_{D\in\DD_{j}:\,\wit D\cap U\neq\emptyset,\,\wit D\cap U^c\neq\emptyset}D\big)\lesssim2^{(m-j)/C}\ell(U)^n,$ and combining this with (\ref{5 def Sj2}) and (\ref{5 def Sj1}) we conclude that, for $m<j$,
\begin{equation*}
\begin{split}
\|S_j(u_m)\|^2_{L^2(\mu)}&\lesssim
2^{(m-j)/C}\sum_{U\in\DD_{m+1}}|m_U^\mu f-m_{P(U)}^\mu f|^2\ell(U)^n\\
&\approx2^{(m-j)/C}\int\sum_{U\in\DD_{m+1}}\chi_U|m_U^\mu f-m_{P(U)}^\mu f|^2\,d\mu=2^{-|m-j|/C}\|u_m\|^2_{L^2(\mu)},
\end{split}
\end{equation*}
which gives (\ref{5 vei quasiort})  with $\sigma(k)=2^{-\frac{|k|}{2C}}$ and finishes the proof of the lemma.
\end{proof}

\begin{lema}\label{5 var eq8a}
Under the notation above, we have
\begin{equation*}
\sum_{S\in\Trees}\sum_{D\in S}\int_D\sum_{m\in\SSS_D(x)}\bigg|\sum_{R\in V(D)}
(K\chi_{\epsilon_{m+1}}^{\epsilon_m}*((m_R^\mu f)\chi_R\mu))(x)\bigg|^2\,d\mu(x)\lesssim\|f\|_{L^2(\mu)}^2.
\end{equation*}
\end{lema}
\begin{proof}[{\bf{\em Proof}}]

Recall that, given $D\in\DD$, we have set $\wit D:=\bigcup_{R\in V(D)}R$. For $x\in D$, we have
\begin{equation}\label{5 var quasiort eq1}
\begin{split}
\sum_{m\in\SSS_D(x)}\!\bigg|\!\sum_{R\in V(D)}
(K\chi_{\epsilon_{m+1}}^{\epsilon_m}*&((m_R^\mu f)\chi_R\mu))(x)\bigg|^2\lesssim\!\sum_{m\in\SSS_D(x)}\!\big|
(K\chi_{\epsilon_{m+1}}^{\epsilon_m}*((m_D^\mu f)\chi_{\wit D}\mu))(x)\big|^2\\
&\quad+\sum_{m\in\SSS_D(x)}\bigg|\sum_{R\in V(D)}
(K\chi_{\epsilon_{m+1}}^{\epsilon_m}*((m_R^\mu f-m_D^\mu f)\chi_R\mu))(x)\bigg|^2.
\end{split}
\end{equation}
We are going to estimate the two terms on the right hand side of (\ref{5 var quasiort eq1}) separately. For the second one, recall also that,  given $m\in\SSS_D(x)$, we have set $A_m(x):=A(x,\epsilon_{m+1},\epsilon_m)$. We write
\begin{equation*}
\begin{split}
|(K\chi_{\epsilon_{m+1}}^{\epsilon_m}*((m_R^\mu f-m_D^\mu f)\chi_R\mu))(x)|
&\leq|m_R^\mu f-m_D^\mu f|\int_{A_m(x)}|K(x-y)|\chi_R(y)\,d\mu(y)\\
&\lesssim|m_R^\mu f-m_D^\mu f|\,\mu(A_m(x)\cap R)\ell(D)^{-n}.
\end{split}
\end{equation*}
Therefore, interchanging the order of summation, 
\begin{equation*}
\begin{split}
\sum_{m\in\SSS_D(x)}&\bigg|\sum_{R\in V(D)}
(K\chi_{\epsilon_{m+1}}^{\epsilon_m}*((m_R^\mu f-m_D^\mu f)\chi_R\mu))(x)\bigg|^2\\
&\lesssim\bigg(\sum_{m\in\SSS_D(x)}\sum_{R\in V(D)}|m_R^\mu f-m_D^\mu f|\,\mu(A_m(x)\cap R)\ell(D)^{-n}\bigg)^2\\
&\leq\bigg(\sum_{R\in V(D)}|m_R^\mu f-m_D^\mu f|\,\frac{\mu(R)}{\ell(D)^{n}}\bigg)^2
\approx\bigg(\sum_{R\in V(D)}|m_R^\mu f-m_D^\mu f|\bigg)^2=(a_D(f))^2,
\end{split}
\end{equation*}
where $a_D(f)$ are the coefficients introduced in Lemma \ref{5 veins quasiortogonalitat}. If we integrate on $D$ and sum over all $D\in S$ and $S\in\Trees$, we can apply Lemma \ref{5 veins quasiortogonalitat}, and we finally obtain
\begin{equation}\label{5 var quasiort eq2}
\begin{split}
\sum_{S\in\Trees}\,\sum_{D\in S}\int_D\sum_{m\in\SSS_D(x)}\bigg|\sum_{R\in V(D)}
(K\chi_{\epsilon_{m+1}}^{\epsilon_m}*((m_R^\mu f&-m_D^\mu f)\chi_R\mu))(x)\bigg|^2d\mu(x)\\
&\lesssim\sum_{D\in\DD}(a_D(f))^2\mu(D)\lesssim\|f\|^2_{L^2(\mu)}.
\end{split}
\end{equation}

Let us estimate now the first term on the right hand side of (\ref{5 var quasiort eq1}).
Let $L_D$ be a minimizing $n$-plane for $\alpha_{\mu}(D)$ and let $L_D^x$ be the $n$-plane parallel to $L_D$ which contains $x$. 
Given $z\in\R^d$, let $p_0^x$ denote the orthogonal projection onto $L_D^x$. Let $g_1,g_2:\R\to[0,1]$ be such that $\supp g_1\subset(-2\varepsilon\ell(D),2\varepsilon\ell(D))$, $\supp g_2\subset(-\ell(D)\varepsilon,\ell(D)\varepsilon)^c$, and $g_1+g_2=1$, where $\varepsilon>0$ is some fixed constant small enough.  For $z\in\R^d$, consider the  projection onto $L_D^x$ given by 
\begin{equation}\label{5projeccio angular}
p^x(z):=\bigg(x+(p^x_0(z)-x)\frac{|z-x|}{|p^x_0(z)-x|}\bigg)g_2(|p^x_0(z)-x|)+p^x_0(z)g_1(|p^x_0(z)-x|).
\end{equation}
Since $\supp g_2$ does not contain the origin, $p^x$ is well defined. Moreover, if $z\in\R^d$ is such that $g_2(|p^x_0(z)-x|)=1$, then $|z-x|=|p^x(z)-x|$.

Let $C_*>0$ be a small constant which will be fixed below. Assume that $\alpha_\mu(10D)\geq C_*$. Then, we can easily estimate
\begin{equation}\label{5teorema L2 no suau 6'}
\begin{split}
\sum_{m\in\SSS_D(x)}\big|(K\chi_{\epsilon_{m+1}}^{\epsilon_m}*&((m_D^\mu f)\chi_{\wit D}\mu))(x)\big|^2
=|m_D^\mu f|^2\sum_{m\in\SSS_\DD(x)}\bigg|\int_{A_m(x)\cap\wit D}K(x-y)\,d\mu(y)\bigg|^2\\
&\lesssim|m_D^\mu f|^2\bigg(\sum_{m\in\SSS_\DD(x)}\int_{A_m(x)\cap\wit D}|K(x-y)|\,d\mu(y)\bigg)^2\\
&\lesssim|m_D^\mu f|^2\bigg(\int_{\wit D}\ell(D)^{-n}\,d\mu(y)\bigg)^2
\lesssim|m_D^\mu f|^2\lesssim|m_D^\mu f|^2\alpha_\mu(10D)^2.
\end{split}
\end{equation}

From now on, we assume that $\alpha_\mu(10D)<C_*$. By assuming $C_*$ small enough, it is not difficult to show that then
the distance between $\wit D$ and $L_D^x$ is smaller than $\ell(D)/1000$. Moreover, $p^x$ restricted to $\{y\in A_m(x):\, \dist(y,L_D^x)\leq\ell(D)/1000\}$ is a Lipschitz function with Lipschitz constant depending only $n$, $d$, and the AD regularity constant of $\mu$. Furthermore, by taking $\varepsilon$ small enough, we have 
\begin{equation}\label{mimimim}
p^x(z)=x+(p^x_0(z)-x)\frac{|z-x|}{|p^x_0(z)-x|}
\end{equation} 
for all $z\in\{y\in \wit D\cap A_m(x):\, \dist(y,L_D^x)\leq\ell(D)/1000\}\subset\supp\mu$.

Recall that $D\in S$ for some $S\in\Trees$. Let $Q_S$ be the maximal $\mu$-cube of $S$, and set $\nu_x:=p^x_\sharp(\chi_{40Q_S}\mu)$. Then, since $\supp\mu\cap A_m(x)\subset\wit D$ by the construction of $\wit D$,
\begin{equation}\label{5teorema L2 no suau 6}
\begin{split}
(K\chi_{\epsilon_{m+1}}^{\epsilon_m}*((&m_D^\mu f)\chi_{\wit D}\mu))(x)
=(m_D^\mu f)\int_{A_m(x)}K(x-y)\,d\mu(y)\\
&=(m_D^\mu f)\int_{A_m(x)}K(x-y)\,d(\mu-\nu_x)(y)
+(m_D^\mu f)\int_{A_m(x)}K(x-y)\,d\nu_x(y)\\
&=:U1_m(x)+U2_m(x).
\end{split}
\end{equation}

\begin{claim}\label{5teorema L2 no suau 18}
Under the notation above, we have
\begin{equation*}
\sum_{m\in\SSS_D(x)}|U1_m(x)|^2
\lesssim|m_D^\mu f|^2\bigg(\beta_{1,\mu}(D)^2+\alpha_\mu(D)^2+\bigg(\frac{\dist(x,L_D)}{\ell(D)}\bigg)^2\bigg).
\end{equation*}
\end{claim}
\begin{proof}[{\bf{\em Proof of }Claim \ref{5teorema L2 no suau 18}}]

By (\ref{mimimim}), $y\in A_m{(x)}$ if and only if $p^x(y)\in A_m(x)$ in the integral defining $U1_m(x)$. Since $|y-p^x(y)|\lesssim \dist(y,L_D^x)\leq\dist(y,L_D)+\dist(x,L_D)$ for all $y\in\sup\mu\cap A_m(x)$,
\begin{equation*}
\begin{split}
|U1_m(x)|&\leq|m_D^\mu f|\int_{A_m(x)}|K(x-y)-K(x-p^x(y))|\,d\mu(y)\\
&\lesssim\frac{|m_D^\mu f|}{\ell(D)^{n+1}}\int_{A_m(x)}|y-p^x(y)|\,d\mu(y)\\
&\lesssim\frac{|m_D^\mu f|}{\ell(D)^{n+1}}\int_{A_m(x)}(\dist(y,L_D)+\dist(x,L_D))\,d\mu(y).
\end{split}
\end{equation*}
If $L_D^1$ denotes a minimizing $n$-plane for $\beta_1(D)$, then $\dist_\HH(L_D\cap B_D,L_D^1\cap B_D)\lesssim\alpha_\mu(D)\ell(D)$, so $\dist(y,L_D)\lesssim\dist(y,L^1_D)+\alpha_\mu(D)\ell(D)$ for $y\in CD$ (see \cite{To}).
Therefore,
\begin{equation*}
\begin{split}
\sum_{m\in\SSS_D(x)}|U1_m(x)|^2
&\lesssim\bigg(\frac{|m_D^\mu f|}{\ell(D)^{n+1}}\sum_{m\in\SSS_D(x)}\int_{A_m(x)} (\dist(y,L_D)+\dist(x,L_D))\,d\mu(y)\bigg)^2\\
&\lesssim|m_D^\mu f|^2\bigg(\ell(D)^{-n-1}\int_{CD}(\dist(y,L_D)+\dist(x,L_D))\,d\mu(y)\bigg)^2\\
&\lesssim|m_D^\mu f|^2\bigg(\beta_{1,\mu}(D)^2+\alpha_\mu(D)^2+\bigg(\frac{\dist(x,L_D)}{\ell(D)}\bigg)^2\bigg).
\end{split}
\end{equation*}
\end{proof}

Let us consider $U2_m(x)$ now. We can assume that $\nu_x$ is absolutely continuous with respect to $\HH^n_{L_D^x}$ (for example, by convolving it with an approximation of the identity and making a limiting argument). Let $h_x$ be the corresponding density, so $\nu_x=h_x\HH^n_{L_D^x}$. We may also assume that $h_x\in L^2(\HH^n_{L_D^x})$.
So,
\begin{equation*}
\begin{split}
U2_m(x)&=(m_D^\mu f)\int_{A_m(x)}K(x-y)\,d\nu_x(y)=(m_D^\mu f)\int_{A_m(x)}K(x-y)h_x(y)\,d\HH^n_{L_D^x}(y).
\end{split}
\end{equation*}

At this point, we need to introduce a wavelet basis.
\begin{defi}\label{definicio wavelet}
Let $\DD^n$ denote the standard dyadic lattice of $\R^n$. Let $\{\psi_Q^k\}_{Q\in\DD^n,\,k=1,\ldots,2^n-1}$ be an orthonormal basis of $\CC^1$ wavelets on $\R^n$ in the following manner (see \cite[Part I]{David-LNM}):
\begin{enumerate}
\item[$(a)$] $\psi_Q^k:\R^n\to\R$ is a $\CC^1$ function for all $Q\in\DD^n$ and $k=1,\ldots,2^n-1$.
\item[$(b)$] There exists $C>1$ and $\psi_0:[0,C]^n\to\R$ with $\|\psi_0\|_2=1$, $\|\psi_0\|_\infty\lesssim1$, and such that, for any $Q\in\DD^n$ and $k=1,\ldots,2^n-1$, there exists $l\in\Z^n$ such that $\psi_Q^k(y)=\psi_0(y/\ell(Q)-l)\ell(Q)^{-n/2}$ for all $y\in\R^n$.
\item[$(c)$] $\|\psi_Q^k\|_2=1$, $\int\psi_Q^k\,d\LL^n=0$ and $\int\psi_Q^k\psi_R^l\,d\LL^n=0$, for all $Q,R\in\DD^n$ and $k,l=1,\ldots,2^n-1$ such that $(Q,k)\neq(R,l)$, where $\LL^n$ denotes the Lebesgue measure in $\R^n$.
\item[$(d)$] $\supp\psi_Q^k\subset C_wQ$ for all $Q\in\DD^n$ and $k=1,\ldots,2^n-1$, where $C_w>1$ is some fixed constant (which depends on $n$). In particular, for any $j\in\Z$ the supports of the functions in $\bigcup_{Q\in\DD^n:\, \ell(Q)=2^{-j}}\{\psi_Q^k\}_{k=1,\ldots,2^n-1}$ have finite overlap.
\item[$(e)$] $\|\psi_Q^k\|_\infty\lesssim\ell(Q)^{-n/2}$ and $\|\nabla\psi_Q^k\|_\infty\lesssim\ell(Q)^{-n/2-1}$ for all $Q\in\DD^n$, $k=1,\ldots,2^n-1$.
\item[$(f)$] If $h\in L^2(\LL^n)$, then $h=\sum_{Q\in\DD^n,\,k=1,\ldots,2^n-1}\Delta_Q^k h$, where $\Delta_Q^k h:=\big(\int h\psi_Q^k\,d\LL^n\big)\psi_Q^k.$
\end{enumerate}
\end{defi}

In order to reduce the notation, we may think that a cube of $\DD^n$ is not only a subset of $\R^n$, but a couple $(Q,k)$, where $Q$ is a subset of $\R^n$ and $k=1,\ldots,2^n-1$. In particular, there exist $2^n-1$ cubes in $\DD^n$ such that the subsets that they represent in $\R^n$ coincide. We make this abuse of notation to avoid using the superscript $k$ in the previous definition. Then, we can rewrite the wavelet basis as $\{\psi_Q\}_{Q\in\DD^n}$, with the evident adjustments of the properties $(a),\ldots,(f)$ in Definition \ref{definicio wavelet}.

Let $\DD^{n,0}_x$ be a fixed dyadic lattice of the $n$-plane $L_{D}^x$, and let $\{\psi_Q\}_{Q\in\DD^{n,0}_x}$ be a wavelet basis as the one introduced in Definition \ref{definicio wavelet} but defined on $L_{D}^x$. Denote by $E_D^x$ the $n$-dimensional vector space which defines $L_{D}^x$, and let $\{Q^0_k\}_{k\in\Z}$ be a fixed sequence of nested dyadic cubes in $E_D^x$ having the origin as a common vertex and such that $\ell(Q^0_k)=2^{-k}$ for all $k\in\Z$. Given $s\in E_D^x$, set 
$\DD^{n,s}_x:=\{s+Q:\,Q\in\DD^{n,0}_x\}$ (notice that, for any $k\in\Z$, the family $\{Q\in\DD^{n,s}_x:\, \ell(Q)=2^{-k}\}$ is periodic in the parameter $s$), 
 For any $Q\in\DD^{n,0}_x$ and $y\in L_{D}^x$, if $Q'=s+Q\in \DD^{n,s}_x$, we define $\psi_{Q'}(y)\equiv\psi_{s+Q}(y):=\psi_{Q}(y-s)$. Then $\{\psi_{Q'}\}_{Q'\in\DD^{n,s}_x}$ is also a wavelet basis defined on $L_{D}^x$.
Consider the decomposition of $h_x$ with respect to this basis, 
\begin{equation}\label{mklnko}
h_x=\sum_{Q\in\DD^{n,s}_x}\Delta^\psi_{Q}h_x=\sum_{Q\in\DD^{n,0}_x}\Delta^\psi_{Q,s}h_x,
\end{equation}
where $\Delta^\psi_{Q,s}h_x(z):=\big(\int h_x(y)\psi_Q(y-s)\,d\mu(y)\big)\,\psi_Q(z-s)$ (recall that, for any $Q\in\DD^{n,s}_x$, $\int\psi_{Q}\,d\HH^n_{L_D^x}=0$). 
We set $J(Q_S):=-\log_2(\ell(Q_S))$, and given $Q\in\DD^{n,s}_x$, we set $J(Q):=-\log_2(\ell(Q))$ and $J'(Q):=\max\{J(Q_S),J(Q)\}$.
Given $\Omega\subset E_D^x$, denote by $m_{s\in\Omega} g$ the average of a function $g:E_D^x\to\R$ over all $s\in\Omega$ and with respect to $\HH^n_{E_D^x}$. Then, by the periodicity of $\{\psi_Q\}_{Q\in\DD^{n,s}_x}$ in the parameter $s$ (recall Definition \ref{definicio wavelet}$(b)$) and (\ref{mklnko}), we can write
\begin{equation*}
h_x=m_{s\in Q^0_{J(Q_S)}} (h_x)=\sum_{Q\in\DD^{n,0}_x}m_{s\in Q^0_{J(Q_S)}}(\Delta^\psi_{Q,s}h_x)
=\sum_{Q\in\DD^{n,0}_x}m_{s\in Q^0_{J'(Q)}}(\Delta^\psi_{Q,s}h_x).
\end{equation*}

Set $J:=\{Q\in\DD_x^{n,0}\,:\,\supp\psi_{Q}(\cdot-s)\cap\supp\chi_{2^{-j-1}}^{\,2^{-j}}(x-\cdot)\neq\emptyset\text{ for some }s\in Q^0_{J'(Q)}\}$. Then,
\begin{equation}\label{chili}
U2_m(x)=(m_D^\mu f)\int_{A_m(x)}K(x-y)\sum_{Q\in J}m_{s\in Q^0_{J'(Q)}}(\Delta^\psi_{Q,s}h_x(y))\,d\HH^n_{L_D^x}(y).
\end{equation}
 
Recall that $D\in\DD_j$ and $m\in\SSS_D(x)$.
Since $x\in D$ and $\ell(D)=2^{-j}$, if $Q\in J$, then $D\subset B(x,C_a\ell(Q))$ or $Q\subset B(x,C_a\ell(D))$ for some constant $C_a>0$ big enough. In particular, if $\ell(Q)\gtrsim\ell(D)$ then $D\subset B(z_Q,C_a\ell(Q))$, and if $\ell(Q)\leq C\ell(D)$ with $C>0$ small enough then $Q\subset B(z_D,C_a\ell(D))$, where $z_Q$ denotes the center of $Q\subset L_D^x$ and $z_D$ denotes the center of $D\in\DD$.
We define 
\begin{equation*}
\begin{split}
J_1:&=\{Q\in J:\,\ell(Q)\leq C\ell(D)\}\subset\{Q\in\DD_x^{n,0}:\,Q\subset B(z_D,C_a\ell(D))\},\text{ and}\\
J_2:&=J\setminus J_1\subset\{Q\in\DD_x^{n,0}:\,D\subset  B(z_Q,C_a\ell(Q))\}.
\end{split}
\end{equation*}
Then, using (\ref{chili}), that  $\supp\chi_{\epsilon_{m+1}}^{\,\epsilon_m}(x-\cdot)\subset\supp\chi_{2^{-j-1}}^{\,2^{-j}}(x-\cdot)$ for all $m\in\SSS_D(x)$, that $\int_{A_m(x)}K(x-y)\,d\HH^n_{L_D^x}(y)=0$ by antisymmetry, and that $J'(Q)=J(Q)$ for all $Q\in J_1$ (because $D\subset Q_S$), if $x'$ denotes some fixed point in $A(x,2^{-j-1},2^{-j})\cap L_D^x$, we have
\begin{equation}\label{5teorema L2 no suau 16}
\begin{split}
U2_m&(x)=(m_D^\mu f)\int_{A_m(x)}K(x-y)\sum_{Q\in J_1}m_{s\in Q^0_{J(Q)}}(\Delta^\psi_{Q,s}h_x(y))\,d\HH^n_{L_D^x}(y)\\
&\quad+(m_D^\mu f)\int_{A_m(x)}K(x-y)\sum_{Q\in J_2}m_{s\in Q^0_{J'(Q)}}\big(\Delta^\psi_{Q,s}h_x(y)
-\Delta^\psi_{Q,s}h_x(x')\big)\,d\HH^n_{L_D^x}(y)\\
&=:U3_m(x)+U4_m(x).
\end{split}
\end{equation}

\begin{claim}\label{5teorema L2 no suau 20}
Under the notation above, we have
\begin{equation*}
\sum_{m\in\SSS_D(x)}|U4_m(x)|^2
\lesssim|m_D^\mu f|^2\sum_{Q\in J_2}\bigg(\frac{\ell(D)}{\ell(Q)}\bigg)^{1/2}\ell(Q)^{-n}\big(m_{s\in Q^0_{J'(Q)}}\|\Delta^\psi_{Q,s}h_x\|_2\big)^2.
\end{equation*}
\end{claim}
\begin{proof}[{\bf{\em Proof of }Claim \ref{5teorema L2 no suau 20}}]
By property $(e)$ of the wavelet basis in Definition \ref{definicio wavelet}, we have $|\Delta^\psi_{Q,s}h_x(y)-\Delta^\psi_{Q,s}h_x(x')|\leq\|\nabla(\Delta^\psi_{Q,s}h_x)\|_\infty|x'-y|\lesssim\|\Delta^\psi_{Q,s}h_x\|_2|x'-y|\ell(Q)^{-n/2-1}$. Moreover, if $y\in A_m(x)$, then $|x'-y|\lesssim\ell(D)$. Therefore, 
\begin{equation*}
\begin{split}
|U4_m(x)|&\leq\sum_{Q\in J_2}|m_D^\mu f|\int_{A_m(x)}|K(x-y)|m_{s\in Q^0_{J'(Q)}}\big(\big|\Delta^\psi_{Q,s}h_x(y)-\Delta^\psi_{Q,s}h_x(x')\big|\big)\,d\HH^n_{L_D^x}(y)\\
&\lesssim\sum_{Q\in J_2}|m_D^\mu f|m_{s\in Q^0_{J'(Q)}}\big(\|\Delta^\psi_{Q,s}h_x\|_2\big)\ell(D)^{1-n}\ell(Q)^{-n/2-1}\HH^n_{L_D^x}({A_m(x)}),
\end{split}
\end{equation*}
and then, by Cauchy-Schwarz inequality and since $J_2\subset\{Q\in\DD_x^{n,0}:\,D\subset  B(z_Q,C_a\ell(Q))\}$ (in particular, $\ell(D)/\ell(Q)\lesssim(\ell(D)/\ell(Q))^{1/2}$),
\begin{equation*}
\begin{split}
\sum_{m\in\SSS_D(x)}|U4_m(x)|^2
&\lesssim\bigg(\sum_{m\in\SSS_D(x)}\sum_{Q\in J_2}|m_D^\mu f|\,
m_{s\in Q^0_{J'(Q)}}\big(\|\Delta^\psi_{Q,s}h_x\|_2\big)\frac{\ell(Q)^{n/2+1}}{\ell(D)^{n-1}}\,\HH^n_{L_D^x}({A_m(x)})\bigg)^2\\
&\leq\bigg(\sum_{Q\in J_2}|m_D^\mu f|\,m_{s\in Q^0_{J'(Q)}}\big(\|\Delta^\psi_{Q,s}h_x\|_2\big)\ell(D)\ell(Q)^{n/2+1}\bigg)^2\\
&\leq\bigg(\sum_{Q\in J_2}\frac{\ell(D)}{\ell(Q)}\bigg)\bigg(\sum_{Q\in J_2}|m_D^\mu f|^2\big(m_{s\in Q^0_{J'(Q)}}\|\Delta^\psi_{Q,s}h_x\|_2\big)^2\frac{\ell(D)}{\ell(Q)^{n+1}}\bigg)\\
&\lesssim|m_D^\mu f|^2\sum_{Q\in J_2}\bigg(\frac{\ell(D)}{\ell(Q)}\bigg)^{1/2}\ell(Q)^{-n}\big(m_{s\in Q^0_{J'(Q)}}\|\Delta^\psi_{Q,s}h_x\|_2\big)^2.
\end{split}
\end{equation*}
\end{proof}

We are going to estimate $U3_m(x)$ with techniques very similar to the ones used in Subsections \ref{5 ss5321} and \ref{5 ss5321b}. First of all, let $b_*>0$ be a small constant which will be fixed later on, and consider the family $\PP:=\{Q\in\DD^{n,0}_x:\,\ell(Q)\leq \ell(D)\}$. Let $\Stop$ denote the set of cubes $Q\in \PP$ such that there exists $R_Q\in\DD$ with $\ell(R_Q)=\ell(Q)$, $10R_Q\cap(p^x)^{-1}(\supp\psi_Q)\neq\emptyset$, and 
\begin{equation}\label{reduced lema1}
\sum_{R\in\DD:\,R_Q\subset R,\,\ell(R)\leq\ell(D)}\alpha_{\mu}(10R)\geq b_*\quad\text{but}\quad\sum_{R\in\DD:\,P(R_Q)\subset R,\,\ell(R)\leq\ell(D)}\alpha_{\mu}(10R)<b_*.
\end{equation}
Observe that if $Q$ and $Q'$ are different and belong to $\Stop$, then $Q\cap Q'=\emptyset$. Notice also that $D\not\in\Stop$ because we assumed $\alpha_\mu(10D)<C_*$. Finally, denote by $\Tre$ the set of cubes $Q\in\PP\setminus\Stop$ such that $R\not\in\Stop$ for all $R\in\PP$ with $R\supset Q$. Then $\PP=\Tre\cup\bigcup_{Q\in\Stop}\{R\in\PP:\,R\subset Q\}$. By taking $C_*$ small enough we can assume that, if $R\in J_1\cap\PP$ and $R\subset Q$ for some $Q\in\Stop$, then $Q\in J_1$. So we write
\begin{equation*}
\begin{split}
\sum_{Q\in J_1}m_{s\in Q^0_{J(Q)}}(&\Delta^\psi_{Q,s}h_x)\\
&=\sum_{Q\in J_1\cap\Tre}m_{s\in Q^0_{J(Q)}}(\Delta^\psi_{Q,s}h_x)
+\sum_{Q\in J_1\cap\Stop}\,\sum_{R\in J_1\cap\PP:\,R\subset Q}m_{s\in Q^0_{J(Q)}}(\Delta^\psi_{R,s}h_x)
\end{split}
\end{equation*}
Set $\wit\Delta^\psi_{Q,s}h_x:=\sum_{R\in\PP:\,R\subset Q}\Delta^\psi_{R,s}h_x$. Then, using the definition of $J_1$ and $J$, we can split 
\begin{equation}\label{ernwvijf}
\begin{split}
U3_m(x)&=(m_D^\mu f)\int_{A_m(x)}K(x-y)\sum_{Q\in J_1\cap\Tre}m_{s\in Q^0_{J(Q)}}(\Delta^\psi_{Q,s}h_x(y))\,d\HH^n_{L_D^x}(y)\\
&\quad+(m_D^\mu f)\int_{A_m(x)}K(x-y)\sum_{Q\in J_1\cap\Stop}m_{s\in Q^0_{J(Q)}}(\wit\Delta^\psi_{Q,s}h_x(y))\,d\HH^n_{L_D^x}(y)\\
&=:U3_m^{a}(x)+U3_m^{b}(x).
\end{split}
\end{equation}

\begin{claim}\label{5teorema L2 no suau 19}
Under the notation above, we have
\begin{equation*}
\sum_{m\in\SSS_D(x)}|U3^{a}_m(x)|^2\lesssim|m_D^\mu f|^2\sum_{Q\in J_1\cap\Tre}\bigg(\frac{\ell(Q)}{\ell(D)}\bigg)^{1/2}\|m_{s\in Q^0_{J(Q)}}(\Delta^\psi_{Q,s}h_x)\|_2^2\,\ell(D)^{-n}.
\end{equation*}
For simplicity of notation, we have set $\|\cdot\|_p:=\|\cdot\|_{L^p(\HH^n_{L_D^x})}$.
\end{claim}
\begin{proof}[{\bf{\em Proof of }Claim \ref{5teorema L2 no suau 19}}]
Notice that $\HH^n_{L_D^x}(A_m(x))
\lesssim(\epsilon_m-\epsilon_{m+1})\ell(D)^{n-1}$. Moreover, the function $m_{s\in Q^0_{J(Q)}}(\Delta^\psi_{Q,s}h_x)$ is supported in $CQ$ and has vanishing integral, because the same holds for each $\Delta^\psi_{Q,s}h_x$ with $s\in Q^0_{J(Q)}$.
Hence, 
the sum $\sum_{m\in\SSS_D(x)}|U3^{a}_m(x)|^2$ can be estimated using arguments very similar to the ones in Subsection \ref{5 ss5321} (see (\ref{5 var eq7})), and the analogues of Lemma \ref{lema pendent petita3} and Claims \ref{5 claim1} and \ref{5 claim2} for $\HH^n_{L_D^x}$ follow easily. One obtains the expected estimate.
\end{proof}

\begin{claim}\label{5teorema L2 no suau 19b}
Under the notation above, we have
\begin{equation*}
\sum_{m\in\SSS_D(x)}|U3^{b}_m(x)|^2\lesssim|m_D^\mu f|^2\sum_{Q\in J_1\cap\Stop}\bigg(\frac{\ell(Q)}{\ell(D)}\bigg)^{1/2}\frac{\|m_{s\in Q^0_{J(Q)}}(\wit\Delta^\psi_{Q,s}h_x)\|_1^2}{\ell(D)^{n}\ell(Q)^{n}}.
\end{equation*}
\end{claim}
\begin{proof}[{\bf{\em Proof of }Claim \ref{5teorema L2 no suau 19b}}]
Since $m_{s\in Q^0_{J(Q)}}(\wit\Delta^\psi_{Q,s}h_x)$ has vanishing integral and it is supported in a neighbourhood of $Q$, the term $U3^{b}_m(x)$ can be estimated in the same manner (but now we do not use the estimate $\|m_{s\in Q^0_{J(Q)}}(\wit\Delta^\psi_{Q,s}h_x)\|_1^2\lesssim\ell(Q)^{n}\|m_{s\in Q^0_{J(Q)}}(\wit\Delta^\psi_{Q,s}h_x)\|_2^2$), and one obtains the expected estimate (compare with (\ref{5 var eq9})).
\end{proof}

Recall that we have fixed $x\in D\in S\in\Trees$, and we denote by $Q_S$ the maximal $\mu$-cube in $S$ from the corona decomposition, so $D\subset Q_S$. The following lemma, whose proof is given in Subsection \ref{oooo}, yields the suitable estimates for $m_{s\in Q^0_{J'(Q)}}(\Delta^\psi_{Q,s}h_x)$ and $m_{s\in Q^0_{J'(Q)}}(\wit\Delta^\psi_{Q,s}h_x)$.

\begin{lema}\label{5lema L2}
Assume that $\alpha_\mu(D)< C_*$, for  some constant $C_*>0$ small enough. Given $Q\in\DD_x^{n,0}$, there exists constants $C_1,C_2>1$ depending on $C_*$ and $b_*$ (see $(\ref{reduced lema1})$)such that,
\begin{enumerate}
\item[$(a)$]if $Q\in J_2$ and $\ell(Q)>\ell(Q_S)$, then
$m_{s\in Q^0_{J'(Q)}}(\|\Delta^\psi_{Q,s}h_x\|_2)\lesssim\ell(Q_S)^n\ell(Q)^{-n/2}$,
\item[$(b)$]if $Q\in J_2$ and $\ell(Q)\leq\ell(Q_S)$, then
$$m_{s\in Q^0_{J'(Q)}}\big(\|\Delta^\psi_{Q,s}h_x\|_2\big)\lesssim\bigg(\sum_{R\in\DD:\,D\subset R\subset  B(z_Q,C_1\ell(Q))}\alpha_{\mu}(C_1R)+\frac{\dist(x,L_D)}{\ell(D)}\bigg)\ell(Q)^{n/2},$$
\item[$(c)$] if $Q\in J_1\cap\Tre$, then there exists  $Q_0\equiv Q_0(x,Q)\in\DD$ depending on $x$ and $Q\in\DD_x^{n,0}$ such that $Q_0\subset C_2D$, $\ell(Q_0)\approx\ell(Q)$, $Q_0\cap(p^x)^{-1}(\supp\psi_Q)\neq\emptyset$ and 
$$\|m_{s\in Q^0_{J(Q)}}(\Delta^\psi_{Q,s}h_x)\|_2\lesssim\bigg(\sum_{R\in\DD:\,Q_0\subset R\subset C_2D}\alpha_{\mu}(C_2R)+\frac{\dist(x,L_D)}{\ell(D)}\bigg)\ell(Q)^{n/2},\quad\text{and}$$
\item[$(d)$] if $Q\in J_1\cap\Stop$,  then $\|m_{s\in Q^0_{J(Q)}}(\wit\Delta^\psi_{Q,s}h_x)\|_1\lesssim\ell(Q)^n.$
\end{enumerate}
\end{lema}

We are ready to put all the estimates together to bound the first term on the right hand side of (\ref{5 var quasiort eq1}). From (\ref{5teorema L2 no suau 6'}), (\ref{5teorema L2 no suau 6}), (\ref{5teorema L2 no suau 16}), and (\ref{ernwvijf}) we have
\begin{equation}\label{5estimate ss8.2}
\begin{split}
\sum_{m\in\SSS_D(x)}&\big|
(K\chi_{\epsilon_{m+1}}^{\epsilon_m}*((m_D^\mu f)\chi_{\wit D}\mu))(x)\big|^2
\lesssim|m_D^\mu f|^2\alpha_\mu(10D)^2\\
&\quad+\sum_{m\in\SSS_D(x)}(|U1_m(x)|^2+|U3^a_m(x)|^2+|U3^b_m(x)|^2+|U4_m(x)|^2).
\end{split}
\end{equation}

Let us deal with $U1_m(x)$ (the term $|m_D^\mu f|^2\alpha_\mu(10D)^2$ above is handled in the same manner). If $L_D^1$ and $L_D^2$ denote a minimizing $n$-plane for $\beta_{1,\mu}(D)$ and $\beta_{2,\mu}(D)$, respectively, one can show that
$\dist_\HH(L_D\cap B_D,L_D^1\cap B_D)\lesssim\alpha_\mu(D)\ell(D)$ and $\dist_\HH(L_D^1\cap B_D,L_D^2\cap B_D)\lesssim\beta_{2,\mu}(D)\ell(D)$, so we have $\dist(x,L_D)\lesssim\dist(x,L^2_D)+\beta_{2,\mu}(D)\ell(D)+\alpha_\mu(D)\ell(D)$ for $x\in D$.
Then, by Claim \ref{5teorema L2 no suau 18} and Carleson's embedding theorem,
\begin{equation}\label{est u1}
\begin{split}
\sum_{S\in\Trees}\sum_{D\in S}&\int_D\sum_{m\in\SSS_D(x)}|U1_m|^2\,d\mu\\
&\lesssim\sum_{D\in\DD}\int_D
|m_D^\mu f|^2\bigg(\beta_{1,\mu}(D)^2+\alpha_\mu(D)^2+\bigg(\frac{\dist(x,L_D)}{\ell(D)}\bigg)^2\bigg)\,d\mu(x)\\
&\lesssim\sum_{D\in\DD}
|m_D^\mu f|^2\ell(D)^n\big(\beta_{1,\mu}(D)^2+\alpha_\mu(D)^2+\beta_{2,\mu}(D)^2\big)\lesssim\|f\|_{L^2(\mu)}^2.
\end{split}
\end{equation}

For the case of $U3^a_m (x)$, by Claim \ref{5teorema L2 no suau 19} and Lemma \ref{5lema L2}$(c)$ applied to the $\mu$-cubes in $J_1\cap\Tre$, we have
\begin{equation*}
\begin{split}
\sum_{S\in\Trees}\sum_{D\in S}&\int_D\sum_{m\in\SSS_D(x)}|U3_m^a|^2\,d\mu\\
&\lesssim\sum_{D\in\DD}|m_D^\mu f|^2\int_D\sum_{Q\in J_1\cap\Tre}\bigg(\frac{\ell(Q)}{\ell(D)}\bigg)^{1/2}\frac{\|m_{s\in Q^0_{J(Q)}}(\Delta^\psi_{Q,s}h_x)\|^2_2}{\ell(D)^{n}}\,d\mu(x)\\
&\lesssim\sum_{D\in\DD}|m_D^\mu f|^2\int_D\sum_{Q\in J_1\cap\Tre}\bigg(\frac{\ell(Q)}{\ell(D)}\bigg)^{n+1/2}\bigg(\sum_{\begin{subarray}{c}R\in\DD:\\Q_0(x,Q)\subset R\subset C_2D\end{subarray}}\alpha_{\mu}(C_2R)\bigg)^2\,d\mu(x)\\
&\quad+\sum_{D\in\DD}|m_D^\mu f|^2\int_D\sum_{Q\in J_1\cap\Tre}\bigg(\frac{\ell(Q)}{\ell(D)}\bigg)^{n+1/2}\bigg(\frac{\dist(x,L_D)}{\ell(D)}\bigg)^2\,d\mu(x)=:S_1+S_2.
\end{split}
\end{equation*}
Recall that $J_1\subset\{Q\in\DD_x^{n,0}:\,Q\subset  B(z_D,C_a\ell(D))\}$. Then $\sum_{Q\in J_1}(\ell(Q)/\ell(D))^{n+1/2}\lesssim1$, and since $\dist(x,L_D)\lesssim\dist(x,L^2_D)+\beta_{2,\mu}(D)\ell(D)+\alpha_\mu(D)\ell(D)$ for $x\in D$, then $S_2\lesssim\sum_{D\in\DD}|m_D^\mu f|^2(\beta_{2,\mu}(D)^2+\alpha_\mu(D)^2)\ell(D)^n,$ and hence $S_2\leq C\|f\|_{L^2(\mu)}^2$, by Carleson's embedding theorem. For $S_1$, since $\ell(Q)\approx\ell(Q_0(x,Q))$ (recall the definition of $Q_0\equiv Q_0(x,Q)$ in Lemma \ref{5lema L2}$(c)$), $Q_0(x,Q)\subset C_2D$, and every $Q_0\in\DD$ intersects $(p^x)^{-1}(\supp\psi_Q)$ for finitely many cubes $Q\in\DD_x^{n,0}$ (with a bound for the number of such cubes $Q$ independent of $x$ and $Q_0$), we have
\begin{equation*}
\begin{split}
\sum_{Q\in J_1\cap\Tre}&\bigg(\frac{\ell(Q)}{\ell(D)}\bigg)^{n+1/2}
\bigg(\sum_{R\in\DD:\,Q_0(x,Q)\subset R\subset C_2D}\alpha_{\mu}(C_2R)\bigg)^2\\
&=\sum_{P\in\DD\,:\,P\subset C_2D}\,\sum_{\begin{subarray}{c}Q\in\DD_x^{n,0}:\,Q\subset B(z_D,C_a\ell(D)),\\Q_0(x,Q)=P\end{subarray}}\bigg(\frac{\ell(Q)}{\ell(D)}\bigg)^{n+1/2}
\bigg(\sum_{R\in\DD:\,P\subset R\subset C_2D}\alpha_{\mu}(C_2R)\bigg)^2\\
&\lesssim\sum_{P\in\DD\,:\,P\subset C_2D}\bigg(\frac{\ell(P)}{\ell(D)}\bigg)^{n+1/2}
\bigg(\sum_{R\in\DD:\,P\subset R\subset C_2D}\alpha_{\mu}(C_2R)\bigg)^2.
\end{split}
\end{equation*}
By Cauchy-Schwarz inequality,
\begin{equation}\label{fnjid}
\begin{split}
\sum_{P\in\DD\,:\,P\subset C_2D}&\bigg(\frac{\ell(P)}{\ell(D)}\bigg)^{n+1/2}
\bigg(\sum_{R\in\DD:\,P\subset R\subset C_2D}\alpha_{\mu}(C_2R)\bigg)^2\\
&\lesssim\sum_{P\in\DD\,:\,P\subset C_2D}\bigg(\frac{\ell(P)}{\ell(D)}\bigg)^{n+1/2}
\log_2\bigg(\frac{\ell(D)}{\ell(P)}\bigg)\sum_{R\in\DD:\,P\subset R\subset C_2D}\alpha_{\mu}(C_2R)^2\\
&\lesssim\sum_{R\in\DD:\,R\subset C_2D}\alpha_{\mu}(C_2R)^2\sum_{P\in\DD\,:\,P\subset R}\bigg(\frac{\ell(P)}{\ell(D)}\bigg)^{n+1/4}\\
&\lesssim\sum_{R\in\DD:\,R\subset C_2D}\alpha_{\mu}(C_2R)^2\bigg(\frac{\ell(R)}{\ell(D)}\bigg)^{n+1/4}=:\lambda_1(D)^2.
\end{split}
\end{equation}

By standard arguments one can easily show that 
these $\lambda_1$ coefficients satisfy a Carleson packing condition, so by (\ref{fnjid}) and Carleson's embedding theorem we obtain
$S_1\lesssim\sum_{D\in\DD}|m_D^\mu f|^2\ell(D)^n\lambda_1(D)^2\lesssim\|f\|_{L^2(\mu)}^2,$ which combined with $S_2\lesssim\|f\|_{L^2(\mu)}^2$ yields 
\begin{equation}\label{est u3a}
\sum_{S\in\Trees}\sum_{D\in S}\int_D\sum_{m\in\SSS_D(x)}|U3^a_m|^2\,d\mu\lesssim\|f\|_{L^2(\mu)}^2.
\end{equation}

Let us deal now with $U3^b_m$. By Claim \ref{5teorema L2 no suau 19b} and Lemma \ref{5lema L2}$(d)$ applied to the $\mu$-cubes in $J_1\cap\Stop$, we have
\begin{equation*}
\begin{split}
\sum_{S\in\Trees}\sum_{D\in S}&\int_D\sum_{m\in\SSS_D(x)}|U3_m^b|^2\,d\mu\\
&\lesssim\sum_{D\in\DD}|m_D^\mu f|^2\int_D\sum_{Q\in J_1\cap\Stop}\bigg(\frac{\ell(Q)}{\ell(D)}\bigg)^{1/2}\frac{\|m_{s\in Q^0_{J(Q)}}(\wit\Delta^\psi_{Q,s}h_x)\|^2_1}{\ell(D)^{n}\ell(Q)^{n}}\,d\mu(x)\\
&\lesssim\sum_{D\in\DD}|m_D^\mu f|^2\int_D\sum_{Q\in J_1\cap\Stop}\bigg(\frac{\ell(Q)}{\ell(D)}\bigg)^{n+1/2}\,d\mu.\\
\end{split}
\end{equation*}
Given $D\in\DD$, consider the family $\Lambda_D:=\{R\in\DD:\, R=R_Q \text{ for some }x\in D \text{ and some }Q\in J_1\cap\Stop\}$ (see the definition of $R_Q$ in (\ref{reduced lema1})). Observe that every $R\in\DD$ intersects $(p^x)^{-1}(Q\cap L_D^x)$ for finitely many $\mu$-cubes $Q\in\DD_x^{n,0}$ such that $\ell(Q)=\ell(R)$. Thus, simlilarly to what we did for $Q\in J_1\cap\Tre$ in the case of $U3^a_m$, we have 
\begin{equation*}
\begin{split}
\sum_{D\in\DD}|m_D^\mu f|^2\int_D\sum_{Q\in J_1\cap\Stop}\bigg(\frac{\ell(Q)}{\ell(D)}\bigg)^{n+1/2}\,d\mu
\lesssim\sum_{D\in\DD}|m_D^\mu f|^2\int_D\sum_{R\in\Lambda_D}\bigg(\frac{\ell(R)}{\ell(D)}\bigg)^{n+1/2}\,d\mu\\
\lesssim\sum_{D\in\DD}|m_D^\mu f|^2\sum_{R\in\Lambda_D}\bigg(\frac{\ell(R)}{\ell(D)}\bigg)^{n+1/2}\mu(D)=\sum_{D\in\DD}|m_D^\mu f|^2\lambda_2(D)^2\mu(D),
\end{split}
\end{equation*}
where we have set $\lambda_2(D)^2:=\sum_{R\in\Lambda_D}(\ell(R)/\ell(D))^{n+1/2}$.
Since the $\alpha_\mu$'s satisfy a Carleson packing condition, it is not hard to show that the same holds for the $\lambda_2$'s. Indeed, since for any $R\in\Lambda_D$ we have
$\sum_{R'\in\DD:\,R\subset R',\,\ell(R)\leq\ell(D)}\alpha_{\mu}(10R')\geq b_*$ by (\ref{reduced lema1}), then $$\lambda_2(D)^2\leq b_*^{-2}\sum_{R\in\Lambda_D}\bigg(\frac{\ell(R)}{\ell(D)}\bigg)^{n+1/2}\bigg(\sum_{R'\in\DD:\,R\subset R',\,\ell(R)\leq\ell(D)}\alpha_{\mu}(10R')\bigg)^2,$$
and we can proceed as in (\ref{fnjid}). Hence, putting these estimates together and using Carleson's embedding theorem for the $\lambda_2$'s, we obtain
\begin{equation}\label{est u3b}
\sum_{S\in\Trees}\sum_{D\in S}\int_D\sum_{m\in\SSS_D(x)}|U3_m^b|^2\,d\mu
\lesssim\|f\|_{L^2(\mu)}^2.
\end{equation}

We deal now with $U4_m(x)$. By Claim \ref{5teorema L2 no suau 20} and Lemma \ref{5lema L2}$(a)$ and $(b)$ applied to the cubes in $J_2$, 
\begin{equation}\label{qwqe}
\begin{split}
&\sum_{S\in\Trees}\sum_{D\in S}\int_D\sum_{m\in\SSS_D(x)}|U4_m|^2\,d\mu\\
&\quad\lesssim\sum_{S\in\Trees}\sum_{D\in S}|m_D^\mu f|^2\int_D\sum_{Q\in J_2}\bigg(\frac{\ell(D)}{\ell(Q)}\bigg)^{1/2}\,\frac{m_{s\in Q^0_{J'(Q)}}\big(\|\Delta^\psi_{Q,s}h_x\|_2\big)^2}{\ell(Q)^{n}}\,d\mu\\
&\quad\lesssim\sum_{S\in\Trees}\sum_{D\in S}|m_D^\mu f|^2\int_D\sum_{Q\in J_2:\, \ell(Q)\leq\ell(Q_S)}\bigg(\frac{\ell(D)}{\ell(Q)}\bigg)^{1/2}\\
&\quad\qquad\qquad\qquad\qquad\quad\quad\bigg[\bigg(\sum_{R\in\DD:\,D\subset R\subset  B(z_Q,C_1\ell(Q))}\alpha_{\mu}(C_1R)\bigg)^2+\bigg(\frac{\dist(x,L_D)}{\ell(D)}\bigg)^2\bigg]\,d\mu\\
&\quad\quad+\sum_{S\in\Trees}\sum_{D\in S}|m_D^\mu f|^2\int_D\sum_{Q\in J_2:\,\ell(Q)>\ell(Q_S)}\bigg(\frac{\ell(D)}{\ell(Q)}\bigg)^{1/2}\frac{\ell(Q_S)^{2n}}{\ell(Q)^{2n}}\,d\mu=:S_3+S_4.
\end{split}
\end{equation}
Regarding $S_3$, since $\dist(x,L_D)\lesssim\dist(x,L^2_D)+\beta_{2,\mu}(D)\ell(D)+\alpha_\mu(D)\ell(D)$ for $x\in D$ and $\sum_{Q\in J_2}(\ell(D)/\ell(Q))^{1/2}\lesssim1$, the second term in the definition of  $S_3$ is bounded by $\sum_{D\in\DD}|m_D^\mu f|^2(\beta_{2,\mu}(D)^2+\alpha_\mu(D)^2)\ell(D)^n,$ and hence by $C\|f\|_{L^2(\mu)}^2$, by Carleson's embedding theorem. For the first term in $S_3$, by Cauchy-Schwarz inequality,
\begin{equation*}
\begin{split}
&\sum_{S\in\Trees}\sum_{D\in S}|m_D^\mu f|^2\int_D\sum_{Q\in J_2:\, \ell(Q)\leq\ell(Q_S)}\bigg(\frac{\ell(D)}{\ell(Q)}\bigg)^{1/2}\bigg(\sum_{\begin{subarray}{c}R\in\DD:\\D\subset R\subset  B(z_Q,C_1\ell(Q))\end{subarray}}\alpha_{\mu}(C_1R)\bigg)^2\,d\mu\\
&\quad\lesssim\sum_{S\in\Trees}\sum_{D\in S}|m_D^\mu f|^2\int_D\sum_{\begin{subarray}{c}Q\in J_2:\\ \ell(Q)\leq\ell(Q_S)\end{subarray}}\bigg(\frac{\ell(D)}{\ell(Q)}\bigg)^{1/2}\log_2\bigg(\frac{\ell(Q)}{\ell(D)}\bigg)\!\!\sum_{\begin{subarray}{c}R\in \DD:\\D\subset R\subset  B(z_Q,C_1\ell(Q))\end{subarray}}\alpha_{\mu}(C_1R)^2\,d\mu\\
&\quad\lesssim\sum_{D\in\DD}|m_D^\mu f|^2\int_D\sum_{\begin{subarray}{c}R\in\DD:\\D\subset R\end{subarray}}\alpha_{\mu}(C_1R)^2\sum_{\begin{subarray}{c}Q\in\DD_x^{n,0}:\\R\subset  B(z_Q,C_1\ell(Q))\end{subarray}}\bigg(\frac{\ell(D)}{\ell(Q)}\bigg)^{1/4}\,d\mu.
\end{split}
\end{equation*}
Notice that $\sum_{Q\in\DD_x^{n,0}:\,R\subset  B(z_Q,C_1\ell(Q))}\big(\ell(D)/\ell(Q)\big)^{1/4}\lesssim\big(\ell(D)/\ell(R)\big)^{1/4}$, thus the right side of the preceeding inequality is bounded above by
\begin{equation}\label{bncdenacl}
\sum_{D\in\DD}|m_D^\mu f|^2\ell(D)^n\sum_{R\in \DD:\,D\subset R}\alpha_{\mu}(C_1R)^2\bigg(\frac{\ell(D)}{\ell(R)}\bigg)^{1/4}=:\sum_{D\in\DD}|m_D^\mu f|^2\ell(D)^n\lambda_3(D)^2.
\end{equation}
By standard arguments one can show that the $\lambda_3$'s satisfy a Carleson packing condition, so by Carleson's embedding theorem again, the last term in (\ref{bncdenacl}) is bounded by $C\|f\|_{L^2(\mu)}^2$. Thus we obtain $S_3\lesssim\|f\|_{L^2(\mu)}^2$.

The estimate of $S_4$ from (\ref{qwqe}) is easier:
\begin{equation*}
\begin{split}
S_4\lesssim\sum_{S\in\Trees}\sum_{D\in S}|m_D^\mu f|^2\int_D\sum_{\begin{subarray}{c}Q\in\DD_x^{n,0}:\,\ell(Q)>\ell(Q_S),\\D\subset  B(z_Q,C_1\ell(Q))\end{subarray}}\frac{\ell(D)^{1/2}\ell(Q_S)^{2n}}{\ell(Q)^{2n+1/2}}\,d\mu(x).
\end{split}
\end{equation*}
As before, $\sum_{Q\in\DD_x^{n,0}:\,\ell(Q)>\ell(Q_S),\,
D\subset  B(z_Q,C_1\ell(Q))}\ell(Q)^{-2n-1/2}\lesssim\ell(Q_S)^{-2n-1/2}$, thus 
\begin{equation*}
\begin{split}
S_4&\lesssim\sum_{S\in\Trees}\sum_{D\in S}|m_D^\mu f|^2\ell(D)^{n}\bigg(\frac{\ell(D)}{\ell(Q_S)}\bigg)^{1/2}\lesssim\sum_{D\in\DD}|m_D^\mu f|^2\ell(D)^{n}\sum_{S\in\Trees:\,S\ni D}\bigg(\frac{\ell(D)}{\ell(Q_S)}\bigg)^{1/2}\\
&=:\sum_{D\in\DD}|m_D^\mu f|^2\ell(D)^{n}\lambda_4(D)^2.
\end{split}
\end{equation*}
Similarly to the case of the $\lambda_3$ coefficients, one can show that the $\lambda_4$'s also satisfy a Carleson packing condition, thus 
$S_4\lesssim\|f\|_{L^2(\mu)}^2$ by Carleson's embedding theorem. Actually, if one defines $\widehat\alpha_\mu(Q)=1$ if $Q=Q_S$ for some $S\in\Trees$ and $\widehat\alpha_\mu(Q)=0$ otherwise, using the packing condition for the $\mu$-cubes $Q_S$ with $S\in\Trees$, one can easily verify that the $\widehat\alpha_\mu$'s  satisfy a Carleson packing condition. Then, 
$$\lambda_4(D)^2=\sum_{S\in\Trees:\,D\subset Q_S}\bigg(\frac{\ell(D)}{\ell(Q_S)}\bigg)^{1/2}\widehat\alpha_\mu(Q_S)^2=\sum_{Q\in\DD:\,D\subset Q}\bigg(\frac{\ell(D)}{\ell(Q)}\bigg)^{1/2}\widehat\alpha_\mu(Q)^2,$$
and we can argue as in the case of the $\lambda_3$'s in (\ref{bncdenacl}).

By the estimates of $S_3$ and $S_4$, we obtain 
\begin{equation}\label{est u4}
\sum_{S\in\Trees}\sum_{D\in S}\int_D\sum_{m\in\SSS_D(x)}|U4_m|^2\,d\mu\lesssim\|f\|_{L^2(\mu)}^2.
\end{equation}

Finally, plugging (\ref{est u1}), (\ref{est u3a}), (\ref{est u3b}), and (\ref{est u4}) in (\ref{5estimate ss8.2}), and combining the result with (\ref{5 var quasiort eq1}) and (\ref{5 var quasiort eq2}), we conclude that
\begin{equation*}
\begin{split}
\sum_{S\in\Trees}\sum_{D\in S}\int_D\sum_{m\in\SSS_D(x)}\bigg|\sum_{R\in V(D)}
(K\chi_{\epsilon_{m+1}}^{\epsilon_m}*((m_R^\mu f)\chi_R\mu))(x)\bigg|^2\,d\mu(x)\lesssim\|f\|_{L^2(\mu)}^2,
\end{split}
\end{equation*}
and Lemma \ref{5 var eq8a} is finally proved, except for Lemma \ref{5lema L2}.
\end{proof}

\subsubsection{{\bf Proof of Lemma \ref{5lema L2}}}\label{oooo}
\begin{proof}[\bf{\em Proof of} Lemma \ref{5lema L2}$(a)$]
By Definition \ref{definicio wavelet}$(e)$, for any $s\in Q^0_{J'(Q)}$ we have
\begin{equation*}
\begin{split}
\|\Delta^\psi_{Q,s}h_x\|_\infty&\lesssim|\langle h_x,\psi_{s+Q}\rangle|\ell(Q)^{-n/2}\lesssim\ell(Q)^{-n}\int h_x\,d\HH^n_{L_D^x}=\ell(Q)^{-n}\int\,d\nu_x\\
&=\ell(Q)^{-n}\int\,d(p^x_\sharp(\chi_{40Q_S}\mu))=\ell(Q)^{-n}\int_{40Q_S}\,d\mu
\lesssim\frac{\ell(Q_S)^n}{\ell(Q)^n}.
\end{split}
\end{equation*}
Hence,
$\|\Delta^\psi_{Q,s}h_x\|_2\leq\|\Delta^\psi_{Q,s}h_x\|_\infty\LL^n({\supp\psi_{s+Q}})^{1/2}\lesssim\ell(Q_S)^n\ell(Q)^{-n/2}$ for all $s\in Q^0_{J'(Q)}$, and  Lemma \ref{5lema L2}$(a)$ follows by taking the average over $s\in Q^0_{J'(Q)}$.
\end{proof}

\begin{proof}[\bf{\em Proof of} Lemma \ref{5lema L2}$(b)$]
Since $D\subset  B(z_Q,C_a\ell(Q))$, $D\in S$, and $\ell(D)\lesssim\ell(Q)\leq\ell(Q_S)$, by taking $C_{cor}$ big enough (see property $(f)$ in Subsection \ref{5ss corona decomposition}), we can assume that $\mu$ is well approximated by $\Gamma_S$ in a neighborhood of $Q$.  We are going to show that, for each $s\in Q^0_{J'(Q)}$,
\begin{equation}\label{kikiki}
\|\Delta^\psi_{Q,s}h_x\|_2\lesssim\bigg(\sum_{R\in\DD:\,D\subset R\subset  B(z_Q,C_1\ell(Q))}\alpha_{\mu}(C_1R)+\frac{\dist(x,L_D)}{\ell(D)}\bigg)\ell(Q)^{n/2},
\end{equation}
and Lemma \ref{5lema L2}$(b)$ will follow by taking the average over $s\in Q^0_{J'(Q)}$. 

Fix $Q\in J_2$, so $D\subset B(z_Q,C_a\ell(Q))$ with $\ell(Q)\leq\ell(Q_S)$, and $s\in Q^0_{J'(Q)}$. Take $Q'\in\DD$ such that $\ell(Q)=\ell(Q')$ and $Q\subset B(z_{Q'},3\ell(Q))$. Recall that $\supp\psi_{s+Q}\subset CQ$ and $|\nabla\psi_{s+Q}|\lesssim\ell(Q)^{-n/2-1}$. Let $\phi_{s+Q'}$ be an extension of $\psi_{s+Q}$, i.e., let $\phi_{s+Q'}:\R^d\to\R$ be such that $\supp\phi_{s+Q'}\subset B_{Q'}\subset\R^d$, $|\nabla\phi_{s+Q'}|\lesssim\ell(Q')^{-n/2-1}$ and $\phi_{s+Q'}=\psi_{s+Q}$ in $L_D^x$.

Let $L_{Q'}$ be a minimizing $n$-plane for $\alpha_{\mu}(C_1Q')$, where $C_1>1$ is some big constant to be fixed below, and let $L_{Q'}^x$ be the $n$-plane parallel to $L_{Q'}$ which contains $x$. Let $\sigma_{Q'}:=c_{Q'}\HH^n_{L_{Q'}}$ be a minimizing measure for $\alpha_{\mu}(C_1Q')$ and define $\sigma^x_{Q'}:=c_{Q'}\HH^n_{L_{Q'}^x}$. Finally, set $\sigma:=c_{Q'}\HH^n_{L_D^x}$. Since $\psi_{s+Q}$ has vanishing integral in $L_D^x$, we also have $\int \phi_{s+Q'}\,d\HH^n_{L_D^x}=0$. Hence,
\begin{equation}\label{pipipi}
\begin{split}
\|\Delta_{Q,s}^\psi h_x\|_2&=\|\langle h_x,\psi_{s+Q}\rangle\psi_{s+Q}\|_2=|\langle h_x,\psi_{s+Q}\rangle|
=\Big|\int_{L_D^x}\phi_{s+Q}(y)\,d\nu_x(y)\Big|\\&=\Big|\int\phi_{s+Q'}(y)\,d(\nu_x-\sigma)(y)\Big|
\lesssim\ell(Q)^{-n/2-1}\dist_{B_{Q'}}(\nu_x,\sigma).
\end{split}
\end{equation}

We can assume that 
\begin{equation}\label{ririri}
\sum_{R\in\DD:\,D\subset R\subset  B(z_Q,C_1\ell(Q))}\alpha_{\mu}(C_1R)\leq b_*,
\end{equation}
otherwise Lemma \ref{5lema L2}$(b)$ follows easily. By assuming (\ref{ririri}) one can show that the angle between $L_D^x$ and $L_{Q'}^x$ is small. By the triangle inequality, we have
\begin{equation}\label{eq4}
\begin{split}
\dist_{B_{Q'}}(\nu_x,\sigma)
\leq\dist_{B_{Q'}}(\nu_x,p^x_\sharp\sigma_{Q'}^x)
+\dist_{B_{Q'}}(p^x_\sharp\sigma_{Q'}^x,\sigma).
\end{split}
\end{equation}

To deal with the first term on the right hand side of (\ref{eq4}), let $h$ be a Lipschitz function such that $\supp h\subset B_{Q'}$ and $\Lip(h)\leq1$. Then, using that $\supp\mu$ is well approximated in $CQ'$ by a Lipschitz graph $\Gamma_S$ with small slope, the function $h\circ p^x$ restricted to $\supp\mu\cup L_Q^x$ can be extended to a Lipschitz function supported in $B_{C_1Q'}$ (if $C_1$ is big enough) with $\Lip(h\circ p^x)$ bounded by a constant which only depends on $n$, $d$, and $\Lip(\Gamma_S)$. Therefore,
\begin{equation}\label{eq7}
\begin{split}
\Big|&\int_{B_{Q'}} h\,d(\nu_x-p^x_\sharp\sigma_{Q'}^x)\Big|=\Big|\int_{B_{C_1{Q'}}}h\circ p^x\,d(\mu-\sigma_{Q'}^x)\Big|
\lesssim\dist_{B_{C_1Q'}}(\mu,\sigma_{Q'}^x)\\
&\leq\dist_{B_{C_1Q'}}(\mu,\sigma_{Q'})+\dist_{B_{C_1Q'}}(\sigma_{Q'},\sigma_{Q'}^x)
\lesssim\alpha_{\mu}(C_1Q')\ell(Q)^{n+1}+\dist(x,L_{Q'})\ell(Q)^n.
\end{split}
\end{equation}
Since $x\in D$ and $D\subset C_1Q'$ (if $C_1>C_b$), by \cite[Remark 5.3]{To} we have
\begin{equation}\label{eq9}
\begin{split}
\dist(x,L_{Q'})\lesssim\sum_{R\in\DD:\,D\subset R\subset C_1Q'}\alpha_{\mu}(R)\ell(R)+\dist(x,L_D).
\end{split}
\end{equation}

Taking the supremum over all possible Lipschitz functions $h$ in (\ref{eq7}) and using that $\ell(D)\leq\ell(R)\lesssim\ell(Q)$ in the sum above, we get
\begin{equation}\label{eq8}
\begin{split}
\dist_{B_{Q'}}(\nu_x,p^x_\sharp\sigma_{Q'}^x)\lesssim\sum_{R\in\DD:\,D\subset R\subset C_1Q'}\alpha_{\mu}(C_1R)\ell(Q)^{n+1}+\frac{\dist(x,L_D)}{\ell(D)}\,\ell(Q)^{n+1}.
\end{split}
\end{equation}

To estimate the second term on the right hand side of (\ref{eq4}), notice that $p^x_\sharp\sigma=\sigma$ because $p^x|_{L_D^x}=\operatorname{Id}$. Hence, as in (\ref{eq7}),
\begin{equation*}
\begin{split}
\dist_{B_{Q'}}(p^x_\sharp\sigma_{Q'}^x,\sigma)&=\dist_{B_{Q'}}(p^x_\sharp\sigma_{Q'}^x,p^x_\sharp\sigma)
\lesssim\dist_{B_{C_1{Q'}}}(\sigma_{Q'}^x,\sigma)\\
&\leq\dist_{B_{C_1Q'}}(\sigma_{Q'}^x,\sigma_{Q'})+\dist_{B_{C_1Q'}}(\sigma_{Q'},\sigma)\\
&\lesssim\dist_{B_{C_1Q'}}(\HH^n_{L_{Q'}^x},\HH^n_{L_{Q'}})+\dist_{B_{C_1Q'}}(\HH^n_{L_{Q'}},\HH^n_{L_D})
+\dist_{B_{C_1Q'}}(\HH^n_{L_D},\HH^n_{L_D^x})\\
&\lesssim\dist(x,L_{Q'})\ell(Q)^{n}+\dist_{B_{C_1Q'}}(\HH^n_{L_{Q'}},\HH^n_{L_D})+\dist(x,L_D)\ell(Q)^{n}.
\end{split}
\end{equation*}

The term $\dist_{B_{C_1Q'}}(\HH^n_{L_{Q'}},\HH^n_{L_D})$ can be estimated using the intermediate $\mu$-cubes between $D$ and $C_1Q'$ (similarly to (\ref{eq8})), and we obtain
$$\dist_{B_{C_1Q}}(\HH^n_{L_Q},\HH^n_{L_D})\lesssim\sum_{R\in\DD:\,D\subset R\subset C_1Q}\alpha_{\mu}(C_1R)\ell(Q)^{n+1}.$$ Thus, by (\ref{eq9}) and since $\ell(D)\lesssim\ell(Q)$,
\begin{equation*}
\begin{split}
\dist_{B_{Q'}}(p^x_\sharp\sigma_{Q'}^x,\sigma)\lesssim
\sum_{R\in\DD:\,D\subset R\subset C_1Q'}\alpha_{\mu}(C_1R)\ell(Q)^{n+1}+\frac{\dist(x,L_D)}{\ell(D)}\,\ell(Q)^{n+1}.
\end{split}
\end{equation*}
Then, (\ref{kikiki}) follows by plugging this last inequality and (\ref{eq8}) in (\ref{eq4}) combined with (\ref{pipipi}), and recalling that $\ell(Q)\approx\ell(Q')$. Thus we are done with Lemma \ref{5lema L2}$(b)$.

\end{proof}

\begin{proof}[\bf{\em Proof of} Lemma \ref{5lema L2}$(c)$]
Given $Q\in J_1\cap\Tre$, using (\ref{reduced lema1}) we have
\begin{equation*}
\sum_{R'\in\DD:\,R\subset R',\,\ell(R')\leq\ell(D)}\alpha_{\mu}(10R')<b_*
\end{equation*}
for all $R\in\DD$ with $\ell(R)=\ell(Q)$ and such that  $R\cap(p^x)^{-1}(\supp\psi_{s+Q})\neq\emptyset$ for all $s\in Q^0_{J(Q)}$. By assuming $b_*$ small enough, we are going to show that for some $Q_0(x,Q)\in\DD$ as in the statement $(c)$ and all $s\in Q^0_{J(Q)}$ we have
\begin{equation}\label{lilili}
\|\Delta^\psi_{Q,s}h_x\|_2\lesssim\bigg(\sum_{R\in\DD:\,Q_0\subset R\subset C_2D}\alpha_{\mu}(C_2R)+\frac{\dist(x,L_D)}{\ell(D)}\bigg)\ell(Q)^{n/2}
\end{equation}

As before, Lemma \ref{5lema L2}$(c)$ will follow by averaging over $s\in Q^0_{J(Q)}$, and noting that $\|m_{s\in Q^0_{J(Q)}}(\Delta^\psi_{Q,s}h_x)\|_2\leq m_{s\in Q^0_{J(Q)}}\|\Delta^\psi_{Q,s}h_x\|_2$ by Minkowski's integral inequality.

Take $Q\in J_1\cap\Tre$. Let $C_2$ be some big constant which will be fixed later on, and let $Q_0\in\DD$ be a minimal $\mu$-cube such that $C_2Q_0$ contains $\supp\mu\cap(p^x)^{-1}(\supp\psi_{s+Q}\cap L_D^x)$ for all $s\in J(Q)$. We can assume that $Q_0\subset C_2D$ if $C_2$ is big enough and, by (\ref{reduced lema1}), we may also suppose that $\sum_{R\in\DD:\,Q_0\subset R\subset C_2D}\alpha_{\mu}(C_2R)$ is small enough. Hence, if $L_{Q_0}$ is a minimizing $n$-plane for $\beta_{\infty,\mu}(C_2Q_0)$, the angle between $L_{Q_0}$ and $L_D^x$ is also small enough, since it is bounded by $\sum_{R\in\DD:\,Q_0\subset R\subset C_2D}\alpha_{\mu}(C_2R)$ (see \cite[Lemma 5.2]{To} for a related argument). It is not hard to show that then
\begin{equation}\label{L2 claim}
\diam(\Gamma\cap(p^x)^{-1}(Q\cap L_D^x))\lesssim\ell(Q).
\end{equation}

Let $L_{Q_0}$ and $\sigma_{Q_0}:=c_{Q_0}\HH^n_{L_{Q_0}}$ be a minimizing $n$-plane and measure for $\alpha_{\mu}(C_2Q_0)$, respectively. Fix $z_{Q_0}\in L_{Q_0}\cap B_{C_2Q_0}$ and let $L_r$ be an $n$-plane parallel to $L_{D}^x$ which contains $z_{Q_0}$. Finally, define the measures $\sigma_{r}:=c_{Q_0}\HH^n_{L_{r}}$ and $\sigma':=c_{Q_0}\HH^n_{L_D^x}$.

Since $\sigma'$ is a multiple of $\HH^n_{L_D^x}$, similarly to (\ref{pipipi}) and using the triangle inequality,
\begin{equation}\label{eq6}
\begin{split}
\|\Delta^\psi_{Q,s}h_x\|_2\ell(Q)^{n/2+1}&\lesssim\dist_{B_Q}(\nu_x,\sigma')\\
&\leq\dist_{B_Q}(\nu_x,p^x_\sharp\sigma_{Q_0})
+\dist_{B_Q}(p^x_\sharp\sigma_{Q_0},p^x_\sharp\sigma_r)
+\dist_{B_Q}(p^x_\sharp\sigma_r,\sigma'),
\end{split}
\end{equation}
where we have set $B_Q:=B(z_Q,3\ell(Q))\subset\R^d$ (for these computations, we may also assume that $\ell(Q)$ is small enough in comparison with $\ell(D)$).

Arguing as in (\ref{eq7}), if $C_2$ is big enough, we have
\begin{equation}\label{eq10}
\begin{split}
\dist_{B_Q}(\nu_x,p^x_\sharp\sigma_{Q_0})=\dist_{B_Q}(p^x_\sharp\mu,p^x_\sharp\sigma_{Q_0})\lesssim\alpha_\mu(C_2Q_0)\ell(Q)^{n+1},
\end{split}
\end{equation}
and
\begin{equation*}
\begin{split}
\dist_{B_Q}(p^x_\sharp\sigma_{Q_0},p^x_\sharp\sigma_r)\lesssim\dist_{B_{C_2Q_0}}(\sigma_{Q_0},\sigma_r)\lesssim\dist_\HH(L_{Q_0}\cap B_{C_2Q_0},L_r\cap B_{C_2Q_0})\ell(Q)^n.
\end{split}
\end{equation*}
Let $\gamma$ be the angle between $L_r$ and $L_{Q_0}$ (which is the same as the one between $L_D$ and $L_{Q_0}$).
Since $z_{Q_0}\in L_{Q_0}\cap L_r\cap B_{C_2Q_0}$, we have $\dist_\HH(L_{Q_0}\cap B_{C_2Q_0},L_r\cap B_{C_2Q_0})\lesssim\sin(\gamma)\ell(Q)$, and it is not difficult to show that
$\sin(\gamma)\lesssim\sum_{R\in\DD:\,Q_0\subset R\subset C_2D}\alpha_{\mu}(C_2R)$. Thus,
\begin{equation}\label{eq11}
\begin{split}
\dist_{B_Q}(p^x_\sharp\sigma_{Q_0},p^x_\sharp\sigma_r)\lesssim
\sum_{R\in\DD:\,Q_0\subset R\subset C_2D}\alpha_{\mu}(C_2R)\ell(Q)^{n+1}.
\end{split}
\end{equation}

Let us estimate the last term on the right hand side of (\ref{eq6}). Since $c_{Q_0}\lesssim1$, we have $\dist_{B_Q}(p^x_\sharp\sigma_r,\sigma')\lesssim\dist_{B_Q}(p^x_\sharp\HH^n_{L_r},\HH^n_{L_D^x})$. Let $h$ be a 1-Lipschitz function supported in $B_Q$ and such that
Set $d:=\dist(z_{Q_0},L_D^x)$. Since $Q\in J_1\subset J$ and $\ell(Q)\leq C\ell(D)$, if $C$ is small enough then $\dist(x,B_Q)\gtrsim\ell(D)$.
Without loss of generality, we may assume that $x=0$ and that $L_D^x=\R^n\times\{0\}^{d-n}$, so $L_r=z_{Q_0}+\R^n\times\{0\}^{d-n}$. 
Thus, if we set $z_{Q_0}':=(z_{Q_0}^{n+1},\ldots,z_{Q_0}^d)$, we have that $d=|z_{Q_0}'|$ and $p^x$ restricted to $L_r\cap B_Q$ can be written in the following manner:
$p^x:y=(y^1,\ldots,y^n,z_{Q_0}')\mapsto(F(y^1,\ldots,y^n),0)$, where $F:\R^n\setminus\{0\}^n\to\R^n$ is defined by
$$F(y)=y\frac{\sqrt{|y|^2+d^2}}{|y|}=y\sqrt{1+\frac{d^2}{|y|^2}}.$$
Therefore, $\int h\,d(p^x_\sharp\HH^n_{L_r})=\int h\circ p^x\,d\HH^n_{L_r}=\int_{\R^n}(h\circ p^x)(y,z_{Q_0}')\,dy=\int_{\R^n}h(F(y),0)\,dy$, and we also have
$\int h\,d\HH^n_{L_D^x}=\int_{\R^n}h((y,0))\,dy=\int_{\R^n}h(F(y),0)J(F)(y)\,dy$ by a change of variables, where $J(F)$ denotes the Jacobian of $F$. Hence
\begin{equation}\label{lmlml}
\left|\int h\,d(p^x_\sharp\HH^n_{L_r}-\HH^n_{L_D^x})\right|\lesssim\int_{\R^n}|h(F(y),0)||1-J(F)(y)|\,dy.
\end{equation}
Notice that, because of the assumptions on $\supp h(F(\cdot),0)$ and since $z_{Q_0}\in B_{C_2Q_0}$ and $Q_0\subset C_2D$, we have $d\lesssim|y|$ for all $y\in\supp h(F(\cdot),0)$.
If $F_i$ denotes the $i$'th coordinate of $F$, it is straightforward to check that
$\partial_{y^j}F_i(y)=-d^2y^iy^j|y|^{-3}(|y|^2+d^2)^{-1/2}$ if $i\neq j$ and
$\partial_{y^i}F_i(y)=(1+d^2/|y|^2)^{1/2}-d^2(y^i)^2|y|^{-3}(|y|^2+d^2)^{-1/2}$. Thus, we easily obtain
\begin{equation}\label{lolo}
|1-J(F)(y)|\lesssim d/|y|\lesssim d/\ell(D)
\end{equation}
for all $y\in\supp h(F(\cdot),0)$.
Since $\diam(\supp h(F(\cdot),0))\lesssim\ell(Q)$ and $h((F(\cdot),0))$ is Lipschitz, using (\ref{lolo}) and taking the supremum in (\ref{lmlml}) over all such functions $h$, we have $\dist_{B_Q}(p^x_\sharp\HH^n_{L_r},\HH^n_{L_D^x})\lesssim \ell(Q)^{n+1}d/\ell(D).$
Finally, by \cite[Remark 5.3]{To} and since $z_{Q_0}\in L_{Q_0}$,
$$d\lesssim\dist(z_{Q_0},L_D)+\dist(L_D,L_D^x)
\lesssim\sum_{R\in\DD:\,Q_0\subset R\subset C_2D}\alpha_\mu(C_2R)\ell(R)+\dist(x,L_D),$$
and thus
\begin{equation}\label{eq12}
\begin{split}
\dist_{B_Q}(p^x_\sharp\HH^n_{L_r},\HH^n_{L_D^x})\lesssim
\sum_{R\in\DD:\,Q_0\subset R\subset C_2D}\alpha_\mu(C_2R)\ell(Q)^{n+1}+\frac{\dist(x,L_D)}{\ell(D)}\,\ell(Q)^{n+1}.
\end{split}
\end{equation}

Finally, (\ref{lilili}) follows by applying (\ref{eq10}), (\ref{eq11}), and (\ref{eq12}) to (\ref{eq6}), which yields Lemma \ref{5lema L2}$(c)$.

\end{proof}

\begin{proof}[\bf{\em Proof of} Lemma \ref{5lema L2}$(d)$]
This is the key point where taking averages of dyadic lattices with respect to the parameter $s$ is necessary. Given $Q\in J_1\cap\Stop$,  we have to show that $\|m_{s\in Q^0_{J(Q)}}(\wit\Delta^\psi_{Q,s}h_x)\|_1\lesssim\ell(Q)^n.$ Unlike in $(a),\ldots,(c)$, the estimate in $(d)$ does not hold for a particular choice of $s$ in general but, as we will see, it holds in average. Recall that, for a fixed $s\in Q^0_{J(Q)}$, 
\begin{equation*}
\begin{split}
\wit\Delta^\psi_{Q,s}h_x&=\sum_{R\in\PP:\,R\subset Q}\Delta^\psi_{R,s}h_x\\
&=\sum_{\begin{subarray}{c}R\in\PP:\,\supp\psi_{R}\cap Q\neq\emptyset\\ \ell(R)\leq\ell(Q)\end{subarray}}\chi_{s+Q}\,\Delta^\psi_{R,s}h_x
-\sum_{\begin{subarray}{c}R\in\PP:\,\supp\psi_{R}\cap Q\neq\emptyset\\ \ell(R)\leq\ell(Q),\,R\not\subset Q\end{subarray}}\chi_{s+Q}\,\Delta^\psi_{R,s}h_x\\
&\quad+\sum_{\begin{subarray}{c}R\in\PP:\\ R\subset Q\end{subarray}}\chi_{(s+Q)^c}\,\Delta^\psi_{R,s}h_x
=:I_s+II_s+III_s.
\end{split}
\end{equation*}
We are going to estimate $I_s$, $II_s$, and $III_s$ separately. For the case of $I_s$, we have
\begin{equation*}
\begin{split}
\chi_{s+Q}\, h_x=\chi_{s+Q}\sum_{R\in\DD_x^{n,0}:\, \ell(R)>\ell(Q)}\Delta^\psi_{R,s}h_x+
\chi_{s+Q}\sum_{R\in\DD_x^{n,0}:\, \ell(R)\leq\ell(Q)}\Delta^\psi_{R,s}h_x=\chi_{s+Q}\,I'_s+I_s,\\
\end{split}
\end{equation*}
where we have set $I'_s:=\sum_{R\in\DD_x^{n,0}:\, \ell(R)>\ell(Q)}\Delta^\psi_{R,s}h_x$. On one hand, since $Q\in J_1\cap\Stop$, (\ref{reduced lema1}) holds. Thus, using that $\sum_{R\in\DD:\,P(R_Q)\subset R,\,\ell(R)\leq\ell(D)}\alpha_{\mu}(10R)<b_*$, one can show that 
\begin{equation}\label{njniop}
\|\chi_{s+Q}\,h_x\|_1\lesssim\ell(Q)^n
\end{equation}
(see above (\ref{L2 claim}) for a related argument). On the other hand, since $\|\chi_{s+Q}\,h_x\|_1\lesssim\ell(Q)^n$, it is known that then $\|\chi_{s+Q} I'_s\|_1\lesssim\ell(Q)^n$  (see \cite[Part I]{David-LNM}, in particular pay attention to the last sum in equation (46) of Part I). Combining these estimates, we conclude that $\|I_s\|_1\lesssim\ell(Q)^n$.

Let us now deal with $II_s$. First of all, split $II_s$ into different scales, that is 
$$\sum_{\begin{subarray}{c}R\in\PP:\,\supp\psi_{R}\cap Q\neq\emptyset\\ \ell(R)\leq\ell(Q),\,R\not\subset Q\end{subarray}}\chi_{s+Q}\,\Delta^\psi_{R,s}h_x=
\sum_{k\geq J(Q)}\,\sum_{\begin{subarray}{c}R\in\PP:\,\supp\psi_{R}\cap Q\neq\emptyset\\ \ell(R)=2^{-k},\,R\not\subset Q\end{subarray}}\chi_{s+Q}\,\Delta^\psi_{R,s}h_x.$$
Observe that if $k\geq J(Q)$, $\supp\psi_{R}\cap Q\neq\emptyset$, $\ell(R)=2^{-k}$, and $R\not\subset Q$, then $s+R\subset U_{C2^{-k}}(s+\partial Q)$, where $C>1$ is some fixed constant  and  $U_{C2^{-k}}(s+\partial Q):=\{z\in L_D^x:\, \dist(z,s+\partial Q)<C2^{-k}\}$. Hence, using Definition \ref{definicio wavelet}$(e)$ and the definition of $h_x$, we get
\begin{equation*}
\begin{split}
\|II_s\|_1\leq\sum_{k\geq J(Q)}\,\sum_{\begin{subarray}{c}R\in\PP:\,\supp\psi_{R}\cap Q\neq\emptyset\\ \ell(R)=2^{-k},\,R\not\subset Q\end{subarray}}\|\Delta^\psi_{R,s}h_x\|_1
\lesssim\sum_{k\geq J(Q)}\nu_x(U_{C2^{-k}}(s+\partial Q)).
\end{split}
\end{equation*}

The case of $III_s$ can be dealt with very similar techniques, and then one obtains the same estimate. Therefore, 
\begin{equation}\label{reduced aqpl}
\begin{split}
\|m_{s\in Q^0_{J(Q)}}(\wit\Delta^\psi_{Q,s}h_x)\|_1&=\|m_{s\in Q^0_{J(Q)}}(I_s+II_s+III_s)\|_1
\leq m_{s\in Q^0_{J(Q)}}\|I_s+II_s+III_s\|_1\\
&\lesssim\ell(Q)^n+m_{s\in Q^0_{J(Q)}}\bigg(\sum_{k\geq J(Q)}\nu_x(U_{C2^{-k}}(s+\partial Q))\bigg).
\end{split}
\end{equation}
Using Fubini's theorem, it is not difficult to show that $$m_{s\in Q^0_{J(Q)}}\nu_x(U_{C2^{-k}}(s+\partial Q))\big)\lesssim2^{-k}\ell(Q)^{-1}\nu_x(CQ)$$ for all for $k\geq J(Q)$ (see \cite[Lemma 7.5]{Tolsa3} for example, for a related argument). Since $Q\in\Stop$, then (\ref{reduced lema1}) holds and then, as in (\ref{njniop}), we have $\nu_x(CQ)\lesssim\ell(Q)^n$, thus $$m_{s\in Q^0_{J(Q)}}\bigg(\sum_{k\geq J(Q)}\nu_x(U_{C2^{-k}}(s+\partial Q))\bigg)\lesssim\ell(Q)^n.$$
If we combine this last estimate with (\ref{reduced aqpl}), we are done.
\end{proof}

\subsubsection{{\bf Final estimates}}
From Lemmas \ref{5 var eq8}, \ref{5 var eq11}, and \ref{5 var eq8a}, we obtain the following:
\begin{equation*}
\begin{split}
\sum_{S\in\Trees}\,\sum_{D\in S}&\int_D\sum_{m\in\SSS_D(x)}\bigg|\sum_{R\in V(D)}\sum_{Q\in \Tre(R)}(K\chi_{\epsilon_{m+1}}^{\epsilon_m}*(\Delta_Q f)\mu)(x)\bigg|^2\,d\mu(x)\\
+\sum_{S\in\Trees}&\,\sum_{D\in S}\int_D\sum_{m\in\SSS_D(x)}\bigg|\sum_{R\in V(D)}\sum_{Q\in \Stop(R)}(K\chi_{\epsilon_{m+1}}^{\epsilon_m}*(\wit\Delta_Q f)\mu)(x)\bigg|^2\,d\mu(x)\\
&+\sum_{S\in\Trees}\sum_{D\in S}\int_D\sum_{m\in\SSS_D(x)}\bigg|\sum_{R\in V(D)}
(K\chi_{\epsilon_{m+1}}^{\epsilon_m}*(m_R^\mu f)\chi_R\mu)(x)\bigg|^2\,d\mu(x)\lesssim\|f\|_{L^2(\mu)}^2.
\end{split}
\end{equation*}
Combining this estimate with (\ref{eqmain}), we deduce
\begin{equation*}
\begin{split}
\sum_{S\in\Trees}&\,\sum_{D\in S}\int_D\sum_{m\in\SSS_D(x)}\big|(K\chi_{\epsilon_{m+1}}^{\epsilon_m}*(f\mu))(x)\big|^2\,d\mu(x)\lesssim\|f\|_{L^2(\mu)}^2.
\end{split}
\end{equation*}
Finally, using (\ref{5 var eq1}) and (\ref{5 var eq2}), we conclude that 
\begin{equation*}
\begin{split}
\|(\VV^\SSS_\rho\circ\TT^\mu)f\|_{L^2(\mu)}^2\lesssim\sum_{D\in\DD}\int_D\sum_{m\in\SSS_D(x)}\big|(K\chi_{\epsilon_{m+1}}^{\epsilon_m}*(f\mu))(x)\big|^2\,d\mu(x)\lesssim\|f\|_{L^2(\mu)}^2.
\end{split}
\end{equation*}
This finishes the proof of
Theorem \ref{5teo var no suau acotada L2}.

\section{If $\VV_\rho\circ\RR^\mu:\,L^2(\mu)\to L^2(\mu)$ is a bounded operator,\\then $\mu$ is  a uniformly $n$-rectifiable measure}\label{4sec acotacio implica rectif}
Let $C_\mu>0$ be the AD regularity constant of an AD regular measure $\mu$, that is $C_\mu^{-1}r^n\leq\mu(B(x,r))\leq C_\mu r^{n}$ for all $x\in\supp\mu$ and $0\leq r<\diam(\supp\mu)$. For simplicity of notation, we may assume that $\diam(\supp\mu)=\infty$ (the general case follows with minor modifications in our arguments). As before, we denote by $\DD$ the dyadic lattice of $\mu$-cubes introduced in Subsection \ref{dyadic lattice}.

In this section, we set $K(x)={x}{|x|^{-n-1}}$ for $x\neq0$. Recall that, given $\epsilon>0$, a Borel measure $\mu$, and $f\in L^1(\mu)$, we have set $\RR^\mu f:=\{R_\epsilon^\mu f\}_{\epsilon>0}$, where
\begin{equation*}
R_\epsilon^\mu f(x) = \int_{|x-y|>\epsilon} K(x-y)f(y)\,d\mu(y).
\end{equation*}

In order to prove the main theorem of this section, namely Theorem \ref{4rectif teorema}, we need first to introduce some notation and state some preliminary results.

\begin{defi}[Special truncation of the Riesz transform]\label{c djlcnA}
For $\epsilon>0$, let $\varphi_\epsilon$ be as in Definition \ref{4defi varphi}. 
Given $m\in\Z$ and a Borel measure $\mu$ in $\R^d$, we set
\begin{equation*}
S_m\mu(x) := \int \big(\varphi_{2^{-m-1}}(x-y)-\varphi_{2^{-m}}(x-y)\big)K(x-y)\,d\mu(y).
\end{equation*}
\end{defi}
\begin{lema}[Lemma 5.8 of \cite{DS1}] \label{4lemli}
Given $Q\in\DD$, there exist $n+1$ points $x_0,\ldots,x_n$ in $Q$ (and thus in $\supp\mu$) such that $\dist(x_j,L_{j-1})\geq C\ell(Q)$, where $L_k$ denotes the $k$-plane passing through $x_0,\ldots,x_k$, and where $C$ depends only on $n$ and $C_\mu$.
\end{lema}

\begin{lema}[Lemma 7.4 and Remark 7.5 of \cite{To}]\label{4lemcas1}
Let $Q\in\DD$ and $x_0,\ldots,x_n\in Q$ be like in Lemma \ref{4lemli}.
Denote $r=\diam(Q)$, and let $m,p\in\Z$ be such that $t\geq s>4r$ for $t=2^{-p}$ and $s=2^{-m}$.
Suppose that $A(x_0,2^{-m-1/2},2^{-m+1/2})\cap\supp\mu\neq\emptyset$.
Then any point $x_{n+1}\in 3Q$ satisfies
\begin{equation}\label{4eqrem}
\dist(x_{n+1},L_0)\lesssim s\sum_{j=1}^{n+1}\sum_{k=p}^m|S_k\mu(x_j)-S_k\mu(x_0)|
+\frac{r^2}{s}+\frac{rs}{t},
\end{equation}
where $L_0$ is the $n$-plane passing through $x_0,\ldots,x_n$.
\end{lema}

The following proposition is a direct consequence of the techniques used in the last section of \cite{To}. We give the proof for completeness.
\begin{propo}\label{4rectif propo}
Given $\epsilon_0>0$, there exist $\delta_0>0$ and $m_0,k_0\in\N$ depending on
$\epsilon_0$, $n$, and $C_\mu$ such that, for all $i\in\Z$ and all $Q\in\DD_i$ with
$\beta_{1,\mu}(Q)>\epsilon_0$, there exist $k\in\Z$ with $|k|\leq k_0$ and $P\in\DD_{i+k+m_0}$ such that $P\subset 4Q$ and
$|S_{i+k}\mu(x)|\geq\delta_0\text{ for all }x\in P$.
\end{propo}

\begin{proof}[{\bf {\em Proof.}}]
Fix $\epsilon_0>0$. Let $Q\in\DD_i$ such that $\beta_{1,\mu}(Q)>\epsilon_0$.
Take points $x_0,\ldots,x_n$ in $Q$ as in Lemma \ref{4lemli}, denote $r=\diam Q$, and let $m\in\Z$ to be fixed below such that $4r<2^{-m}=:s$ and $A(x_0,2^{-m-1/2},2^{-m+1/2})\cap\supp\mu\neq\emptyset$ (we assume $\diam(\supp\mu)=\infty$). 
By Lemma \ref{4lemcas1}, for $t:=2^{-p}\geq s$ to be fixed below and all $x_{n+1}\in3Q$, 
\begin{equation*}
\begin{split}
\dist(x_{n+1},L_0)&\lesssim s\sum_{j=1}^{n+1}\sum_{k=p}^m|S_k\mu(x_j)-S_k\mu(x_0)|
+\frac{r^2}{s}+\frac{rs}{t}\\
&\lesssim s\sum_{k=p}^m\sum_{j=0}^{n+1}|S_k\mu(x_j)|
+\frac{r^2}{s}+\frac{rs}{t}.
\end{split}
\end{equation*}
Then, by integrating on $x_{n+1}\in3Q$, for some constant $C_1>0$ depending only on $n$ and $C_\mu$
\begin{equation*}
\begin{split}
\epsilon_0&<\beta_{1,\mu}(Q)\leq\frac{1}{\ell(Q)^n}\int_{3Q}\frac{\dist(x_{n+1},L_0)}{\ell(Q)}\,d\mu(x_{n+1})\\
&\leq C_1\bigg(\frac{s}{r}\sum_{k=p}^m\bigg(\frac{1}{\ell(Q)^n}\int_{3Q}|S_k\mu(x_{n+1})|\,d\mu(x_{n+1})+\sum_{j=0}^{n}|S_k\mu(x_j)|\bigg)+\frac{r}{s}+\frac{s}{t}\bigg).
\end{split}
\end{equation*}
Thus,
\begin{equation*}
\begin{split}
\frac{r}{s}\bigg(\frac{\epsilon_0}{C_1}-\frac{r}{s}-\frac{s}{t}\bigg)
\leq\sum_{k=p}^m\bigg(\int_{3Q}\frac{|S_k\mu(x_{n+1})|}{\ell(Q)^n}\,d\mu(x_{n+1})+\sum_{j=0}^{n}|S_k\mu(x_j)|\bigg).
\end{split}
\end{equation*}
We can easily choose $s$ and $t$ big enough (depending on $r$, $\epsilon_0$, and $C_1$) such that, for some constant $\epsilon_1>0$ depending only on $\epsilon_0$, $n$ and $C_\mu$,
\begin{equation}\label{4rectif eq8}
\begin{split}
0<\epsilon_1\leq\sum_{k=p}^m\bigg(\int_{3Q}\frac{|S_k\mu(x_{n+1})|}{\ell(Q)^n}\,d\mu(x_{n+1})+\sum_{j=0}^{n}|S_k\mu(x_j)|\bigg).
\end{split}
\end{equation}
Notice that, since $t=2^{-p}$ and $s=2^{-m}$ where chosen depending on $r\approx2^{-i}$, the sum on the right hand side of (\ref{4rectif eq8}) has a finite number of terms which only depends on $\epsilon_0$, $n$ and $C_\mu$. Therefore, there exists $k_0\in\N$ and $C_2>0$ depending only on $\epsilon_0$, $n$ and $C_\mu$ such that, for some negative integer $k$ with $|k|\leq k_0$ and some $j=0,\ldots,n$, 
\begin{equation*}
\begin{split}
\epsilon_1\leq C_2\bigg(\frac{1}{\ell(Q)^n}\int_{3Q}|S_{i+k}\mu|\,d\mu+|S_{i+k}\mu(x_{j})|\bigg),
\end{split}
\end{equation*}
which implies that there exists $C_3$ (depending on $C_2$) and $z\in 3Q$ such that $\epsilon_1\leq C_3|S_{i+k}\mu(z)|$.

Given $x\in\supp\mu$, if $|x-z|\leq2^{-i-k}$, then 
\begin{equation*}
\begin{split}
|S_{i+k}\mu(x)-S_{i+k}\mu(z)|&\leq
\int_{|y-z|\lesssim2^{-i-k}}\|\nabla(\varphi_{i+k}K)\|_\infty|x-z|\,d\mu(y)\\
&\lesssim2^{(i+k)(n+1)}|x-z|\int_{|y-z|
\lesssim2^{-i-k}}\,d\mu(y)\lesssim2^{i+k}|x-z|.
\end{split}
\end{equation*}
Hence if $|x-z|\leq C_42^{-i-k}$ with $C_4>0$ small enough, we have $C_3|S_{i+k}\mu(x)-S_{i+k}\mu(z)|\leq\epsilon_1/2$, so $\epsilon_1/2\leq C_3|S_{i+k}\mu(x)|$. 
Therefore, there exist $m_0\in\N$ depending on $C_4$ (and thus on $\epsilon_0$, $n$, and $C_\mu$) and $P\in\DD_{i+k+m_0}$ such that $\epsilon_1/2\leq C_3|S_{i+k}\mu(x)|$ for all $x\in P$. We can also assume that $P\subset4Q$ by taking $C_4$ small enough, and since $|k|\leq k_0$ we have $\ell(P)\approx\ell(Q)$. The proposition follows by setting $\delta_0:=\epsilon_1/(2C_3)>0$.
\end{proof}

\begin{defi}
Given $\epsilon_0>0$, let $\delta_0,m_0>0$ be as in Proposition \ref{4rectif propo}. Set
\begin{equation*}
\begin{split}
\BB&:=\{Q\in\DD\,:\,\beta_{1,\mu}(Q)>\epsilon_0\},\quad
\wit\BB:=\bigcup_{k\in\Z}\{Q\in\DD_{k+m_0}\,:\,|S_k\mu(x)|\geq\delta_0\text{ for all }x\in Q\}.
\end{split}
\end{equation*}
Given $P,R\in\DD$ with $P\subset R$, we set $F_P^R=\sum_{Q\in\wit\BB:\,P\subset Q\subset R}\chi_Q$ and
$F^R=\sum_{Q\in\wit\BB:\,Q\subset R}\chi_Q$.
\end{defi}

\begin{lema}\label{4rectif lema}
Let $\rho>0$. Assume that there exists $C_0>0$ such that, for all $R\in\DD$,
\begin{equation}\label{4rectif eq1}
\int_R \big(F^R\big)^{2/\rho}\,d\mu\leq C_0\mu(R).
\end{equation}
Then, there exists $C>0$ such that
$\sum_{Q\in\wit\BB:\,Q\subset R}\mu(Q)\leq C\mu(R)$ for all $R\in\DD$.
\end{lema}

\begin{proof}[{\bf {\em Proof.}}]
Let $M>1$ big enough (it will be fixed below). For $R\in\DD$, set
\begin{equation*}
\begin{split}
\Tree(R)&:=\big\{Q\in\wit\BB\,:\,Q\subset R,\,\chi_Q F_Q^R\leq M\chi_Q\big\},\\
\Top_0(R)&:=\big\{P\in\wit\BB\,:\,P\subset R,\,\chi_P F_P^R>M\chi_P,\text{ and }\chi_Q F_{Q}^R\leq M\chi_{Q}\\
&\qquad\qquad\qquad\qquad\qquad\qquad\qquad\qquad
\text{ for all }\,Q\in\wit\BB\text{ such that }P\subsetneq Q\subset R\big\}.
\end{split}
\end{equation*}

For $m\geq1$, set $\Top_m(R):=\bigcup_{P\in\Top_{m-1}(R)}\Top_0(P)$, and $\Top(R):=\bigcup_{m\geq0}\Top_m(P)$.
Notice that if $R\in\wit\BB$ then $R\in\Tree(R)$, because $M>1$. Notice also that
\begin{equation}\label{4rectif eq2}
\{Q\in\wit\BB\,:\,Q\subset R\}=\Tree(R)\cup\Big(\textstyle{\bigcup_{P\in\Top(R)}}\Tree(P)\Big),
\end{equation}
and the union is disjoint.

Fix $R\in\DD$. Then, by (\ref{4rectif eq2}),
\begin{equation}\label{4rectif eq3}
\begin{split}
\sum_{Q\in\wit\BB:\,Q\subset R}\mu(Q)&=\sum_{Q\in\Tree(R)}\mu(Q)+\sum_{P\in\Top(R)}\,\sum_{Q\in\Tree(P)}\mu(Q)\\
&=\int_R\sum_{Q\in\Tree(R)}\chi_Q\,d\mu+\int_R\sum_{P\in\Top(R)}\,\sum_{Q\in\Tree(P)}\chi_Q\,d\mu.
\end{split}
\end{equation}
Given $x\in R$ and $P\in\DD$ such that $P\subset R$, by the definition of $\Tree(P)$, we have
$$\sum_{Q\in\Tree(P)}\chi_Q(x)\leq M\chi_P(x).$$
Therefore, by (\ref{4rectif eq3}),
\begin{equation}\label{4rectif eq4}
\begin{split}
\sum_{Q\in\wit\BB:\,Q\subset R}\mu(Q)&\leq M\mu(R)+\int_R\sum_{P\in\Top(R)}M\chi_P\,d\mu
=M\bigg(\mu(R)+\sum_{m\geq0}\sum_{P\in\Top_m(R)}\mu(P)\bigg).
\end{split}
\end{equation}

We are going to prove that, if $M$ is big enough,
\begin{equation}\label{4rectif eq5}
\sum_{P\in\Top_m(R)}\mu(P)\leq2^{-m}\mu(R)
\end{equation}
for all $m\geq0$, and then, by (\ref{4rectif eq4}), we will finally obtain
$$\sum_{Q\in\wit\BB:\,Q\subset R}\mu(Q)\leq M\mu(R)+M\sum_{m\geq0}2^{-m}\mu(R)\leq3M\mu(R),$$
and the lemma will be proven. 

Notice that, if $P,P'\in\Top_0(R)$ are different, then $P\cap P'=\emptyset$ because of the last condition in the definition of $\Top_0(R)$. So, to verify (\ref{4rectif eq5}), it is enough to show that, for all $m\geq0$,
\begin{equation}\label{4rectif eq5bis}
\begin{split}
\sum_{P\in\Top_{m+1}(R)}\mu(P)<\frac{1}{2}\sum_{P\in\Top_{m}(R)}\mu(P).
\end{split}
\end{equation}

We have
\begin{equation}\label{plki}
\sum_{P\in\Top_{m+1}(R)}\mu(P)=\sum_{P\in\Top_{m}(R)}\,\sum_{Q\in\Top_{0}(P)}\mu(Q)
\end{equation}
and $\sum_{Q\in\Top_0(P)}\chi_Q=\chi_{U}$, where $U:=\bigcup_{Q\in\Top_0(P)}Q\subset P$. If $x\in U$, there exists $Q\in\Top_0(P)$ such that $x\in Q$, so $1=\chi_Q(x)<{M^{-2/\rho}}\big(F^P_Q(x)\big)^{2/\rho}\leq{M^{-2/\rho}}\big(F^P(x)\big)^{2/\rho}$,
and then using (\ref{4rectif eq1}) we have
\begin{equation*}
\sum_{Q\in\Top_0(P)}\mu(Q)=\int_P\sum_{Q\in\Top_0(P)}\chi_Q\,d\mu
=\int_U1\,d\mu<M^{-2/\rho}\int_P\big(F^P\big)^{2/\rho}\,d\mu\leq
\frac{C_0}{M^{2/\rho}}\,\mu(P),
\end{equation*}
which, in combination with (\ref{plki}), yields (\ref{4rectif eq5bis}) by taking $M>(2C_0)^{\rho/2}$. 
\end{proof}
 
\begin{lema}\label{4rectif lema 2}
Assume that, for some $C_1>0$, $\sum_{Q\in\wit\BB:\,Q\subset R}\mu(Q)\leq C_1\mu(R)$ for all $R\in\DD$. Then there exists $C_2>0$ such that $\sum_{Q\in\BB:\,Q\subset R}\mu(Q)\leq C_2\mu(R)$ for all $R\in\DD$.
\end{lema}
\begin{proof}[{\bf {\em Proof.}}]
Given $Q\in\BB$, by Proposition \ref{4rectif propo}, there exists $P_Q\in\DD_{k+m_0}$ for some $k\in\Z$ such that
$P_Q\subset 4Q$, $\mu(P_Q)\geq C_0\mu(Q)$, and
$|S_k\mu(x)|\geq\delta_0\text{ for all }x\in P_Q,$ where $C_0>0$ is some small constant. Thus, in particular, $P_Q\in\wit\BB$ for all $Q\in\BB$. Since $P_Q\subset 4Q$ and $\mu(P_Q)\geq C_0\mu(Q)$ for all $Q\in\BB$, given $P\in\wit\BB$ there are finitely many $\mu$-cubes $Q\in\BB$ such that $P_Q=P$, and the number of such $\mu$-cubes is bounded above by a constant depending only on $n$, $C_0$, and $C_\mu$. Hence, since $4R$ is contained in the union of a bounded number of $\mu$-cubes with side length $\ell(R)$,
\begin{equation*}
\sum_{Q\in\BB:\,Q\subset R}\mu(Q)\leq C_0^{-1}\sum_{Q\in\BB:\,Q\subset R}\mu(P_Q)
\lesssim\sum_{P\in\wit\BB:\,P\subset 4R}\mu(P)\leq C_1\mu(R)
\end{equation*}
for all $R\in\DD$, as wished. 
\end{proof}

\begin{teo}\label{4rectif teorema}
Let $\rho>0$. Given an $n$-dimensional AD regular measure $\mu$, if $\VV_\rho\circ\RR^\mu$ is a bounded operator in $L^2(\mu)$, then $\mu$ is uniformly $n$-rectifiable.
\end{teo}
\begin{proof}[{\bf {\em Proof.}}]

It is easy to see that, if $\VV_\rho\circ\RR^\mu$ is a bounded operator in $L^2(\mu)$, then $R_*^\mu$ is also bounded in $L^2(\mu)$.
By Theorem 1.2 in \cite[Part III, Chapter 1]{DS2}, in order to show that $\mu$ is uniformly $n$-rectifiable, it is enough to show that $\mu$ satisfies the Weak Geometric Lemma, i.e., that for any $\epsilon_0>0$, the set $\BB$ is a Carleson set. In other words, it suffices to show that there exists a constant $C>0$ depending on $\epsilon_0$ such that $\sum_{Q\in\BB:\,Q\subset R}\mu(Q)\leq C\mu(R)$ for all $R\in\DD$. By Lemma \ref{4rectif lema 2} and Lemma \ref{4rectif lema}, this holds if, for some $\rho>0$, there exists $C>0$ depending on $\epsilon_0$ such that, for all $R\in\DD$,
\begin{equation}\label{4rectif eq6}
\int_R \big(F^R\big)^{2/\rho}\,d\mu\leq C\mu(R).
\end{equation}

Notice that, for $m\in\Z$ and $f\in L^1(\mu)$, $S_m(f\mu)=T_{\varphi_{2^{-m-1}}}^\mu f-T_{\varphi_{2^{-m}}}^\mu f$, where $S_m$ is introduced in Definition \ref{c djlcnA} and $T_{\varphi_{\epsilon}}^\mu$ is as in Definition \ref{4defi varphi} (remember that now $K$ denotes the Riesz kernel), thus
\begin{equation}\label{4rectif eqvar}
\begin{split}
\sum_{k\in\Z}|S_k(f\mu)(x)|^{\rho}
\leq\big((\VV_{\rho}\circ\TT_\varphi^\mu)f(x)\big)^{\rho}.
\end{split}
\end{equation}
We may assume that $\rho\geq1$, since $(\VV_{\wit\rho}\circ\RR^\mu)f(x)\leq(\VV_{\rho}\circ\RR^\mu)f(x)$ for $\wit\rho\geq\rho$, and then the $L^2(\mu)$ boundedness of $\VV_{\rho}\circ\RR^\mu$ for some $\rho>0$ implies the $L^2(\mu)$ boundedness of $\VV_{\wit\rho}\circ\RR^\mu$ for all $\wit\rho\geq\rho$. Since $\varphi_\R\big(2^{2m}t^2\bigr)$ is a convex combination of the functions $\chi_{\{s\in\R\,:\, s>\epsilon\}}(t)$ for $\epsilon>0$, using that $\rho\geq1$ and Minkowski's integral inequality, it is not hard to show that the $L^2(\mu)$ boundedness of $\VV_{\rho}\circ\RR^\mu$ implies the $L^2(\mu)$ boundedness of $\VV_{\rho}\circ\TT_\varphi^\mu$ (see Subsection \ref{5sslong}, or \cite[Lemma 2.4]{CJRW-Hilbert}, for a similar argument). Therefore, for any $M>0$, we have
\begin{equation}\label{4rectif eqvar1}
\|(\VV_{\rho}\circ\TT_\varphi^\mu)\chi_{MR}\|^2_{L^2(\mu)}\leq C\mu(MR)\leq C\mu(R)\text{ for all $R\in\DD$.}
\end{equation}

Fix $\epsilon_0>0$, let $\delta_0,m_0>0$ be as in Proposition \ref{4rectif propo}, and let $R\in\DD$. Given $x\in R$ and $k\in\Z$, for any $Q\in\DD_{k+m_0}\cap\wit\BB$ such that $x\in Q\subset R$ we have $|S_k\mu(x)|\geq\delta_0$. Notice that, since $Q\in\DD_{k+m_0}$ and $Q\subset R$, there exists $M>1$ depending only on $n$ and $m_0$ such that $\delta_0\leq|S_k\mu(x)|=|S_k(\chi_{MR}\mu)(x)|$.
Therefore, using (\ref{4rectif eqvar}) and that for each $k\in\Z$ there is at most one $\mu$-cube $Q\in\DD_{k+m_0}$ such that $x\in Q\subset R$,
\begin{equation}\label{eq osc}
\begin{split}
F^R(x)&=\sum_{k\in\Z}\,\sum_{Q\in\DD_{k+m_0}\cap\wit\BB\,:\,x\in Q\subset R\,}\chi_Q(x)
\leq\sum_{k\in\Z}\,\sum_{Q\in\DD_{k+m_0}\cap\wit\BB\,:\,x\in Q\subset R\,}\delta_0^{-\rho}|S_k(\chi_{MR}\mu)(x)|^{\rho}\\
&\leq\delta_0^{-\rho}\sum_{k\in\Z}|S_k(\chi_{MR}\mu)(x)|^{\rho}
\leq\delta_0^{-\rho}\big((\VV_{\rho}\circ\TT_\varphi^\mu)\chi_{MR}(x)\big)^{\rho}
\end{split}
\end{equation}
and then, by (\ref{4rectif eqvar1}),
\begin{equation*}
\begin{split}
\int_R\big(F^R)^{2/\rho}\,d\mu&
\leq\delta_0^{-2}\int_R\big((\VV_{\rho}\circ\TT_\varphi^\mu)\chi_{MR}\big)^2\,d\mu
\leq\delta_0^{-2}\|(\VV_{\rho}\circ\TT_\varphi^\mu)\chi_{MR}\|^2_{L^2(\mu)}
\leq C\mu(R)
\end{split}
\end{equation*}
for all $R\in\DD$. This yields (\ref{4rectif eq6}), and the theorem follows.
\end{proof}

\begin{remarko}\label{4remark oscil}
{\em Let $\{r_m\}_{m\in\Z}\subset(0,\infty)$ be a fixed decreasing sequence defining $\OO$. If there exists $C>0$ such that $C^{-1}r_m\leq r_m-r_{m+1}\leq Cr_m$ for all $m\in\Z$, then the last inequality in (\ref{eq osc}) still holds if we replace $\VV_{\rho}$ by $\OO$ (by taking from the beginning $\rho=2$). Hence, Theorem \ref{4rectif teorema} still holds replacing $\VV_\rho$ by $\OO$ for this particular sequence $\{r_m\}_{m\in\Z}$. However, we do not know if it holds for any  $\{r_m\}_{m\in\Z}\subset(0,\infty)$. }
\end{remarko}

\end{document}